\documentclass[a4paper,10pt,reqno]{amsart}
\usepackage{bm,amsmath,amsfonts,amscd,amsthm,amssymb,graphicx,mathrsfs,eufrak,mathrsfs}

\usepackage[nospace,noadjust]{cite}

\usepackage[dvips,all,arc,curve,color,frame]{xy}
\usepackage{tikz}
\usetikzlibrary{arrows,decorations.pathmorphing,decorations.pathreplacing,positioning,shapes.geometric,shapes.misc,decorations.markings,decorations.fractals,calc,patterns}
\usepackage{tikz-cd}
\usepackage[colorlinks]{hyperref}
\usepackage{comment}
\usepackage{mathtools}
\usepackage{mathdots}
\usepackage{todonotes}
\usepackage{stmaryrd}

\tikzset{>=stealth',
        cvertex/.style={circle,draw=black,inner sep=1pt,outer sep=3pt},
        vertex/.style={circle,fill=black,inner sep=1pt,outer sep=3pt},
        star/.style={circle,fill=yellow,inner sep=0.75pt,outer sep=0.75pt},
        tvertex/.style={inner sep=1pt,font=\scriptsize},
        gap/.style={inner sep=0.5pt,fill=white}}

\author{Martin Kalck} 
\address{School of Mathematics,
University of Edinburgh,
JCMB,
Peter Guthrie Tait Road, 
Edinburgh,
EH9 3FD.}
\email{m.kalck@ed.ac.uk}
\urladdr{}

\author{Joseph Karmazyn} 
\address{School of Mathematics and Statistics,
University of Sheffield,
Hicks Building,
Hounsfield Road,
Sheffield,
S3 7RH.}
\email{j.h.karmazyn@sheffield.ac.uk}
\urladdr{http://www.jhkarmazyn.staff.shef.ac.uk/}
\date{\today}


\title[Noncommutative Kn\"{o}rrer equivalences]{Noncommutative Kn\"{o}rrer type equivalences via noncommutative resolutions of singularities.}

\theoremstyle{definition} 
\newtheorem{Definition}{Definition}[section]
\newtheorem{Example}[Definition]{Example}
\newtheorem{Conjecture}[Definition]{Conjecture}
\newtheorem{Remark}[Definition]{Remark}

\newtheorem*{pQuestion}{Question}

\theoremstyle{plain} 
\newtheorem{Lemma}[Definition]{Lemma}
\newtheorem{Theorem}[Definition]{Theorem}
\newtheorem{Proposition}[Definition]{Proposition}
\newtheorem{Corollary}[Definition]{Corollary}

\newtheorem*{pTheorem}{Theorem}
\newtheorem*{pLemma}{Lemma}
\newtheorem*{pProposition}{Proposition}
\newtheorem*{pCorollary}{Corollary}

\DeclareMathAlphabet{\mathbbm}{U}{bbm}{m}{n}

\newcommand{\Hom}{\textnormal{Hom}}
\newcommand{\Perf}{\textnormal{Perf}}

\newcommand{\Ext}{\textnormal{Ext}}

\newcommand{\Coh}{\textnormal{Coh}}

\newcommand{\QCoh}{\textnormal{QCoh}}
\newcommand{\op}{\textnormal{op}}
\newcommand{\rk}{\textnormal{rk}}
\newcommand{\End}{\textnormal{End}}

\newcommand{\CH}{\textnormal{H}}
\newcommand{\SL}{\textsf{SL}}
\newcommand{\GL}{\mathsf{GL}}

\newcommand{\Spec}{\textnormal{Spec}}

\def\mod{\mathop{\textnormal{mod}}\nolimits}
\def\Mod{\mathop{\textnormal{Mod}}\nolimits}

\newcommand{\E}{\mathcal{E}}

\newcommand{\Chern}{\textnormal{c}_1}
\newcommand{\Per}{\textnormal{Per}}

\let\oldtocsection=\tocsection
\let\oldtocsubsection=\tocsubsection
\let\oldtocsubsubsection=\tocsubsubsection
\renewcommand{\tocsection}[2]{\hspace{0em}\oldtocsection{#1}{#2}}
\renewcommand{\tocsubsection}[2]{\hspace{1em}\oldtocsubsection{#1}{#2}}
\renewcommand{\tocsubsubsection}[2]{\hspace{2em}\oldtocsubsubsection{#1}{#2}}

\begin{document}

\begin{abstract}
We construct Kn\"{o}rrer type equivalences outside of the hypersurface case; namely, between singularity categories of cyclic quotient surface singularities and certain finite dimensional local algebras. This generalises Kn{\"o}rrer's equivalence for singularities of Dynkin type A (between Krull dimensions $2$ and $0$) and yields many new equivalences between singularity categories of finite dimensional algebras.

Our construction uses noncommutative resolutions of singularities, relative singularity categories, and an idea of Hille \& Ploog yielding strongly quasi-hereditary algebras which we describe explicitly by building on Wemyss's work on reconstruction algebras. Moreover, K-theory gives obstructions to generalisations of our main result.

\end{abstract}
\maketitle

\section{Introduction}
\noindent
The singularity category of a  Noetherian ring was introduced by Buchweitz \cite{BuchweitzMCM} to provide a general framework for Tate-cohomology. Orlov \cite{OrlovTriangulatedCatSing} rediscovered a global version motivated by string theory and homological mirror symmetry, which has recently attracted a lot of interest. The singularity category is defined as the Verdier quotient 
\[
D_{sg}(A) := D^b(A\text{-} \mod)/\textnormal{Perf}(A)
\]
where $\Perf(A)$ denotes the subcategory of perfect complexes.
For commutative hypersurfaces the singularity category recovers the homotopy category of matrix factorisations \cite{BuchweitzMCM}. Matrix factorisations appeared in Dirac's seminal description of the electron \cite{Dirac}, and more recently gave rise to new knot invariants \cite{KhovanovRozansky}, occured in string theory \cite{KapustinLi}, and have been used to describe derived categories of Calabi-Yau hypersurfaces \cite{OrlovGraded} -- see Murfet \cite{MurfetMSRIspring2013} for a nice survey.

Two rings with triangle equivalent singularity categories are called \emph{singular equivalent}. In general, it is a difficult problem to construct singular equivalent rings. 

Kn{\"o}rrer's periodicity is a fundamental phenomenon famously producing singular equivalences between the hypersurfaces $S/(f)$ and $S[[x,y]]/(f+xy)$ for all non-zero polynomials  $f \in S:=\mathbb{C}[[z_0, \ldots, z_d]]$,  \cite{KnorrerCMmodules}. 
A well known and widely used special case shows that 
\[
K = \frac{\mathbb{C}[z]}{(z^r)}  \quad \text{and} \quad R= \frac{\mathbb{C}[[x,y,z]]}{(z^r+xy)}
\]
are singular equivalent. Here, $K$ is a finite dimensional algebra and $R$ is a Gorenstein cyclic quotient surface singularity. It is natural to ask whether there is a generalisation of these Kn\"{o}rrer equivalences to all cyclic quotient surface singularities.

\begin{pQuestion}
Let $R$ be a general cyclic surface quotient singularity. Does a singularly equivalent  finite dimensional algebra $K$ exist? Can $K$ be described explictly?
\end{pQuestion}

In this paper, we give a positive answer to these questions by constructing and explictly describing such finite dimensional algebras, which we call \emph{Kn\"{o}rrer invariant algebras}. In general these are noncommutative algebras, and we also show that their representation theory encodes many geometric aspects of the minimal resolution of $\Spec \, R$. 

Whilst several generalisations and new interpretations of Kn{\"o}rrer's result have been obtained in the hypersurface case, see e.g. \cite{OrlovTriangulatedCatSing, MichaelBrown, Shipman}, our results provide the first evidence that Kn\"{o}rrer type equivalences are not just a hypersurface phenomenon.

\subsection*{Main results and strategy of proof}

We consider the invariant algebras $R_{r,a}:=\mathbb{C}[[x,y]]^{\frac{1}{r}(1,a)}$ where 
\[
\frac{1}{r}(1,a):= \left\langle \left( \begin{array}{ c c} \varepsilon_r & 0 \\ 0 & \varepsilon_r^a \end{array} \right) \right\rangle < \GL_2(\mathbb{C})
\]
for $r$ and $a$ coprime integers such that $0 < a <r$ and $\varepsilon_r$ a primitive $r^{th}$ root of unity. The corresponding singularity $\Spec \, R_{r,a}$ is a cyclic quotient surface singularity, and it is a hypersurface if and only if $a=r-1$. In all other cases, the singularities $R_{r,a}$ are rational but not Gorenstein.

For each invariant algebra $R_{r,a}$, we construct a finite dimensional algebra $K_{r,a}$ and an equivalence of singularity categories. Rather than attempting to do this directly, the key idea of this paper is to work with noncommutative resolutions of the singularities. The singular equivalence is then induced by an equivalence of relative singularity categories. 
Here, the \emph{relative singularity category} associated to a noncommutative resolution $\Lambda$ of a singularity $K$ is the triangulated quotient category (see Section \ref{Singularity categories} for details) 
\[
\Delta_K(\Lambda):= \frac{D^b( \Lambda \text{-}\mod)}{\textnormal{Perf}(K)}.
\] 

The noncommutative resolutions $A_{r,a}$ of $R_{r,a}$ and $\Lambda_{r,a}$ of $K_{r,a}$ in the following key result are described in more detail immediately below the Corollary.

\begin{pTheorem}[{Theorem \ref{Equivalence of Relative Singularity Categories Theorem}}] There is an equivalence of relative singularity categories
\[
\Delta_{K_{r,a}}(\Lambda_{r,a}):=\frac{D^b(\Lambda_{r,a})}{\Perf(K_{r, a})}\cong \frac{D^b(A_{r,a})}{\Perf(R_{r, a})} =:\Delta_{R_{r,a}}(A_{r,a}).
\]
\end{pTheorem}

This equivalence descends to an equivalence of singularity categories, which are explicit Verdier quotients of the relative singularity categories, see Section \ref{Singularity categories} and \cite{ KalckYang, TV}.

\begin{pCorollary}[Theorem \ref{SC for KI}]
There is an equivalence of singularity categories
\[
D_{sg}(K_{r,a}):=\frac{D^b(K_{r,a})}{\Perf(K_{r, a})} \cong \frac{D^b(R_{r,a})}{\Perf(R_{r, a})} =:D_{sg}(R_{r,a}).
\]
\end{pCorollary}

Let us explain our strategy in more detail - in particular, the involved noncommutative resolutions and the construction of $K_{r,a}$. On the one hand, the surface singularity $\Spec \, R_{r,a}$ admits a minimal resolution, which has a tilting bundle by work of Van den Bergh \cite{3Dflops} which builds on work of Wunram \cite{WunramReflexiveModules}. Its endomorphism algebra $A_{r, a}$ is called reconstruction algebra and was explicitly described by Wemyss \cite{RCAA}, see Sections \ref{Sec:NCRes} and \ref{Reconstruction algebras}. This is a natural choice for a noncommutative resolution of $R_{r, a}$. 

On the other hand, it is one of our key insights that the algebras $\Lambda_{r, a}$ defined by Hille and Ploog \cite{HillePloog} are finite dimensional analogues of type A reconstruction algebras. In order to give an idea of the construction of $\Lambda_{r, a}$, recall that the exceptional divisor of the minimal resolution $V_{r, a}$ of $\Spec \, R_{r,a}$ is a chain of smooth rational curves $C_1, \ldots, C_n$.  Such a chain can also be considered in a smooth rational projective surface $X_{r, a}$. The algebra $\Lambda_{r, a}$ is constructed by iterated universal extensions from a certain exceptional collection of line bundles on $X_{r,a}$ associated to this chain $C_1, \ldots, C_n$, cf. \cite{HillePloog} and Section \ref{Subcategory}. We then define the Kn{\"o}rrer invariant algebra as a corner algebra $K_{r, a}:=e \Lambda_{r, a} e$ for a certain idempotent $e \in \Lambda_{r, a}$ and show that $\Lambda_{r, a}$ is a noncommutative resolution of  $K_{r, a}$.

Summing up, the singularity $R_{r,a}$ is starting data in our construction, whereas the Kn{\"o}rrer invariant algebra $K_{r,a}$ is only produced a posteriori from the algebra $\Lambda_{r,a}$. 
From a geometric perspective the construction of $\Lambda_{r, a}$ captures similar information about the minimal resolution of the singularity as the reconstruction algebra $A_{r,a}$. The equivalence of relative singularity categories in our Theorem above makes this precise.

\

More generally, the techniques developed in this paper allow us to identify the (relative) singularity categories of  partial resolutions of $\Spec \,R_{r,a}$ with (relative) singularity  categories of $e\Lambda_{r, a}e$ for an explicit idempotent $e$, see Corollary \ref{C:GeneralizedEquivofRelSingCats} and Theorem \ref{SC for KI generalized}. As a consequence, we obtain many new and non-trivial  equivalences between singularity categories of finite dimensional algebras which might be of independent interest and appear difficult to produce via other methods, see Corollary \ref{cor:idempotent singularity equivalences}.

\subsection*{Descriptions of the algebras $K_{r,a}$ and $\Lambda_{r,a}$} 
The algebras $\Lambda_{r,a}$ arise from a natural geometric category. Hille \& Ploog \cite{HillePloog} show that  $\Lambda_{r,r-1}$ is the Auslander algebra of $k[x]/(x^r)$.
We give the first explicit description of all other $\Lambda_{r,a}$ in terms of quivers with relations (Proposition \ref{Lambda Presentation Proposition}). 
Using this, we show that
 $\Lambda_{r,a}$ is a noncommutative resolution of $K_{r, a}$ (Theorem \ref{Endomorphism Algebra Isomorphism Theorem}). This generalises \cite{HillePloog} and yields an explicit description of $K_{r, a}$.

\begin{pLemma}[Lemma \ref{KI Presentation Lemma}]
There is an algebra isomorphism 
\[
K_{r,a} \cong \frac{\mathbb{C}\langle z_1 \dots, z_l \rangle}{I}
\]
where I is the two sided ideal generated by 
\[ z_i z_j \text{ if $i<j$} \quad \text{ and } \quad
z_i \left(  z_i^{\beta_i-2} \right) \left(  z_{i+1}^{\beta_{i+1}-2} \right)  \dots {\left(  z_{j-1}^{\beta_{j-1}-2} \right) }\left(  z_j^{\beta_j-2} \right) z_j \text{ for $j \le i$}
\]
where $l$ and the $\beta_i$ are defined by the Hirzebruch-Jung continued fraction expansion $r/(r-a)=[\beta_1,\dots, \beta_l]$.
\end{pLemma}

This presentation is similar to Riemenschneider's description of $R_{r,a}$, which is recalled in Section \ref{Sec:R presentation}. The algebra $K_{r, a}$ is only commutative in the extreme cases $a=r-1$ and $a=1$, and the first example of a noncommutative Kn\"{o}rrer invariant algebra is
\[
K_{5,2} \cong \frac{\mathbb{C}\langle z_1,z_2 \rangle }{(z_1^2, z_2^3, z_1z_2,z_2^2z_1). }
\]

The algebra $\Lambda_{r,a}$ has global dimension two, see Proposition \ref{Proposition HillePloog Abelian Category}. In the process of calculating the presentation of $K_{r,a}$ we show that $\Lambda_{r,a}$ is a noncommutative resolution of the noncommutative singularity $K_{r,a}$ in the following sense.
\begin{pTheorem}[Theorem \ref{Endomorphism Algebra Isomorphism Theorem}]
There is a $\mathbb{C}$-algebra isomorphism \[
\End_{K_{r,a}} \left( \bigoplus I_i \right) \cong \Lambda_{r,a}\]
where the sum is over all isomorphism classes of indecomposable left ideals $I_i$ of $K_{r,a}$.
\end{pTheorem}

The isomorphism classes of nontrivial indecomposable left ideals $I_i$ of $K_{r,a}$ are in bijection with the curves $C_i$ making up the exceptional divisor in the minimal resolution of $R_{r,a}$, see Theorem \ref{Endomorphism Algebra Isomorphism Theorem}. The Kn\"{o}rrer invariant algebra $K_{r,a}$ captures further aspects of the geometry of the cyclic quotient surface singularity $\Spec \, R_{r,a}$. 

\begin{pProposition}[Proposition \ref{Properties of KI Algebras Proposition}]
Consider the Kn\"{o}rrer invariant algebra $K_{r,a}$.
\begin{enumerate}
\item The dimension of $K_{r,a}$ is $r$, the order of the cyclic group $\frac{1}{r}(1,a)$.
\item The proper monomial left ideal of $K_{r,a}$ of largest $\mathbb{C}$-dimension  has dimension $a$.
\item The minimal number of generators of $K_{r,a}$ is 2 less than the embedding dimension of $\Spec \, R_{r,a}$.
\item The highest degree of a nonzero monomial equals the number of exceptional curves in the minimal resolution of $\Spec \, R_{r,a}$. 
\end{enumerate}
\end{pProposition}

A more detailed analysis shows that the Euclidean algorithm for the pair $(r, a)$ is encoded in the ideal structure of $K_{r,a}$ and conversely one can use the Euclidean algorithm to build $K_{r, a}$ recursively. This will be explained in future work.

 \subsection*{Further results and related work}
We survey a selection of related results in the literature and consequences of our constructions which might be of independent interest.

The singularity category of the following Kn{\"o}rrer invariant algebras
\[
K_{r,1}= \mathbb{C}\langle z_1, \dots z_{r-1} \rangle/(z_1, \dots, z_{r-1})^2
\] has appeared previously in a range of incarnations including categories appearing in Lie-theory, noncommutative projective geometry, and Leavitt path algebras.

\begin{Theorem}
Let $r \geq 1$. The following categories are triangle equivalent:
\begin{itemize}
\item[(a)] $D_{sg}(R_{r, 1})$;
\item[(b)] $D_{sg}(K_{r, 1})$;
\item[(c)] $\mathsf{qgr} \ \mathbb{C}\langle x_1, \ldots, x_{r-1} \rangle$ with $\deg x_i =1$ and degree shift functor, see \cite{SPSmithCategoryequivalences2012};
\item[(d)] $U\left(\mathfrak{sl}\left((r-1)^\infty\right)\right)/I-\mathrm{fpmod}$, see \cite{HennigSierra}. 
\item[(e)] $L(Q_{r-1})$-$\mathsf{grproj}$, where $L(Q_{r-1})$ denotes the Leavitt path algebra of a particular quiver $Q_{r-1}$ with degree shift funtor, see e.g. \cite{ChenYang};
\end{itemize}
\end{Theorem}

Whilst the equivalence between (a) and (b) is a consequence of the results of this paper, the authors in fact learned of this particular equivalence from Dong Yang, who obtained it using explicit dg algebra techniques. This was the motivating example for this paper.

For the equivalences between (b) and (c)  and the definition of $\mathsf{qgr}$-construction for non-noetherian algebras  see \cite{SPSmithCategoryequivalences2012} - in particular, Theorem 7.2.  The equivalence between the categories in (c) and (d)  follows from \cite[Theorem 1.2.]{HennigSierra} by passing to the subcategories of finitely presented objects, cf. \cite[Proposition 1.4.]{SPSmithCategoryequivalences2012}. The equivalence between the categories in (b) and (e) and the definition of the quiver $Q_{r-1}$ and the Leavitt path algebra of a quiver can be found in \cite[Theorem 6.1.]{ChenYang}, which builds on \cite{SPSmithCategoryequivalences2012} and \cite{XWChenSingCatofRadicalsquare0alg}. 

The singularities $R_{r,a}$ generalise the hypersurface singularities $R_{r,1}$ appearing in (a), and the results of this paper produce finite dimensional algebras $K_{r,a}$ and equivalences that generalise the equivalence between (a) and (b); it is an obvious question whether there are  generalisations of the objects and equivalences occurring in (c), (d), and (e).

\medskip

It is an obvious question whether our results can be generalised to  produce  Kn\"{o}rrer type equivalences for further singularities. For nonabelian quotient surface singularities, 
we conjecture that there are no straightforward generalisations and our main result is optimal in the following sense.

\begin{Conjecture} \label{conjecture} If $G < \GL(2, \mathbb{C})$ is a nonabelian finite subgroup and $R=\mathbb{C}\llbracket x, y \rrbracket^G$, then there does not exist a finite dimensional local $\mathbb{C}$-algebra $S$ such that
\[
D_{sg}(R) \cong D_{sg}(S).
\]
\end{Conjecture}

We provide evidence for this conjecture in Section \ref{Sec: Obstructions}: an analysis of the Grothendieck groups of $D_{sg}(R)$ and $D_{sg}(S)$ yields obstructions to singular equivalences and shows that the conjecture is true for all finite subgroups $G < \SL(2,\mathbb{C})$ and many non-Gorenstein singularities of type $D$.

 \medskip

The results of this paper also expand on aspects of Hille \& Ploog's paper \cite{HillePloog}. The construction of the algebra $\Lambda_{r,a}$ in this paper is taken directly from \cite{HillePloog}, and we elaborate on Hille and Ploog's construction by providing an explicit presentation of $\Lambda_{r,a}$ in Proposition \ref{Lambda Presentation Proposition} and proving that $\Lambda_{r,a}$ is a noncommutative resolution of the Kn\"{o}rrer invariant algebra $K_{r,a}$ in Theorem \ref{Endomorphism Algebra Isomorphism Theorem}.

Moreover, the description of $\Lambda_{r,a}$ as a noncommutative resolution of $K_{r,a}$ allows us to understand a further aspect of Hille \& Ploog's work: the construction of $\Lambda_{r,a}$ depends on a choice of direction for the curves and it is natural to ask how the two choices are related. We address this question in further work, \cite{KalckKarmazynRingeldual}, where we show that the quasi-hereditary structure on the algebras is unique and the algebras associated to the two orientations of the curves are related by Ringel duality.

\subsection*{Structure of Paper}
Section 2 recalls Hille and Ploog's construction of the algebras $\Lambda_{r,a}$ and also Wemyss's reconstruction algebras $A_{r,a}$, which are endomorphism algebras of Van den Bergh's tilting bundles in the special case of cyclic quotient surface singularities. The definitions of singularity categories and relative singularity categories and relevant related results are recapped in Section 3.
The main results of this paper are proved in Section 4, where a functor is constructed and shown to induce the desired equivalences of (relative) singularity categories. Evidence supporting Conjecture \ref{conjecture}, stating that our main results cannot be naively generalised to nonabelian quotient surface singularities, is presented in Section 5. Finally, in Section 6 we provide explicit descriptions of the algebras $\Lambda_{r,a}$ and $K_{r,a}$ in terms of quivers and relations.

\subsection*{Notational preliminaries}

For $X$ a quasi-compact and separated scheme we let $D(X)=D(\QCoh \, X)$ denote the unbounded derived category of quasicoherent sheaves on $X$ and $D^b(X)=D^b(\Coh \, X)$ denote the bounded derived category of coherent sheaves. Then $D(X)$ is closed under small direct sums \cite[Example 1.3]{Neeman} and $D(X)$ is compactly generated with compact objects the perfect complexes \cite[Proposition 2.5]{Neeman} which we denote by $\Perf(X)$.  We use similar notation for left modules $A$-$\Mod$ and finitely generated left modules $A$-$\mod$ over a Noetherian $\mathbb{C}$-algebra $A$. 

We  will use the composition rule for functions that that $f$ followed by $g$ is denoted $fg$, and this will also be our composition rule for arrows in a quiver.

\subsection*{Acknowledgments}

The authors thank Michael Wemyss, Agnieszka Bodzenta, Xiao-Wu Chen, Osamu Iyama, Alastair King, Dmitri Orlov, and Michel Van den Bergh for useful discussions and input. They also particularly thank Lutz Hille, David Ploog and Dong Yang for access to unpublished work that inspired this paper.

The first author was supported by EPSRC grant EP/L017962/1. The second author was supported at the beginning of this work on EPSRC grant EP/I004130/2 and afterwards on the EPSRC grant EP/M017516/1.

\section{Resolutions of singularities} \label{Geometry of curves} This section recalls geometric and noncommutative resolutions of cyclic surface quotient singularities $R_{r,a}$ that are later used to build equivalences between the singularity categories of $R_{r,a}$ and a finite dimensional algebra $K_{r,a}$. 

\subsection{Geometric setup}

We are interested in cyclic surface quotient singularities $R_{r,a}:= \mathbb{C}[[x,y]]^{\frac{1}{r}(1,a)}$, and we first consider how such a singularity may occur as an isolated singular point of a projective surface. 

Take $X$ a smooth projective surface such that $\CH^i(\mathcal{O}_X)=0$ for $i >1$ that contains a type $A_n$ configuration of rational curves, $C:=\cup \, C_i \subset X$ such that $C_i \cong \mathbb{P}^1$ with self-intersection numbers $C_i \cdot C_i := -\alpha_i  \le -2$
that can be contracted to a point, and let $\pi:X \rightarrow Y$ denote the this contraction.

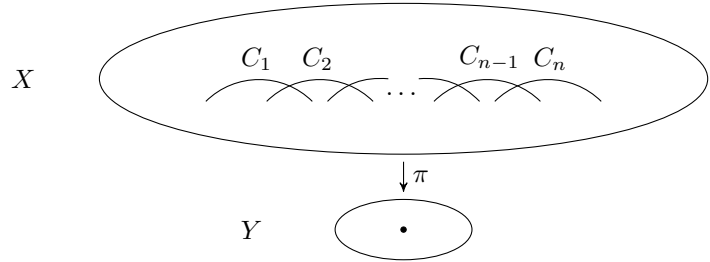
\[
\begin{tikzpicture}   
\node (X) at (-5,0) {$X$};
\node (Y) at (-2,-2) {$Y$};

\draw (-2.6,-0.3) to [bend left=45] node[pos=0.48, above] {$C_1$} (-1.2,-0.3);
\draw (-1.8,-0.3) to [bend left=45] node[pos=0.48, above] {$C_2$} (-0.4,-0.3);
\draw (-1,-0.3) to [bend left=25] node[pos=0.48, right] {} (-0.2,0);
\node (D) at (0,-0.2) {$\dots$};

\draw (2.6,-0.3) to [bend right=45] node[pos=0.48, above] {$C_n$} (1.2,-0.3);
\draw (1.8,-0.3) to [bend right=45] node[pos=0.48, above] {$C_{n-1}$} (0.4,-0.3);
\draw (1,-0.3) to [bend right=25] node[pos=0.48, right] {} (0.2,0);

\draw[black] (0,0) ellipse (4 and 1);
\draw[->] (0,-1.1) -- (0,-1.5) node[pos=0.48, right] {$\pi$};
\draw[black] (0,-2) ellipse (0.9 and 0.4);
\filldraw [black] (0,-2) circle (1pt);
\end{tikzpicture}
\]
The morphism $\pi:X \rightarrow Y$ is a minimal resolution of singularities and $Y$ has a unique singular point with germ $R_{r,a}=\mathbb{C}[[x,y]]^{\frac{1}{r}(1,a)}$ where $0<a<r$ are coprime integers that can be calculated from the Hirzebruch-Jung continued fraction 
\[
\frac{r}{a} = \alpha_1 - \cfrac{1}{\alpha_{2}
          -\cfrac{1}{\dots - \cfrac{1}{\alpha_n} } } =\left[ \alpha_1, \dots, \alpha_n \right],
\]
see \cite[Section 3]{RiemenschneiderNach} or \cite{Hirzebruch1953,Jung1908}. 

We will use either the notation $X_{r,a}$ or $X_{[\alpha_1, \dots, \alpha_n]}$ for the variety $X$ to emphasize the additional data of a chosen type $A$ configuration as above.

\begin{Remark}\label{R:taut}
Cyclic quotient surface singularities are \emph{taut} in the sense of Laufer \cite{Laufer}: if $\Spec \, R$ is the germ of an isolated surface singularity whose minimal resolution also has a type $A$ configuration of curves $E:=\cup E_i$ as the exceptional divisor with $E_i \cong \mathbb{P}^1$ and $E_i \cdot E_i \cong -\alpha_i$ then $R \cong R_{r,a}$ where $r/a= [ \alpha_1, \dots, \alpha_n]$.
\end{Remark}

We can also consider the singularity as the complete local affine scheme $\Spec \, R_{r,a}$. Let $U_{r,a}$ denote a complete, local, affine, neighbourhood of the singular point in $Y_{r,a}$. Then $U_{r,a} \cong \Spec \,  R_{r,a}$. Let $u:U_{r,a} \rightarrow Y_{r,a}$ denote the inclusion of this affine scheme, and  define $V_{r,a} \xrightarrow{p} U_{r,a}$ to be the pullback of the minimal resolution $\pi: X \rightarrow Y$. Then the pullback $p:V_{r,a} \rightarrow U_{r,a}$ is the minimal resolution of the singularity $\Spec \, R_{r,a}$ as a scheme and there is an affine inclusion $v:V_{r,a} \rightarrow X_{r,a}$. In particular, the morphisms $u$ and $v$ are flat and affine and the closed immersion of the curves $C:= \cup \, C_i$ into $X_{r,a}$ factors through $v:V_{r,a} \rightarrow X_{r,a}$.

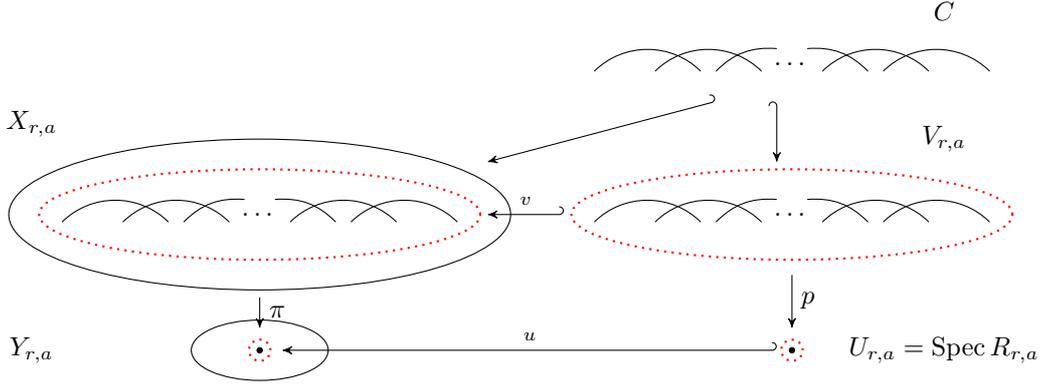
\[
\begin{tikzpicture}

\draw (4.4,1.7) to [bend left=45] node[pos=0.48, above] {} (5.8,1.7);
\draw (5.2,1.7) to [bend left=45] node[pos=0.48, above] {} (6.6,1.7);
\draw (6,1.7) to [bend left=25] node[pos=0.48, right] {} (6.8,2);
\node (D) at (7,1.8) {$\dots$};

\draw (9.6,1.7) to [bend right=45] node[pos=0.48, above] {} (8.2,1.7);
\draw (8.8,1.7) to [bend right=45] node[pos=0.48, above] {} (7.4,1.7);
\draw (8,1.7) to [bend right=25] node[pos=0.48, right] {} (7.2,2);

\draw (-2.6,-0.3) to [bend left=45] node[pos=0.48, above] {} (-1.2,-0.3);
\draw (-1.8,-0.3) to [bend left=45] node[pos=0.48, above] {} (-0.4,-0.3);
\draw (-1,-0.3) to [bend left=25] node[pos=0.48, right] {} (-0.2,0);
\node (D) at (0,-0.2) {$\dots$};

\draw (2.6,-0.3) to [bend right=45] node[pos=0.48, above] {} (1.2,-0.3);
\draw (1.8,-0.3) to [bend right=45] node[pos=0.48, above] {} (0.4,-0.3);
\draw (1,-0.3) to [bend right=25] node[pos=0.48, right] {} (0.2,0);

\draw (4.4,-0.3) to [bend left=45] node[pos=0.48, above] {} (5.8,-0.3);
\draw (5.2,-0.3) to [bend left=45] node[pos=0.48, above] {} (6.6,-0.3);
\draw (6,-0.3) to [bend left=25] node[pos=0.48, right] {} (6.8,0);
\node (D) at (7,-0.2) {$\dots$};

\draw (9.6,-0.3) to [bend right=45] node[pos=0.48, above] {} (8.2,-0.3);
\draw (8.8,-0.3) to [bend right=45] node[pos=0.48, above] {} (7.4,-0.3);
\draw (8,-0.3) to [bend right=25] node[pos=0.48, right] {} (7.2,0);

\draw[black] (0,-0.2) ellipse (3.3 and 1);
\draw[->] (0,-1.3) -- (0,-1.7) node[pos=0.48, right] {$\pi$};
\draw[black] (0,-2) ellipse (0.9 and 0.4);
\filldraw [black] (0,-2) circle (1pt);
\draw [red, thick, dotted] (0,-2) circle (4pt);

\draw[red, thick, dotted] (0,-0.2) ellipse (2.9 and 0.6);

\draw[->] (7,-1) -- (7,-1.7) node[pos=0.48, right] {$p$};
\draw [red, thick, dotted] (7,-2) circle (4pt);

\draw[red, thick, dotted] (7,-0.2) ellipse (2.9 and 0.6);
\filldraw [black] (7,-2) circle (1pt);

\draw [left hook->] (6.8,-2) to node[above]  {\scriptsize{$u$}} (0.3,-2);
\draw [left hook->] (4,-0.2) to node[above]  {\scriptsize{$v$}} (3,-0.2);

\draw [left hook->] (6.8,1.3) to node[above]  {} (6.8,0.5);

\draw [left hook->] (6,1.3) to node[above]  {} (3,0.5);

\node (C1) at (-3,-2)  {$Y_{r,a}$};
\node (C2) at (9,-2)  {$U_{r,a} = \Spec \, R_{r,a}$};
\node (C3) at (-3,1)  {$X_{r,a}$};
\node (C4) at (9,0.8)  {$V_{r,a}$};

\node (C5) at (9,2.5)  {$C$};

\end{tikzpicture}
\]

Having introduced the singularity and resolution we now introduce two noncommutative resolutions; one of $R_{r,a}$ and one of, the yet undefined, $K_{r,a}$.

\subsection{Noncommutative resolution of $R_{r,a}$} \label{Sec:NCRes} The projective morphism $p:V_{r,a} \rightarrow U_{r,a} = \Spec \, R_{r,a}$ contracts a collection of curves $C_i$ for $1 \le i \le n$ to a point. It has one dimensional fibres and it follows from $H^1(\mathcal{O}_{X_{r,a}})=0$ for $i>1$ that  $\mathbf{R}p_* \mathcal{O}_{V_{r,a}}\cong \mathcal{O}_{U_{r,a}}$. A corresponding noncommutative resolution was constructed for such morphisms by Van den Bergh \cite{3Dflops}, and we apply these results in the particular case of $p:V_{r,a} \rightarrow U_{r,a}$ considered in this paper. We start by recalling the structure of line bundles on $V_{r,a}$.

\begin{Proposition}[{See \cite{WunramReflexiveModules} or\cite[Section 3.4 and Lemma 3.5.1]{3Dflops}}] \label{Line Bundle Proposition}
For the projective morphism $p:V_{r,a} \rightarrow U_{r,a} = \Spec \, R_{r,a}$ contracting curves $C_i$:
\begin{enumerate}
\item There exist divisors $D_i \subset V_{r,a}$ such that $D_i \cap C_j = \left\{ 
\begin{array}{cc}
pt & \text{ if $i=j$},\\
\varnothing & \text{ otherwise}
\end{array}
\right.$.
\item A line bundle $\mathcal{L}$ on $V_{r,a}$ is isomorphic to $\mathcal{O}_{V_{r,a}}(\sum a_i D_i)$ where $a_i = \deg \mathcal{L} \, |_{C_i}$. In particular, this map is an isomorphism between the Picard group of $V_{r,a}$ and  $\mathbb{Z}^n$. 
\end{enumerate}
\end{Proposition}

In such a situation recall the abelian category $\mathcal{A}_{r,a}:={^{0}\Per(V_{r,a}/U_{r,a})}$.
\begin{Definition}[{See \cite[Section 3.1]{3Dflops}}] \label{Perverse Definition} Define $\mathfrak{C}$ to be the abelian subcategory of $\Coh \, V_{r,a}$ consisting of $\mathcal{F} \in \Coh\, V$ such that $\mathbf{R}p_*\mathcal{F} \cong 0$. The abelian category $\mathcal{A}_{r,a}$ is defined to be the heart of a $t$-structure on $D^b(V_{r,a})$ that contains the objects $\E \in D^b(V_{r,a})$ satisfying the following conditions:
\begin{enumerate}
\item The only non-vanishing cohomology of $\E$ lies in degrees $-1$ and $0$.
\item $p_*\mathcal{H}^{-1}(\E)=0$ and  $\mathbf{R}^1 p_* \mathcal{H}^0(\E)=0$, where $\mathcal{H}^j(-)$ denotes the $j^{th}$ cohomology\mbox{ sheaf.}
\item $\Hom_{V_{r,a}}(C,\mathcal{H}^{-1}(\E))=0$ for all $C \in \mathfrak{C}$.
\end{enumerate}
\end{Definition}
Applying this result to a cyclic surface quotient singularity $R_{r,a}$ provides a noncommutative resolution of $R_{r,a}$.

\begin{Theorem}[{\cite[Section 3.5]{3Dflops}}] \label{Affine Abelian Theorem}
The abelian category $\mathcal{A}_{r,a}:={^{0}}\Per(V_{r,a}/R_{r,a})$ has a projective generator, $n+1$ simple objects, and $n+1$ indecomposable projective objects.
\begin{enumerate}
\item The simple objects are $s_i := \mathcal{O}_{C_i}(-1)$ for $1 \le i \le n$ and $s_{0}= \omega_C[1]$. 
\item The indecomposable projective objects are $P_{0}:= \mathcal{O}$ and $P_i:= \mathcal{O}(-D_i)$ for \mbox{$1 \le i \le n$.}
\item Any projective object in $\mathcal{A}_{r,a}$ is a direct sum of the $P_i$ and so is uniquely determined by its rank and first Chern class.
\end{enumerate}
The basic projective generator $T=\bigoplus_{i=0}^n P_i$ induces an equivalence of abelian categories
\begin{center}
\[
\begin{tikzpicture} [bend angle=15, looseness=1]
\node (C1) at (-1,0)  {$\mathcal{A}_{r,a}$};
\node (C2) at (3,0)  {$A_{r,a}-\mod$};
\node (C) at (1,0)  {$\cong$};
\draw [->,bend left] (C1) to node[above]  {$\scriptstyle{\Hom_{D(V_{r,a})}(T,-)}$} (C2);
\draw [->,bend left] (C2) to node[below]  {\scriptsize{$ T \otimes_{A_{r,a}}(-)$}} (C1);
\end{tikzpicture}
\]
\end{center}
where $A_{r,a}:= \End_{V_{r,a}}(T)$. This induces an triangle equivalence $D^b(A_{r,a}) \cong D^b(V_{r,a})$.
\end{Theorem}
\begin{Remark}
To ease notation we will also let $s_i$ and $P_i$ denote the corresponding simple and projective objects in $A_{r,a}$-$\mod$ under this equivalence of abelian categories.
\end{Remark}

Such a situation, the minimal resolution of a cyclic quotient surface singularity, occurs in the $\GL_2(\mathbb{C})$ McKay correspondence \cite{WemyssGL2}, and in this situation the algebra $A_{r,a}^{\op} \cong \End_{V_{r,a}}(T^{\vee})$ has been explicitly identified as the reconstruction algebra of type $A$ in \cite{RCAA}. We recall a presentation of the reconstruction algebra of type $A$ in Section \ref{Reconstruction algebras}.

\subsection{A triangulated category of Hille and Ploog}\label{Subcategory} Having recalled a noncommutative resolution of $R_{r,a}$ we now consider a different triangulated category associated to the resolution $\pi: X_{r,a} \rightarrow Y_{r,a}$ introduced by Hille and Ploog \cite{HillePloog}. This is a full, thick subcategory of $D^b(X_{r,a})$, and Hille and Ploog have shown, using universal extension techniques, that it has an internal structure: it is the derived category of finitely generated modules over a particular quasi-hereditary, finite dimensional algebra.

\begin{Definition}
For $0 \le i  \le n$ define the line bundles
\[
\mathcal{L}_i:= \mathcal{O}(-C_{i+1} \dots -C_{n}),
\]
and the full triangulated subcategory closed under summands
\[
D_{r,a}:= \langle \mathcal{L}_{0}, \dots, \mathcal{L}_n \rangle \subset D^b(X_{r,a}).
\]
\end{Definition}

\begin{Remark}
These definitions depend on a choice of orientation of the type $A_n$ configuration of curves. In particular, labelling the curves with the opposite orientation, $C_n, \dots, C_1$, produces the subcategory $D_{r,a^{-1}}$ where $a^{-1}$ is the inverse of $a$ modulo $r$. This is consistent with the fact that if $r/a=[\alpha_1, \dots, \alpha_n]$ then $r/{a^{-1}}=[\alpha_n, \dots, \alpha_1]$ (see \cite[Section 3.4. (17)]{Hirzebruch1953}).
\end{Remark}

As $D_{r,a}$ is a full subcategory it comes equipped with a fully faithful inclusion $D_{r,a} \subset D^b(X_{r,a})$. Moreover,  Hille and Ploog have shown that the intersection of this subcategory with $\Coh \, X_{r,a}$ is itself an abelian category and is equivalent to the abelian category of finitely generated modules over a finite dimensional algebra. Interpreting their results yields the following description of the category.

\begin{Proposition}[Hille and Ploog \cite{HillePloog}] \label{Proposition HillePloog Abelian Category} The intersection $D_{r,a} \cap \Coh(X_{r,a}) $ is an abelian category with a projective generator. It is (strongly) quasihereditary of global dimension 2, and it has $n+1$ simple objects, $n+1$ standard objects, and $n+1$ indecomposable projective objects. 
\begin{enumerate}
\item
The $n+1$ simple objects $\sigma_0, \dots,\sigma_{n}$ are defined by
 \begin{align*} 
&\sigma_{0}:= \mathcal{O}_X(-C_1 - \dots -C_n)), \\
&\sigma_i := \mathcal{O}_{C_i}(-1), \quad \text{ and } \\
&\sigma_n:=\mathcal{O}_{C_n}.
\end{align*}
\item
The $n+1$ standard objects $\mathcal{L}_0, \dots, \mathcal{L}_{n}$ are defined by $\mathcal{L}_i:= \mathcal{O}_X(-C_{i+1} \dots - C_{n})$. They are related to the simple modules by the following set of short exact sequences.
\begin{equation*}
\begin{array}{ccccccccc}
 0 &\rightarrow & \mathcal{L}_{n-1}& \rightarrow&  \mathcal{L}_n & \rightarrow&    \sigma_n&  \rightarrow&  0 \\
& \vdots &&&&&& \vdots & \\
 0& \rightarrow & \mathcal{L}_{i-1}&  \rightarrow & \mathcal{L}_i & \rightarrow &   \sigma_i  &\rightarrow&  0 \\
& \vdots &&&&&& \vdots & \\
 0 &\rightarrow &\mathcal{L}_{0} & \rightarrow&  \mathcal{L}_1 & \rightarrow&   \sigma_1 & \rightarrow&  0 \\
 0& \rightarrow  &0  &\rightarrow & \mathcal{L}_{0}&  \rightarrow  &  \sigma_{0} & \rightarrow&  0 \\
\end{array}
\end{equation*}
\item
The $n+1$ indecomposable projective objects $\Lambda_0, \dots, \Lambda_{n}$ can be defined as universal extensions of the $\mathcal{L}_i$. They are related to the standard modules by the following set of short exact sequences.
\begin{equation*}
\begin{array}{ccccccccc}
 0 & \rightarrow& 0 & \rightarrow& \Lambda_n  &\rightarrow &   \mathcal{L}_n &  \rightarrow & 0 \\
 0&  \rightarrow & \Lambda_n^{\oplus \alpha_n-2} \oplus \Lambda_n & \rightarrow  &\Lambda_{n-1}  &\rightarrow &  \mathcal{L}_{n-1} & \rightarrow   & 0 \\
& \vdots &&&&&& \vdots & \\
 0 & \rightarrow & \bigoplus_{j=i}^n \Lambda_j^{\oplus \alpha_j-2} \oplus \Lambda_{i}  & \rightarrow &\Lambda_{i-1} & \rightarrow  & \mathcal{L}_{i-1} & \rightarrow & 0 \\
& \vdots &&&&&& \vdots & \\
 0 & \rightarrow & \bigoplus_{j=1}^n \Lambda_j^{\oplus \alpha_j-2} \oplus \Lambda_1 & \rightarrow & \Lambda_{0} & \rightarrow  & \mathcal{L}_{0} & \rightarrow  & 0  \\
\end{array}
\end{equation*}
\end{enumerate}
The basic projective generator $\Lambda:=\bigoplus_{i=0}^{n} \Lambda_{i}$ induces an equivalence of abelian categories
\begin{center}
\[
\begin{tikzpicture} [bend angle=15, looseness=1]
\node (C1) at (0,0)  {$D_{r,a} \cap \Coh \, X_{r,a}$};
\node (C2) at (4,0)  {$\Lambda_{r,a}-\mod$};
\node (C) at (2,0)  {$\cong$};
\draw [->,bend left] (C1) to node[above]  {$\scriptstyle{\Hom_{X_{r,a}}(\Lambda,-)}$} (C2);
\draw [->,bend left] (C2) to node[below]  {\scriptsize{$ \Lambda \otimes_{\Lambda_{r,a}}(-)$}} (C1);
\end{tikzpicture}
\]
\end{center}
where we define the algebra $\Lambda_{r,a}:=\End_{X_{r,a}}(\Lambda)$. This induces an equivalence of triangulated categories $D_{r,a} \cong D^b(\Lambda)$.
\end{Proposition}
\begin{proof}
The existence of a projective generator, the exact equivalence of abelian categories, and the fact that the category $D_{r,a} \cap \Coh(X_{r,a}) $ has global dimension 2 follows from \cite[Theorem 2.5]{HillePloog}. The projective generator is produced by taking iterated universal extensions of the standard modules. Since the sequence of standard objects has only non-zero $\Ext$-groups in degrees $0$ and $1$ this implies that the category is strongly quasi-hereditary; i.e. the standard modules have projective dimension $\le 1$. 

The simple objects, the standard objects, and the relationship between them are also explicitly stated in \cite[Theorem 2.5]{HillePloog}.  It remains to explain part (3).

The indecomposable projective objects $\Lambda_i$ correspond to  indecomposable summands of the projective generator and satisfy $\dim \Hom_X(\Lambda_i,\sigma_j)=\delta_{i,j}$. In a quasi-hereditary category there are surjections $\Lambda_i \rightarrow \mathcal{L}_i$ with kernel $K_i$:
\[
0 \rightarrow K_i \rightarrow \Lambda_i \rightarrow \mathcal{L}_i \rightarrow 0.
\] In a strongly quasi-hereditary category the standard modules have global dimension 1 so the kernel $K_i$ is projective and splits into a sum of indecomposable projective objects $\bigoplus_{j=0}^n \Lambda_j^{\oplus c_{i,j}}$ where $c_{i,j}= \dim \Hom_X(K_i,\sigma_j)$. As $\dim \Hom_X(\Lambda_i,\sigma_j)=\dim \Hom_X(\mathcal{L}_i,\sigma_j)=\delta_{i,j}$, we get $
c_{i,j}= \dim \Ext^1_X(\mathcal{L}_i,\sigma_j)
$. This can be computed
using a long exact sequence:
\[
\dim \Ext^1_X(\mathcal{L}_i,\sigma_j)= \left\{ \begin{array}{c c}
\alpha_j-1 & \text{ if $j=i+1$} \\
\alpha_j-2 & \text{ if $j>i+1$}  \\
0 & \text{ if $j<i+1$} 
\end{array}
\right.
\]
See the proof of Proposition \ref{Ext Quiver Calculation Lemma} for the explicit calculation.
\end{proof}

In particular, $\Lambda_{r,a}$ is basic, finite dimensional, and quasihereditary. As there is an exact equivalence between $\Lambda_{r,a}$-$\mod$ and $D_{r,a} \cap \Coh(X_{r,a}) $ we will abuse notation by identifying the simple, standard, and projective objects in either category.

The two noncommutative algebras  $\Lambda_{r,a}$ and $A_{r,a}$ are related by the pullback functor
\[
\begin{tikzpicture} [bend angle=15, looseness=1]
\node (C2) at (3,0)  {$D_{r,a}\subset D^b(X_{r,a})$};
\node (C1) at (-1,0)  {$D^b(\Lambda_{r,a})$};
\node (C) at (1,0)  {$\cong$};
\draw [->,bend left] (C1) to node[above]  {$\scriptstyle{\mathbf{R}\Hom_{X_{r,a}}(\Lambda,-)}$} (C2);
\draw [->,bend left] (C2) to node[below]  {\scriptsize{$ \Lambda \otimes^{\mathbf{L}}_{\Lambda_{r,a}}(-)$}} (C1);

\node (A1) at (6,0)  {$D^b(V_{r,a})$};
\node (A2) at (9,0)  {$D^b(A_{r,a})$};
\node (C) at (7.5,0)  {$\cong$};
\draw [->,bend left] (A1) to node[above]  {$\scriptstyle{\mathbf{R}\Hom_{V_{r,a}}(T,-)}$} (A2);
\draw [->,bend left] (A2) to node[below]  {\scriptsize{$ T \otimes^{\mathbf{L}}_{A}(-)$}} (A1);

\draw [->, bend left=0] (C2) to node[above]  {$v^*$} (A1);

\end{tikzpicture}
\]
and we will investigate the relationship between $D^b(\Lambda_{r,a})\cong D_{r,a}$ and $D^b(A_{r,a}) \cong D^b(V_{r,a})$ in Section \ref{Relationship}. In order to make the relationship precise we must work with relative singularity categories, and so we recall the relevant definitions and results in the next section.

\section{Singularity categories} \label{Singularity categories}

In this section we recall the definitions of singularity categories and relative singularity categories, and we recap several results that will be vital later on.

\subsection{Notation}
For a triangulated $\mathbb{C}$-linear category $\mathfrak{C}$ we let $\mathfrak{C}^{\omega}$ denote the idempotent completion of the category, which is naturally triangulated \cite{BalmerSchlichting}, and $\mathfrak{C}^{C}$ denote the subcategory of compact objects. A full subcategory is \emph{thick} if it is triangulated and closed under direct summands. If $\mathfrak{C}$ also contains small direct sums then a full subcategory is \emph{localising} if it is triangulated and also closed under all small direct sums, and a localising subcategory is necessarily closed under direct summands and so thick \cite[Proposition 1.6.8]{NeemanTriangulatedCategories}. For a subset $S \subset \mathfrak{C}$ we will let $\langle S \rangle$ denote the smallest thick subcategory of $\mathfrak{C}$ containing $S$ and $ \langle S \rangle^{\oplus}$ denote the smallest localising subcategory of $\mathfrak{C}$ containing $S$.

In this section we assume that $A$ is a Noetherian $\mathbb{C}$-algebra.

\subsection{(Relative) singularity categories} \label{Singularity and relative singularity categories}
The following definition is due to Buchweitz \cite{BuchweitzMCM} and Orlov \cite{OrlovTriangulatedCatSing}.
\begin{Definition}
The \emph{singularity category} of $A$ is the triangulated quotient category
\[
D_{sg}(A):= \frac{D^b(A)}{\Perf(A)}
\]
where $\Perf(A)$ denotes the subcategory of perfect complexes.
\end{Definition}

We also recall the notion of a relative singularity category.

\begin{Definition}\label{D:Singularity Cat} Let $A$ be a Noetherian $\mathbb{C}$-algebra, $e \in A$ an idempotent, and $e A e$ the algebra defined by this idempotent. The relative singularity category of $A$ with respect to $e A e$ is defined to be
\[
\Delta_{eAe}(A):=\frac{D^b(A)}{\langle A e \rangle} \cong \frac{D^b(A)}{\Perf(eAe)};
\]
the idempotent induces an inclusion of triangulated categories $\Perf(eAe) \subseteq D^b(A)$ with image $\langle A e \rangle$.

A common application is when $R$ is a commutative, Noetherian $\mathbb{C}$-algebra, $M:=R\oplus M'$ is a finitely generated $R$-module, $A:=\End_R(R \oplus M)$, and $e \in A$ is the idempotent corresponding to the projection onto the direct summand $R$. In this situation $e A e \cong R$ and we denote the relative singularity category by $\Delta_{R}(A)$.

\end{Definition}

\subsection{Recollements generated by idempotents} \label{Recollemonts generated by idempotents}
Recall that a recollement of triangulated categories $\mathcal{T}, \mathcal{T}'$, and $\mathcal{T}''$ is a collection of functors 
\[
\begin{tikzpicture}
\node (C1) at (0,0)  {$\mathcal{T}''$};
\node (C2) at (4,0)  {$\mathcal{T}$};
\node (C3) at (8,0)  {$\mathcal{T}'$};
\draw [->, bend right=0] (C1) to node[gap]  {$i_*=i_!$} (C2);
\draw [->, bend right=25] (C2) to node[gap]  {$i^*$} (C1);
\draw [->, bend left=25] (C2) to node[gap]  {$i^!$} (C1);

\draw [->, bend right=0] (C2) to node[gap]  {$j^!=j^*$} (C3);
\draw [->, bend right=25] (C3) to node[gap]  {$j_!$} (C2);
\draw [->, bend left=25] (C3) to node[gap]  {$j_*$} (C2);

\end{tikzpicture}
\]
such that 
\begin{enumerate}
\item The functors $(i^*, i_*=i_!,i^!)$ and $(j_!,j^!=j^*,j_*)$ are adjoint triples.
\item The functors $j_!, i_*=i_!, j_*$ are fully faithful.
\item The composition $j^*i_*$ equals zero.
\item For every object $t \in \mathcal{T}$ there exist two distinguished triangles
\[
i_!i^! t \rightarrow t \rightarrow j_*j^* t \rightarrow i_!i^! t[1]
\]
and
\[
j_!j^! t \rightarrow t \rightarrow i_*i^* t \rightarrow j_!j^! t [1]
\]
induced from the unit and counit morphisms. 
\end{enumerate}

Consider an algebra $A$ with an idempotent $e$.  The idempotent produces a recollement 
\[
\begin{tikzpicture}
\node (C1) at (0,0)  {$D_{A/AeA}(A)$};
\node (C2) at (4,0)  {$D(A)$};
\node (C3) at (8,0)  {$D(eAe)$};
\draw [->, bend right=0] (C1) to node[gap]  {$i_*=i_!$} (C2);
\draw [->, bend right=25] (C2) to node[gap]  {$i^*$} (C1);
\draw [->, bend left=25] (C2) to node[gap]  {$i^!$} (C1);

\draw [->, bend right=0] (C2) to node[gap]  {$j^!=j^*$} (C3);
\draw [->, bend right=25] (C3) to node[gap]  {$j_!$} (C2);
\draw [->, bend left=25] (C3) to node[gap]  {$j_*$} (C2);

\end{tikzpicture}
\] induced by deriving the exact functor $j^*=\Hom_A(Ae,-) \cong eA \otimes_A  (-): A$-$\Mod \rightarrow eAe$-$\Mod$ with left adjoint $Ae \otimes_{eAe}(-)$ and right adjoint $\Hom_{eAe}(eA,-)$. The kernel of $j^*$ is $D_{A/AeA}(A) \subset D(A)$; the full subcategory of objects with cohomology in $A/AeA$-$\Mod$.

\begin{Proposition}[{\cite[Remark 2.14 ]{KalckYang}}] \label{Second Isomorphism Theorem Proposition}
With assumptions as above, the functor $i^*$ induces an equivalence of triangulated categories
\[
\frac{D(A)}{\langle Ae \rangle^{\oplus}} \xrightarrow{\cong} D_{A/AeA}(A)
\]
that restricts to an equivalence
\[
\Delta_{eAe}(A)^\omega \xrightarrow{\cong} \left( D_{A/AeA}(A)\right)^C.
\]
\end{Proposition}

The relative singularity category captures the information in the singularity category that is preserved in a  (partial) resolution. In particular, the singularity category $D_{sg}(e A e)$ can be recovered as a Verdier quotient of the relative singularity category.

\begin{Lemma}\label{Quotient from relative to singularity category Lemma}
In the notation above, assume that $A/AeA$ is finite dimensional. 
Let $S \subset \Delta_{eAe}(A)$ be the thick subcategory generated by all simple $A/AeA$-modules. Then the Verdier quotient induces an equivalence of triangulated categories
\[
 \frac{\Delta_{e \Lambda e}(A)}{S} \cong D_{sg}(e \Lambda e),
\]
see \cite[Propositions 3.3 and 6.13]{KalckYang} or \cite{TV}. 
\end{Lemma}

\subsection{Idempotent completeness of (relative) singularity categories}

Motivated by Proposition \ref{Second Isomorphism Theorem Proposition} we show that the categories we are interested in are idempotent complete. The following result can be proved along the lines of \cite[Proposition 2.69]{Kalckthesis}. 

\begin{Proposition}
Let $A$ be a Noetherian $\mathbb{C}$-algebra and let $e \in A$ be an idempotent. Then the relative singularity category $\Delta_{eAe}(A)$ is idempotent complete if and only if the singularity category $D_{sg}(eAe)$ is idempotent complete.  
\end{Proposition}

We recall a special case of Orlov's  \cite[Remark 3.6]{Orlovidempotent}.

\begin{Proposition}
Let $G \in \mathsf{GL}(2, \mathbb{C})$ be a finite subgroup and let $R=\mathbb{C}[[x, y]]^G$ be the  corresponding two dimensional quotient singularity.  Then
$D_{sg}(R)$ is idempotent complete. 
\end{Proposition}

The following result is due to Chen \cite[Corollary 2.4]{XWChenSingCatofRadicalsquare0alg}.

\begin{Proposition} \label{P:DsgFDidempcomplete}
Let $K$ be a finite dimensional algebra. Then $D_{sg}(K)$ is idempotent complete.   
\end{Proposition}

Combining these results we obtain the following consequence which will be useful later.

\begin{Corollary} \label{Idempotent completeness}
Let $A$ be a Noetherian $\mathbb{C}$-algebra and let $e \in A$ be an idempotent such that one of the following conditions hold:
\begin{itemize}
\item[(a)] $eAe$ is a finite dimensional algebra.
\item[(b)] $eAe$ is a two-dimensional quotient singularity over $\mathbb{C}$.
\end{itemize}
Then the relative singularity category $\Delta_{eAe}(A)$ is idempotent complete
\end{Corollary}

\section{Equivalence of (relative) singularity categories} \label{Relationship}
This section uses the geometric construction of Section \ref{Geometry of curves} to create a functor between the derived categories of the noncommutative resolutions $\Lambda_{r,a}$ and $A_{r,a}$ defined in Sections \ref{Sec:NCRes} and \ref{Subcategory}. This functor is shown to induce an equivalence between the relative singularity categories, which then induces an equivalence between the singularity categories of $R_{r,a}$ and the Kn\"{o}rrer invariant algebra $K_{r,a}$.

\subsection{Constructing a functor}
It is a natural question how the two categories $D^b(\Lambda_{r,a}) \cong D_{r,a} \subset D^b(X_{r,a})$ and $D^b(A_{r,a}) \cong D^b(V_{r,a})$ are related\footnote{Note that $D^b(X_{r,a})$ is Hom-finite whereas $D^b(V_{r,a})$ is Hom-infinite, so these categories won't be equivalent.}. We recall that there is a morphism $v: V_{r,a} \rightarrow X_{r,a}$ and a collection of rational curves $C=\cup_{i=1}^n C_i$ contained in both schemes.
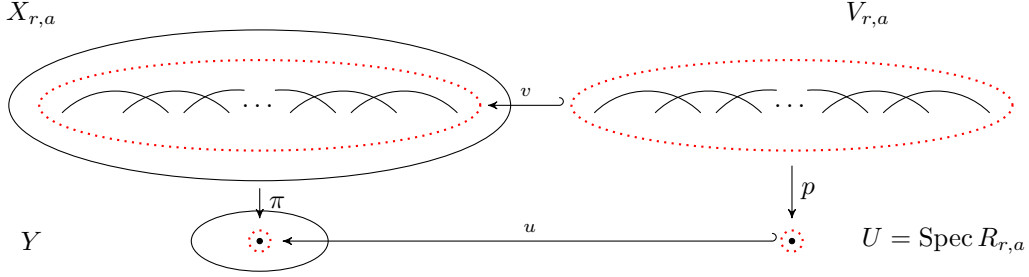
\[
\begin{tikzpicture}   

\draw (-2.6,-0.3) to [bend left=45] node[pos=0.48, above] {} (-1.2,-0.3);
\draw (-1.8,-0.3) to [bend left=45] node[pos=0.48, above] {} (-0.4,-0.3);
\draw (-1,-0.3) to [bend left=25] node[pos=0.48, right] {} (-0.2,0);
\node (D) at (0,-0.2) {$\dots$};

\draw (2.6,-0.3) to [bend right=45] node[pos=0.48, above] {} (1.2,-0.3);
\draw (1.8,-0.3) to [bend right=45] node[pos=0.48, above] {} (0.4,-0.3);
\draw (1,-0.3) to [bend right=25] node[pos=0.48, right] {} (0.2,0);

\draw (4.4,-0.3) to [bend left=45] node[pos=0.48, above] {} (5.8,-0.3);
\draw (5.2,-0.3) to [bend left=45] node[pos=0.48, above] {} (6.6,-0.3);
\draw (6,-0.3) to [bend left=25] node[pos=0.48, right] {} (6.8,0);
\node (D) at (7,-0.2) {$\dots$};

\draw (9.6,-0.3) to [bend right=45] node[pos=0.48, above] {} (8.2,-0.3);
\draw (8.8,-0.3) to [bend right=45] node[pos=0.48, above] {} (7.4,-0.3);
\draw (8,-0.3) to [bend right=25] node[pos=0.48, right] {} (7.2,0);

\draw[black] (0,-0.2) ellipse (3.3 and 1);
\draw[->] (0,-1.3) -- (0,-1.7) node[pos=0.48, right] {$\pi$};
\draw[black] (0,-2) ellipse (0.9 and 0.4);
\filldraw [black] (0,-2) circle (1pt);
\draw [red, thick, dotted] (0,-2) circle (4pt);

\draw[red, thick, dotted] (0,-0.2) ellipse (2.9 and 0.6);

\draw[->] (7,-1) -- (7,-1.7) node[pos=0.48, right] {$p$};
\draw [red, thick, dotted] (7,-2) circle (4pt);

\draw[red, thick, dotted] (7,-0.2) ellipse (2.9 and 0.6);
\filldraw [black] (7,-2) circle (1pt);

\draw [left hook->] (6.8,-2) to node[above]  {\scriptsize{$u$}} (0.3,-2);
\draw [left hook->] (4,-0.2) to node[above]  {\scriptsize{$v$}} (3,-0.2);

\node (C1) at (-3,-2)  {$Y$};
\node (C2) at (9,-2)  {$U = \Spec \, R_{r,a}$};
\node (C3) at (-3,1)  {$X_{r,a}$};
\node (C4) at (8,1)  {$V_{r,a}$};

\end{tikzpicture}
\]
As $u$ is flat and affine so is $v$,  and hence $v^*$ and $v_*$ are exact and the counit $v^*v_* \rightarrow id$ is an equivalence. We then define the functor  
\begin{align*}
F:D(X_{r,a}) &\rightarrow D(V_{r,a}) \\
\mathcal{E} & \mapsto v^*(\mathcal{E}) \otimes_{V_{r,a}} \mathcal{O}_{V_{r,a}}(-D_n)
\end{align*}
where $D_n$ is a divisor specified by Lemma \ref{Line Bundle Proposition} such that $D_n \cdot C_j = \delta_{n,j}$.
In particular, this restricts a bounded functor $F:D_{r,a} \rightarrow D^b(V_{r,a})$ and, as $F$ is the composition of the exact pullback $v^*$ with the exact autoequivalence of tensoring by the line bundle $\mathcal{O}_{V_{r,a}}(-D_n)$, it also restricts to an exact functor $\Coh \, X \cap D_{r,a} \rightarrow \Coh \, V_{r,a}$.

We recall the equivalences of abelian categories $D_{r,a} \cap \, \Coh(X_{r,a}) \cong \Lambda_{r,a}$-$\mod$ and $\mathcal{A}_{r,a}={^{0}\Per}(V_{r,a}/ R_{r,a}) \cong A_{r,a}$-$\mod$, and we now now deduce some properties of this functor by calculating the images of the simple, standard, and projective $\Lambda_{r,a}$-modules. Recall from Proposition \ref{Line Bundle Proposition} the distinguished divisors $D_i$ on $V_{r,a}$ and that any line bundle on $V_{r,a}$ is isomorphic to one of the form $\mathcal{O}_{V_{r,a}}( \sum a_i D_i)$.

\begin{Lemma} \label{Functor Calculation Lemma} The functor $F$ restricts to a functor from the abelian category $D_{r,a} \cap\, \Coh(X_{r,a})$ to the abelian category $\mathcal{A}_{r,a}$ that preserves projective objects. We compute the images of the simple, standard, and projective objects:
\begin{enumerate}
\item
$F(\sigma_i) = \left\{ \begin{array}{ll} 
\mathcal{O}_{C_i}(-1) \cong s_i & \text{ if $1 \le i \le n$ } \\
 \mathcal{O}(-C_1-\dots -C_n -D_n)  & \text{ if $i=0$} 
\end{array} \right.$
\item
$F(\mathcal{L}_i) = \left\{ \begin{array}{ll} 
\mathcal{O}(\sum_{j=i+1}^{n} (\alpha_j-2)D_j + D_{i+1} - D_i) & \text{ if $1 \le i \le n$ } \vspace{2pt} \\
 \mathcal{O}(\sum_{j=1}^n(\alpha_j-2)D_j + D_1)  & \text{ if $i=0$} 
\end{array} \right.$
\item
$F(\Lambda_i) = \left\{ \begin{array}{ll} 
\mathcal{O}_V(-D_i) \oplus \mathcal{O}_V^{ \oplus \lambda_i-1}  \cong P_i \oplus \left( P_{0} \right)^{\oplus \lambda_i-1} & \text{ if $1 \le i \le n$}  \\
\mathcal{O}_V^{ \oplus \lambda_{0}} \cong P_{0}^{\oplus \lambda_{0}}  & \text{ if $i=0$} 
\end{array} \right.$
\end{enumerate}
where $\lambda_i$ is defined to be the rank of $\Lambda_i$.
\end{Lemma}
\begin{proof}

We begin by checking that $F$ maps simple $\Lambda_{r,a}$-modules to $\mathcal{A}_{r,a}$.  It is clear from the definition of the functor that the $n$ simple $\Lambda_{r,a}$-modules $\sigma_1=\mathcal{O}_{C_1}(-1), \dots, \sigma_n=\mathcal{O}_{C_n}$ are mapped to the $n$ simple objects $s_1=\mathcal{O}_{C_1}(-1), \dots, s_n=\mathcal{O}_{C_n}(-1)$ in $\mathcal{A}_{r,a}$. This only leaves the simple module $\sigma_{0}=\mathcal{L}_{0} =\mathcal{O}_{X_{r,a}}(-C_1 - \dots - C_n)$, and to check that $F(\mathcal{L}_{0})$ is in $\mathcal{A}_{r,a}$ we verify the conditions of Definition \ref{Perverse Definition}. As $F(\mathcal{L}_0)$ is a sheaf we need only show that $\mathbf{R}^1 p_* F(\mathcal{L}_{0}) =0$. To do this we recall the short exact sequences
\[
0 \rightarrow \mathcal{L}_{i-1} \rightarrow \mathcal{L}_{i} \rightarrow \sigma_i \rightarrow 0
\]
appearing in Proposition \ref{Proposition HillePloog Abelian Category} and as $F$ is exact there is a short exact sequence in $\Coh \, V_{r,a}$
\[
0 \rightarrow F(\mathcal{L}_{i-1}) \rightarrow F(\mathcal{L}_{i}) \rightarrow \mathcal{O}_{C_i}(-1) \rightarrow 0
\]
for $1 \le i \le n$. We recall that $\mathbf{R}p_* \mathcal{O}_{C_i}(-1)=0$ for all $i$ and so by considering the long exact sequence obtained by applying $\mathbf{R}p_*$ we recover the short exact sequence
\[
0 \rightarrow \mathbf{R}^{1} p_*  F(\mathcal{L}_{i-1}) \rightarrow  \mathbf{R}^{1} p_*F(\mathcal{L}_{i})  \rightarrow  0.
\]
Then as $F(\mathcal{L}_{n}) \cong \mathcal{O}_{V_{r,a}}(-D_n)=P_n \in {^0\Per({V_{r,a}}/R_{r,a})}$ by Theorem \ref{Affine Abelian Theorem} it follows that $\mathbf{R}^1 p_* F(\mathcal{L}_{n})=0$ and these sequences imply that $\mathbf{R}^1 p_* F(\mathcal{L}_{i}) \cong 0$ for all $i$. 

As $\Lambda_{r,a}$ is finite dimensional all finitely generated $\Lambda_{r,a}$-modules are finite dimensional. Having shown that simple $\Lambda_{r,a}$-modules are mapped to $\mathcal{A}_{r,a}$ it follows by induction on the dimension of a module that that all finitely generated $\Lambda_{r,a}$-modules are sent to $\mathcal{A}_{r,a}$. 

Now we show that the functor $F$ must send projective objects to projective objects. Suppose that $\Omega$ is a projective $\Lambda_{r,a}$-module. Then $\Ext_{\Lambda_{r,a}}^i(\Omega,M)$ vanishes for all $\Lambda_{r,a}$-modules $M$ and $i \ge 1$, and we know that $F(M) \in \mathcal{A}_{r,a} \cap \Coh V_{r,a}$. We now show that $\Ext^1_{A_{r,a}}(F(\Omega), s) $ vanishes for all simple objects $s$ in $\mathcal{A}_{r,a}\cong A_{r,a}$-$\mod$:
\begin{align*}
\Ext^1_{A_{r,a}}(F(\Omega), s_i) 
&\cong \Hom_{D(V_{r,a})}(v^*\Omega \otimes \mathcal{O}_{V_{r,a}}(-D_n) , \mathcal{O}_{C_i}(-1)[1])\\
&\cong \Hom_{D(V_{r,a})}(v^*\Omega , \mathcal{O}_{C_i}(-1)[1] \otimes \mathcal{O}_{V_{r,a}}(D_n) )\\
&\cong \Hom_{D(X_{r,a})}(\Omega, v_*(\mathcal{O}_{C_i}(-1) \otimes \mathcal{O}_{V_{r,a}}(D_n))[1]) \tag{$v_*$ and $v^*$ adjoint}\\
&\cong \Hom_{D(\Lambda_{r,a})}(\Omega, \sigma_i[1]) \tag{$\mathcal{O}_{C_i}(-1) \otimes \mathcal{O}_{V_{r,a}}(D_n)) \cong \mathcal{O}_{C_i}(D_n \cdot C_i-1)$} \\
&\cong \Ext^1_{\Lambda_{r,a}}(\Omega, \sigma_i)=0
\end{align*}
for $1 \le i \le n$ and
\begin{align*}
\Ext^1_{A_{r,a}}(F(\Omega), s_{0}) &= \Hom_{D(A_{r,a})}(F(\Omega), s_{0}[1]) \\
&= \Hom_{D(V_{r,a})}(F(\Omega), \omega_C[2]) \\
&= \Ext^2_{V_{r,a}}(F(\Omega), \omega_C)
\end{align*}
where $\Ext^2_{V_{r,a}}(F(\Omega), \omega_C)=0$ by \cite[Proposition 5.2.34]{AGandarithmeticcurves} as $\omega_C$ is a coherent sheaf on a projective over affine surface with fibres of dimension $\le 1$ and $F(\Omega)$ is a vector bundle due to  $\Omega$ being a vector bundle by Proposition \ref{Proposition HillePloog Abelian Category}(3). Hence $\Ext_{A_{r,a}}^1(F(\Omega),s_i)=0$ for all the simple $A_{r,a}$-modules $s_i$. Then take a projective cover $P$ of $F(\Omega)$ as an $A_{r,a}$-module (i.e. no summand of $P$ is in the kernel of $P \rightarrow F(\Omega)$), and take the kernel to produce a short exact sequence of $A_{r,a}$-modules
\[
0 \rightarrow K \rightarrow P \rightarrow F(\Omega) \rightarrow 0.
\]
For any simple $A_{r,a}$-module $s$ applying $\Hom_{A_{r,a}}(-,s)$ produces a long exact sequence
\[
0 \rightarrow \Hom_{A_{r,a}}( F(\Omega), s) \rightarrow \Hom_{A_{r,a}}(P,s) \rightarrow \Hom_{A_{r,a}}(K,s) \rightarrow \Ext_{A_{r,a}}^1( F(\Omega),s)=0.
\]
As $P$ is a projective cover it follows that the first map is an isomorphism and hence $\Hom_{A_{r,a}}(K,s)=0$ for any simple module $s$. Hence $K \cong 0$, as any $A_{r,a}$ module must have a map to a simple module, and so $F(\Omega)$ is itself projective.

We now show part (2). As $F(\mathcal{L}_i)=\mathcal{O}_{V_{r,a}}(-C_{i+1} - \dots - C_{n} -D_n)$ is a line bundle by  Proposition \ref{Line Bundle Proposition} we need only calculate the degree of its restriction to the curves to write it in the required form. We make the calculation
\begin{equation*}
\deg F(\mathcal{L}_i) |_{C_j}= -(C_{i+1} + \dots C_{n}+D_n) \cdot C_j = \left\{ \begin{array}{cc} 
\alpha_j-2 & \text{ if $i+1  <  j \le n$} \\
\alpha_{i+1}-1 & \text{ if $j=i+1$} \\
-1& \text{ if $j = i$} \\
0& \text{ otherwise} 
\end{array}\right.
\end{equation*}
The data of $F(\mathcal{L}_i)$ written in the form $\mathcal{O}_{V_{r,a}}(\sum a_j D_j)$ is then summarized by putting the coefficients $a_j$ into the following $(n + 1) \times n$ matrix
\begin{equation*}
\begin{array}{l | c c c c c c c c}
 & D_1 & D_2 & \cdots & D_{i-1} & D_i &\cdots & D_{n-1} &D_n\\
\hline
F(\mathcal{L}_{0})& \alpha_1-1 & \alpha_2-2& \cdots & \alpha_{i-1}-2 & \alpha_i-2 & \cdots & \alpha_{n-1}-2 &\alpha_n-2 \\
F(\mathcal{L}_1) & -1 & \alpha_2-1& \cdots & \alpha_{i-1}-2 & \alpha_i-2 & \cdots &\alpha_{n-1}-2 & \alpha_n-2 \\
\vdots& \ddots&\ddots&\ddots&\ddots&\ddots&\ddots&\ddots& \\
F(\mathcal{L}_{i-1})  &0&0& \cdots &-1 & \alpha_i-1 & \cdots & \alpha_{n-1}-2 &\alpha_n-2 \\
F(\mathcal{L}_i) & 0 &0& \cdots &0& -1 & \cdots & \alpha_{n-1}-2 &\alpha_n-2\\
\vdots& \ddots&\ddots&\ddots&\ddots&\ddots&\ddots&\ddots&\\
F(\mathcal{L}_{n-1}) &0& 0& \cdots & 0 &0& \cdots &-1 & \alpha_n-1 \\
F(\mathcal{L}_{n})& 0 & 0& \cdots & 0 & 0 & \cdots& 0 & -1 \\
\end{array}
\end{equation*}
which calculates (2). Recall the short exact sequences 
\begin{equation*}
 0  \rightarrow  \bigoplus_{j=i}^n \Lambda_j^{\oplus \alpha_j-2} \oplus \Lambda_i  \rightarrow \Lambda_{i-1}  \rightarrow  \mathcal{L}_{i-1}  \rightarrow  0.
\end{equation*}
of Proposition \ref{Proposition HillePloog Abelian Category}. Combining these short exact sequences with part (2) we can inductively calculate $F(\Lambda_i)$ to deduce (3). We first note that $F(\Lambda_i)$ is a projective object in $\mathcal{A}_{r,a}$ and as such $F(\Lambda_i)$ is determined by the rank and first Chern class of $F(\Lambda_i)$ due to Theorem \ref{Affine Abelian Theorem}. In particular, $\Lambda_n=\mathcal{O}_{X_{r,a}}$ so $F(\Lambda_n)= \mathcal{O}(-D_n)$. As first Chern characters are additive in short exact sequences it follows that
\begin{equation*}
\Chern( F(\Lambda_{i-1}) ) =  \sum_{j=i}^n (\alpha_j-2)\Chern( F(\Lambda_j)) + \Chern(F(\Lambda_{i})) + \Chern(F(\mathcal{L}_{i-1})).
\end{equation*}
so by induction on $j$ it follows that $\Chern( F(\Lambda_j ))= \mathcal{O}(-D_j)$. Then the result follows by the identification of projectives by rank and first Chern class in Theorem \ref{Affine Abelian Theorem}.
\end{proof}

\subsection{Equivalence of relative singularity categories}
These calculations give us a hint that, in some sense, $D^b(\Lambda_{r,a}) \cong D_{r,a}$ and $D^b(A_{r,a}) \cong D^b(V_{r,a})$ are very similar away from the distinguished modules corresponding to $i=0$.  We will make this idea precise by showing that the functor $F$ descends to an equivalence of relative singularity categories
\[
 \frac{D_{r,a}}{\langle \Lambda_0 \rangle} 
\quad \text{ and } \quad
 \frac{D^b(V_{r,a})}{\langle \mathcal{O}_{V_{r,a}}  \rangle}  .
\]
We first show that $F$ induces an equivalence between objects supported on the curves.

\begin{Proposition} \label{Equivalence of objects supported on curve Proposition}
The functor $F$ induces an equivalence of triangulated categories
\[
F:D_C( X_{r,a}) \rightarrow D_C(V_{r,a}),
\]
where $D_C(-)$ denotes the full subcategory of unbounded derived category of quasicoherent sheaves whose cohomology sheaves are supported on $C= \cup C_i$.
\end{Proposition}

\begin{proof}
On the unbounded level the functor  $F:D(X_{r,a}) \rightarrow D(V_{r,a})$ has a right adjoint defined by
\begin{align*}
G: D(V_{r,a}) &\rightarrow D(X_{r,a}) \\
\mathcal{E} &\mapsto v_*(\mathcal{E} \otimes \mathcal{O}_V(D_n)).
\end{align*}
The counit $FG \rightarrow id $ is a natural isomorphism on the whole of $D(V_{r,a})$ as the counit $v^*v_* \rightarrow id$ is a natural isomorphism and $\mathcal{O}_{V_{r,a}}(D_n) \otimes \mathcal{O}_{V_{r,a}}(-D_n) \cong \mathcal{O}_{V_{r,a}}$. 

Further,  as $v$ is flat and affine and $\mathcal{O}_{V_{r,a}}(-D_n)$ and $\mathcal{O}_{V_{r,a}}(D_n)$ are line bundles the functors $F= v^*(-) \otimes \mathcal{O}_{V_{r,a}}(-D_n)$ and $G=v_*((-) \otimes \mathcal{O}_{V_{r,a}}(D_n))$ commute with taking cohomology sheaves,  restrict to exact functors between $\QCoh \, X_{r,a}$ and $\QCoh \, V_{r,a}$, and preserve the property of a sheaf being supported on $C$. It follows that $F$ and $G$ restrict to functors between $D_C(X_{r,a})$ and $D_C(V_{r,a})$. 

 Moreover, if $i:C \rightarrow V_{r,a}$ denotes the closed immersion of the curve in $V_{r,a}$ then $(v \circ i):C \rightarrow X_{r,a}$ is the closed immersion of the curve in $X_{r,a}$. Indeed, by construction $V_{r,a}$ contains the formal neighbourhood of $C$ and hence the morphism $v$ induces an isomorphism between the completions of $X_{r,a}$ and $V_{r,a}$ along $C$, and in particular produces an equivalence between the categories of quasicoherent sheaves on the corresponding formal schemes. It then follows that $F$ and $G$ induce an equivalence
 \[
 \textnormal{QCoh}_C X_{r,a} \cong \textnormal{QCoh}_C V_{r,a}
 \]
between the categories of quasicoherent sheaves supported on $C$, see \cite[Proposition 2.8]{Orlovidempotent}. We recall that $D_C(X_{r,a}) \cong D( \textnormal{QCoh}_C X_{r,a})$ and $D_C(V_{r,a}) \cong D( \textnormal{QCoh}_C V_{r,a})$, see \cite[Section 2]{Orlovidempotent}, and hence it follows that the functors $F$ and $G$ are an equivalence between $D_C(V_{r,a})$ and $D_C(X_{r,a})$.
\end{proof}

\begin{Theorem} \label{Equivalence of Geometric Relative Singularity Categories Theorem}
 The functor $F$ restricts to  induce an equivalence of quotient categories
\[
\bar{F}: \frac{D^b(\Lambda_{r,a})}{\langle \Lambda_0 \rangle} \rightarrow \frac{D^b(V_{r,a})}{\langle\mathcal{O}_{V_{r,a}} \rangle}.
\]
\end{Theorem}

\begin{proof}
To simplify notation we omit the $r,a$ subscripts from $V_{r,a}$, $\Lambda_{r,a}$, and $A_{r,a}$ for the remainder of this proof.

We work with the unbounded derived functor $F:D(\Lambda) \rightarrow D(V)$. This can be composed with the Verdier quotient $D(V) \rightarrow D(V)/\langle \mathcal{O}_{V} \rangle^{\oplus}$. Then $F$ maps the projective module $\Lambda_0$  into $\langle \mathcal{O}_{V}\rangle$ by Lemma \ref{Functor Calculation Lemma}, and hence it follows from  the universal property of Verdier localisation that the functor $F$ restricts to a functor $\bar{F}:D(\Lambda)/\langle \Lambda_0\rangle^{\oplus} \rightarrow D(V)/\langle \mathcal{O}_{V} \rangle^{\oplus} $ as $F$ commutes with arbitrary direct sums.

We now show that the induced functor $\bar{F}$ is fully faithful. Recall from Proposition \ref{Second Isomorphism Theorem Proposition} that \[
\frac{D(\Lambda)}{\langle \Lambda_0 \rangle^{\oplus} } \cong  D_{\Lambda/\Lambda e \Lambda}(\Lambda)\, \text{ and } \, \frac{D(V)}{ \langle \mathcal{O}_{V} \rangle^{\oplus}} \cong \frac{D(A)}{ \langle P_0 \rangle^{\oplus}} \cong  D_{A/AeA}(A) 
\] 
where the equivalence $\frac{D(V)}{ \langle \mathcal{O}_{V} \rangle^{\oplus}} \cong \frac{D(A)}{ \langle P_0 \rangle^{\oplus}}$ is the tilting equivalence of Theorem \ref{Affine Abelian Theorem}, and to ease notation we will also use $F$ and $\bar{F}$ to refer to the compositions of $F$ and $\bar{F}$ with this equivalence.

The functor $F:D_C(X) \rightarrow D_C(V)$ is fully faithful by Proposition \ref{Equivalence of objects supported on curve Proposition}, and hence the induced functor $F:D_C(\Lambda) \rightarrow D(V)$ is fully faithful. As the simples $\sigma_i$ of $\Lambda$ are supported on $C$ for $1 \le i \le n$ it follows that $D_{\Lambda/\Lambda e \Lambda}(\Lambda)\subset D_C(X)$, and if $M \in D_{\Lambda/\Lambda e \Lambda}(\Lambda)$, then
\begin{align*}
\Hom_{D(A)}(P_0^{\oplus \lambda_0},F(M)) 
&= \Hom_{D(A)}(F(\Lambda_0),F(M)) \tag{$P_0^{\oplus \lambda_0}\cong F(\Lambda_0)$} \\
&= \Hom_{D(\Lambda)}(\Lambda_0,GF(M)) \tag{$(F,G)$ adjoint pair}\\
&= \Hom_{D(\Lambda)}(\Lambda_0,M) \tag{$GF(M) \cong M$}
\end{align*}
where $GF(M) \cong M$ by Proposition \ref{Equivalence of objects supported on curve Proposition} as $M \in D_{\Lambda/\Lambda e \Lambda}(\Lambda)\subset D_C(X)$. It follows that $F:D_C(\Lambda) \rightarrow D(V)$ maps the kernel $D_{\Lambda/\Lambda e \Lambda}(\Lambda)$ of the functor $\Hom_{D(\Lambda)}(\Lambda_0,-)$ to the kernel $\mathcal{K} \cong D_{A/AeA}(A)$ of the functor $\Hom_{D(V)}(\mathcal{O}_V,-)$ so restricts to a fully faithful functor
$F_{\textnormal{res}}:D_{\Lambda/\Lambda e \Lambda}(\Lambda) \rightarrow D_{A/AeA}(A)$.

That is, there is a diagram
\[
\begin{tikzpicture} 
\node (A1) at (0,0)  {$D_C(\Lambda)$};
\node (A2) at (1.5,-2)  {$D(\Lambda)/\langle \Lambda_0 \rangle^{\oplus}$};
\node (A3) at (0,-4) {$D_{\Lambda/\Lambda e \Lambda}(\Lambda)$};
\node (B1) at (9.6,0) {$D(V) $};
\node (B1') at (10.3,0) {$\cong $};
\node (B1'') at (11,0) {$D(A)$};
\node (B2) at (6.7,-2) {$D(V)/\langle \mathcal{O}_{V} \rangle^{\oplus}$};
\node (B2') at (8.2,-2) {$\cong$};
\node (B2'') at (9.5,-2) {$D(A)/\langle P_0 \rangle^{\oplus}$};
\node (B3) at (9.4,-4) {$\mathcal{K}$};
\node (B3') at (9.8,-4) {$\cong$};
\node (B3'') at (11,-4) {$D_{A/A e A}(A)$};
\draw [->] (A1) to node[above right] {$\scriptstyle{p_{\Lambda}}$} (A2);
\draw [->] (A2) to node[below right] {$\scriptstyle{\theta_{\Lambda}}$} (A3);
\draw [->] (A3) to node[left] {$\scriptstyle{i_{\Lambda}}$} (A1);
\draw [->](B1'') to node[above left] {$\scriptstyle{p_{A}}$} (B2'');
\draw [->] (B2'') to node[below left] {$\scriptstyle{\theta_{A}}$} (B3'');
\draw [->] (B3'') to node[right] {$\scriptstyle{i_A}$} (B1'');
\draw [->] (A1) to node[above] {$\scriptstyle{F}$} (B1);
\draw [->] (A2) to node[above] {$\scriptstyle{\bar{F}}$} (B2);
\draw [->] (A3) to node[above] {$\scriptstyle{F_{\textnormal{res}}}$} (B3); (B1);
\end{tikzpicture}
\]
where the functors $i_{\Lambda}$ and $i_{A}$ are the fully faithful inclusions and the functors $\theta_{\Lambda}$ and $\theta_{A}$ are the equivalences of Proposition \ref{Second Isomorphism Theorem Proposition}. The top and outer squares commute as
\[
\bar{F} \circ p_{\Lambda} \cong p_{A} \circ F \, \text{ and } \, i_{A} \circ F_{\textnormal{res}}  \cong F \circ i_{\Lambda}
\]
by the definitions of $\bar{F}$ and $F_{\textnormal{res}}$. Then $\theta_\Lambda \circ p_\Lambda \circ i_\Lambda \cong id$ and $\theta_A \circ p_A \circ i_A \cong id$ by the definitions of $\theta_\Lambda$ and $\theta_A$ in Proposition \ref{Second Isomorphism Theorem Proposition}; that is, in the recollement situation of Section \ref{Recollemonts generated by idempotents}, the $\theta$ equivalences are induced from the $i^*$ functors, the $i$ functors are the $i_*$ functors, and $i^* \circ i_* \cong id$. This allows us to show the bottom square also commutes:
\begin{align*}
\theta_{A} \circ \bar{F} \circ \theta_{\Lambda}^{-1} & \cong \theta_{A} \circ \bar{F} \circ p_{\Lambda} \circ i_{\Lambda} \tag{$\theta_\Lambda^{-1} \cong p_{\Lambda} \circ i_{\Lambda}$} \\
& \cong \theta_{A} \circ p_A \circ F \circ i_{\Lambda}  \tag{$ \bar{F} \circ p_{\Lambda} \cong p_A \circ F$} \\
& \cong \theta_{A} \circ p_A \circ i_A \circ F_{\textnormal{res}}  \tag{$F \circ i_{\Lambda} \cong i_A \circ F_{\textnormal{res}}$}\\
& \cong F_{\textnormal{res}}. \tag{$\theta_A \circ p_A \circ i_A \cong id$}
\end{align*}
Hence if $N,M \in D(\Lambda)/\langle \Lambda_0 \rangle^{\oplus}$, then 
\begin{align*}
\Hom_{D(A)}(\bar{F}(N),\bar{F}(M)) & \cong  \Hom_{D(A)}(i_A \theta_A \bar{F}(N),i_A \theta_A \bar{F}(M))  \tag{$i_A \circ \theta_A$ fully faithful} \\
& \cong  \Hom_{D(A)}(F i_\Lambda \theta_\Lambda (N), F i_\Lambda \theta_\Lambda (M)) \tag{diagram commutes} \\
& \cong  \Hom_{D(\Lambda)}(N, M)  \tag{$i_A \circ \theta_A \circ F$ fully faithful} 
\end{align*}
so $\bar{F}$ is fully faithful. 

The functor $\bar{F}$ is essentially surjective as $P_0 \cong \mathcal{O}_V$ is identified with 0 in the quotient category so $\bar{F}(\Lambda_i) \cong P_i$ for $1 \le i \le n$ by Lemma \ref{Functor Calculation Lemma}. Hence $\bar{F}$ maps the generator $\bigoplus_{i=1}^n \Lambda_i$ to the generator $\bigoplus_{i=1}^n P_i$.

So the functor $\bar{F}$ is fully faithful and essentially surjective, hence it is an equivalence. Restricting to the subcategories of compact objects we get an equivalence $\bar{F}:\left( \frac{D^b(\Lambda)}{\langle \Lambda_0 \rangle} \right)^{\omega} \rightarrow \left( \frac{D^b(V)}{\langle\mathcal{O}_{V} \rangle}\right)^{\omega}$, by Neeman's \cite[Theorem 2.1]{NeemanThomasonTrobaughYao}, where $(-)^\omega$ denotes the idempotent completion. Both quotient categories $\frac{D^b(\Lambda)}{\langle \Lambda_0 \rangle}$ and $\frac{D^b(V)}{\langle\mathcal{O}_V \rangle}$ are already idempotent complete by Corollary \ref{Idempotent completeness}. This completes the proof. 
\end{proof}

This induces an equivalence of relative singularity categories.
\begin{Theorem} \label{Equivalence of Relative Singularity Categories Theorem}
The functor $\bar{F}$ induces an equivalence 
\[
F_{A}\colon \Delta_{K_{r,a}}(\Lambda_{r,a}) \rightarrow \Delta_{R_{r,a}}(A_{r,a})
\]
with $F_{A}(\sigma_i)=s_i$ and $F_{A}(\Lambda_i) \cong P_i$ for $1 \le i \le n$.
\end{Theorem}
\begin{proof}
We recall that there is a derived equivalence $D^b(A_{r,a}) \cong D^b(V_{r,a}) $ that exchanges the object $P_0 \in D^b(A_{r,a})$ with the structure sheaf $\mathcal{O}_{V_{r,a}} \in D^b(V_{r,a})$. As such, this derived equivalence restricts to the corresponding quotient categories and can be composed with the functor $\bar{F}$, which is an equivalence by  Theorem \ref{Equivalence of Geometric Relative Singularity Categories Theorem}, to produce an equivalence of the relative singularity categories of the algebras:
\[
\begin{tikzpicture} 
\node (A1) at (0,1.4)  {$\Delta_{K_{r,a}}(\Lambda_{r,a})$};
\node (B2) at (7,1.4)  {$ D^b(A_{r,a})/\langle P_0 \rangle$};
\node (C2) at (5.4,1.4) {$:=$};
\node (A2) at (4,1.4)  {$\Delta_{R_{r,a}}(A_{r,a})$};
\node (A3) at (7,0)  {$ D^b(V_{r,a})/\langle \mathcal{O}_{V_{r,a}} \rangle$};
\node (A4) at (0,0)  {$ D^b(\Lambda_{r,a})/\langle \Lambda_{0} \rangle$};
\node (C1) at (0,0.7) [above,rotate=270] {$:=$};
\draw [->] (A1) to node[above right] {$\scriptstyle{F_{A}}$} (A2);
\node (C2) at (7,0.7) [above,rotate=270] {$\cong$};
\draw [->] (A4) to node[above right] {$\scriptstyle{\bar{F}}$} (A3);
\end{tikzpicture}.
\]

By Lemma \ref{Functor Calculation Lemma} $\bar{F}(\sigma_i)=s_i$, and in the quotient category $P_0 \cong \mathcal{O}_V$ is identified with $0$ so by Lemma \ref{Functor Calculation Lemma} $\bar{F}(\Lambda_i) \cong P_i$ for $1 \le i \le n$. 
\end{proof}

We obtain the following generalisation to partial resolutions of singularities.

\begin{Corollary}\label{C:GeneralizedEquivofRelSingCats}
Let $e \in \Lambda_{r, a}$ be an idempotent such that the indecomposable projective $P_0$ is a direct summand of  $\Lambda_{r, a} e$ and let $e \in A_{r, a}$ the corresponding idempotent. Then there is an equivalence of  triangulated categories
\[
F_{A, e} \colon \Delta_{e\Lambda_{r, a} e}(\Lambda_{r,a}) \rightarrow \Delta_{e A_{r, a} e}(A_{r,a})
\]
with $F_{A, e}(\sigma_i)=s_i$ and $F_{A, e}(\Lambda_i) \cong P_i$ for all $1 \le i \le n$ such that $P_i$ is not a direct summand of $\Lambda_{r, a} e$.
\end{Corollary}
\begin{proof}
By definition, $\Delta_{e\Lambda_{r, a} e}(\Lambda_{r,a})=\frac{D^b(\Lambda_{r, a})}{\langle \Lambda_{r, a} e \rangle}$. Our assumption $\Lambda_{r, a} e= P_0 \oplus \Lambda_{r, a} e'$ shows \[\Delta_{e\Lambda_{r, a} e}(\Lambda_{r,a})=\frac{D^b(\Lambda_{r, a})}{\langle \Lambda_{r, a} e \rangle} \cong \frac{D^b(\Lambda_{r, a})/\langle P_0 \rangle}{\langle \Lambda_{r, a} e' \rangle} \cong \frac{\Delta_{K_{r,a}}(\Lambda_{r,a})}{\langle \Lambda_{r, a} e' \rangle}.\] We have corresponding statements for $A_{r, a}$. Applying Theorem \ref{Equivalence of Relative Singularity Categories Theorem} completes the proof.
\end{proof}

\subsection{Equivalence of singularity categories}
The equivalences of relative singularity categories can be used to deduce an equivalence of singularity categories. We now define the Kn\"{o}rrer invariant algebra. (A presentation of $K_{r,a}$ is calculated later in Lemma \ref{KI Presentation Lemma}.)
\begin{Definition} \label{KIdefinition}
The \emph{Kn\"{o}rrer invariant algebra} is defined to be $K_{r,a}:= \End_{X_{r,a}}(\Lambda_0)$. If $e_0 \in \Lambda_{r,a}$ is the idempotent corresponding to the simple $\sigma_0$ then  $K_{r,a} \cong  e_{0} \Lambda_{r,a} e_{0}$.
\end{Definition}

The following  Theorem is the main result of this  article.

\begin{Theorem} \label{SC for KI} There are  equivalences of singularity categories
\[
D_{sg}(K_{r,a}) \cong D_{sg}(R_{r,a}).
\]
\end{Theorem}

Below we prove the more general result Theorem \ref{SC for KI generalized}, from which Theorem \ref{SC for KI} follows by setting $e=e_0$. Before we do this we must first introduce some notation. 

We recall the notation $\Lambda_{r,a}$ for the algebra defined in section \ref{Subcategory}, but we will often instead adjoin the subscript $\Lambda_{[\alpha_1, \dots, \alpha_n]}$ where $r/a=[ \alpha_1, \dots, \alpha_n]$ to emphasize that the algebra $\Lambda_{r,a}$ corresponds to geometry of the curves $C_i$ for $1 \le i \le n$ with self intersection numbers $C_i \cdot C_i=-\alpha_i$ or $\Lambda_{[]}=\mathbb{C}$ for the empty collection of curves; for example, to simplify the notation for arguments that induct on the number of curves.  We'll do the same for $R_{r, a}$, $V_{r, a}$ and so on.

The following Lemma follows from the  construction  of the algebra $\Lambda_{r, a}$.

\begin{Lemma}\label{restrict} \label{LessCurvesLemma}
Let $1 \leq j \leq n+1$ and set $e=\sum_{i=j-1}^n e_{i} \in \Lambda_{r, a}$. Then there are isomorphisms of algebras
\[
e \Lambda_{r, a} e \cong \Lambda_{[\alpha_j, \dots, \alpha_n]} 
\]
and
\[
\frac{\Lambda_{r,a}}{\Lambda_{r,a}  e \Lambda_{r,a}} \cong \Lambda_{[\alpha_{1}, \dots, \alpha_{j-2}]}.
\]
\end{Lemma}
\begin{proof}
By definition (see \ref{Proposition HillePloog Abelian Category}) of  $\Lambda_{r, a}$ and $e$, we see that  
$e \Lambda_{r, a} e$ is isomorphic to the endomorphism algebra of the direct sum of the projective objects $\Lambda_{j-1}, \ldots, \Lambda_n$ in $D_{r, a} \cap \Coh \, X_{r, a}$. By construction the $\Lambda_i$ are universal extensions of the line bundles  \[\mathcal{O}_X, \mathcal{O}_X(-C_n), \ldots, \mathcal{O}_X(-C_j-\cdots-C_n)\] defined using only the curves $C_j, \ldots, C_n$. By definition $\Lambda_{[\alpha_j, \ldots, \alpha_n]}$ is  the endomorphism algebra of the very same direct sum of projective objects. 

The second part corresponds to killing all maps in  $\Lambda_{r,a} \cong \End(\oplus_{i=0}^{n} \Lambda_{i})$ that factor through $\Lambda_{j-1}, \dots, \Lambda_n$. It can be seen from the construction in \ref{Proposition HillePloog Abelian Category}(3) that killing these maps has the same effect as forgetting curves $C_{j-1}, \dots, C_{n}$ in $X$ to produce an algebra $\Lambda_{[\alpha_{1}, \dots, \alpha_{j-2}]} $.  
\end{proof}

\begin{Remark}
The analogous statement for $A_{r, a}$ does not hold.  Indeed, consider for example 
$A_{3, 2}$ 
and the idempotent $e=e_1 + e_2$. Then $e A_{3, 2} e$ has infinite global dimension, see e.g. \cite{KIWY}. In particular, it is not of the form $A_{r, a}$ which always has finite global dimension. Also for any non-zero idempotent $e$ the quotient $A_{r, r-1}/A_{r, r-1}e A_{r, r-1}$ is a preprojective algebra of Dynkin type and therefore finite dimensional, whereas $A_{r, a}$ is always infinite dimensional.
\end{Remark}

The following result is a more general version of Theorem \ref{SC for KI}.

\begin{Theorem} \label{SC for KI generalized}
Let $e \in \Lambda_{r, a}$ be an idempotent such that the indecomposable projective $P_0$ is a direct summand of  $\Lambda_{r, a} e$ and let $e \in A_{r, a}$ be the corresponding idempotent.  Then there is a triangle equivalence
\[
D_{sg}(e\Lambda_{r, a}e) \cong D_{sg}(e A_{r, a} e).
\]
\end{Theorem}
\begin{proof}
We can write $e=e_0 + e_{i_1} + \ldots + e_{i_k}$ with $0 < i_1 < \ldots < i_k \leq n$.
Set $A=A_{r, a}$. Since $e_0$ appears in $e$ and $R_{r, a}$ is an isolated singularity, it follows that $A/A e A$ is finite dimensional \cite{AuslanderIsolatedSingularitiesandExistenceofalmostsplitsequences} and therefore the thick subcategory $ \langle \mod-A/A e A \rangle$ is generated by the simple $A/A e A$ modules. These are precisely the simple $A$-modules $s_i$ with  $1 \leq i \leq n$ and $i \not= i_j$ for all $j$. Now Lemma \ref{Quotient from relative to singularity category Lemma}  shows
\[
\frac{\Delta_{eAe}(A)}{\langle s_i \mid i \not= i_j \rangle }  \xrightarrow{\cong} D_{sg}(eAe).
\]
and similarly
\[
\frac{\Delta_{e\Lambda e}(\Lambda)}{\langle \sigma_i \mid i \not= i_j \rangle } \xrightarrow{\cong}  D_{sg}(e\Lambda e).
\]

By Corollary \ref{C:GeneralizedEquivofRelSingCats}, the equivalence \[F_{A, e}\colon  \Delta_{e\Lambda e}(\Lambda) \to \Delta_{eAe}(A)  \] identifies the subcategories $\langle \sigma_i \mid i \not= i_j \rangle$ and $\langle s_i \mid i \not= i_j \rangle$, so induces an equivalence of quotient categories
\[
D_{sg}(e\Lambda e) \cong \frac{\Delta_{e\Lambda e}(\Lambda)}{\langle \sigma_i \mid i \not= i_j \rangle } \cong\frac{\Delta_{eAe}(A)}{\langle s_i \mid i \not= i_j \rangle } \cong D_{sg}(eAe).
\]
\end{proof}

\begin{Remark} These singularity categories can be described in more detail using geometry and a result of Orlov (cf. \cite{KIWY}). Indeed, in the notation of the proof of Theorem \ref{SC for KI generalized}, let $V^{e}_{r, a}$ be obtained from the minimal resolution of $\Spec(R_{r, a})$ by contracting all exceptional curves $E_i$ for $i \in \{1, \dots, n\} \setminus \{ i_1, \dots, i_k \}$. Here $i_1<i_2< \dots  <i_k$ are defined by $e=e_0 + \sum_{j=1}^k e_{i_j}$. In other words, we contract chains of curves \[
E_{i_j+1}, E_{i_j+2}, \dots E_{i_{j+1}-2}, E_{i_{j+1}-1}
\] for $0 \le j \le k$ where we let $i_0=0$ and $i_{k+1}=n+1$ for the sake of indexing. 

In particular, $V^e_{r, a}$ has an isolated singularity with completion the cyclic quotient singularity
$R_{\widehat{\bm{\alpha}}_j}$
where
\[
\widehat{\bm{\alpha}}_j:=[\alpha_{i_j+1}, \alpha_{i_1+2}, \ldots, \alpha_{i_{j+1}-1}].
\]
(We note that if $i_j+1=i_{j+1}$ then $\widehat{\bm{\alpha}}_j=[]$ corresponds to contracting no curves and $ R_{[]}=\mathbb{C}[[x,y]]$ is smooth.)

Summarising, there is a chain of triangle equivalences
\begin{equation} \label{E:Chain}
D_{sg}(e\Lambda_{r,a} e) \cong D_{sg}(e A_{r, a} e) \cong  D_{sg}(V^e_{r, a}) \cong
 \bigoplus_{j=0}^{k}  D_{sg}(R_{\widehat{\bm{\alpha}}_j}) .
\end{equation}
The first equivalence is Theorem \ref{SC for KI generalized}, the second equivalence follows from \cite[Theorem 4.6]{KIWY} and the last equivalence is due to Orlov \cite{Orlovidempotent} using that $D_{sg}(e\Lambda_{r,a} e)$ is idempotent complete by  Proposition \ref{P:DsgFDidempcomplete}.
\end{Remark}

\begin{Remark}
Theorem \ref{SC for KI generalized} can be used to describe $D_{sg}(e \Lambda_{r, a}e)$ for  an arbitrary idempotent $e$.
To see this write $e=e_{i_1} + \ldots + e_{i_k}$ with $0 \leq i_1 < \ldots < i_k \leq n$.
 By Lemma \ref{restrict}, we have an algebra isomorphism \[
(e_{i_1} + \ldots + e_n) \Lambda_{r,a} (e_{i_1} + \ldots + e_n)  \cong \Lambda_{[\alpha_{i_1+1}, \alpha_{i_1+2}, \ldots, \alpha_n]}.
\]  By definition of the idempotent $e$, this  gives an  algebra  isomorphism
\[
e \Lambda_{r,a}e \cong e'\Lambda_{[\alpha_{i_1+1}, \alpha_{i_1+2}, \ldots, \alpha_n]} e'.
\]
where $e'=e_0 + e_{i_2 - i_1} + \ldots + e_{i_k - i_1}$ as an element in $\Lambda_{[\alpha_{i_1+1}, \alpha_{i_1+2}, \ldots, \alpha_n]}$. Now we can apply Theorem \ref{SC for KI generalized} to $e' \Lambda_{[\alpha_{i_1+1}, \alpha_{i_1+2}, \ldots, \alpha_n]} e'$.
\end{Remark}

Theorem \ref{SC for KI generalized} yields many non-trivial equivalences between singularity categories of finite dimensional algebras -- already the Gorenstein case (answering a question of Michael Wemyss) seems to be new. We give explicit examples below.

\begin{Corollary} \label{cor:idempotent singularity equivalences}
Let $\bm{\alpha}:=[\alpha_1, \ldots, \alpha_n]$ and  $\bm{\alpha'}:=[\alpha_1', \ldots, \alpha_m']$ be sequences of integers $\geq 2$. Let $0 < i_1 < \ldots < i_k \leq n$ and $0 < j_1 < \ldots < j_l \leq m$  be integers. Remove the elements $\alpha_{i_1}, \alpha_{i_2} \ldots, \alpha_{i_k}$ from $\bm{\alpha}$ to obtain a new sequence $\bm{\gamma}$ and produce $\bm{\gamma}'$ similarly from $\bm{\alpha'}$ using the sequence $(j_a)_{a=1}^l$.

If $\bm{\gamma}=\bm{\gamma}'$, then there is a triangle equivalence
\[D_{sg}(e_{\bm{i}} \Lambda_{\bm{\alpha}} e_{\bm{i}}) \cong D_{sg}(e_{\bm{j}} \Lambda_{\bm{\alpha}'} e_{\bm{j}}) \]
where $e_{\bm{i}}:=e_0+e_{i_1} + e_{i_2} + \ldots +e_{i_k}$ and $e_{\bm{j}}:=e_0+e_{j_1} + e_{j_2} + \ldots +e_{j_l}$.
\end{Corollary}
\begin{proof}
Use the chain of equivalences \eqref{E:Chain} twice.
\end{proof}

\begin{Example}
Let $\bm{\alpha}:=[2, 2]$ and $\bm{\alpha'}:=[2]$.
Let $e_{\bm{i}}:=e_{0} + e_{1}$ and $e_{\bm{j}}:=e_{0}$. In particular,  $\bm{\gamma}=[2]=\bm{\gamma'}$, showing \[D_{sg}(e_{\bm{i}} \Lambda_{[2, 2]} e_{\bm{i}}) \cong D_{sg}(e_{\bm{j}} \Lambda_{[2]} e_{\bm{j}}), \]
where (using Theorem \ref{Endomorphism Algebra Isomorphism Theorem}) 
\begin{align*}
e_{\bm{i}} \Lambda_{[2, 2]} e_{\bm{i}} \cong \End_{\mathbb{C}[x]/(x^3)} (\mathbb{C}[x]/(x^3) \oplus \mathbb{C}[x]/(x^2)) \cong 
\frac{\mathbb{C}\left[\begin{tikzpicture}[baseline={([yshift=-.5ex]current bounding box.center)}]
\node (Cn) at (1.5,0) {$1$};
\node (Cn-1) at (3,0)  {$2$};
\draw [->,bend left=15] (Cn) to node[above] {$\scriptstyle{p}$} (Cn-1);
\draw [->,bend left=15] (Cn-1) to node[below] {$\scriptstyle{i}$} (Cn);
\end{tikzpicture}\right]}{(ipip)}
\end{align*}
and  
\[
e_{\bm{j}} \Lambda_{[2]} e_{\bm{j}} \cong \mathbb{C}[x]/(x^2).
\]
Alternatively, taking $e_{\bm{i}}:=e_{0} + e_{2}$ and keeping $e_{\bm{j}}$ also yields $\bm{\gamma}=[2]=\bm{\gamma'}$ and therefore we get a triangle equivalence
\[D_{sg}(e_{\bm{i}} \Lambda_{[2, 2]} e_{\bm{i}}) \cong D_{sg}(e_{\bm{j}} \Lambda_{[2]} e_{\bm{j}}), \] where now 
\[
e_{\bm{i}} \Lambda_{[2, 2]} e_{\bm{i}} \cong \End_{\mathbb{C}[x]/(x^3)} (\mathbb{C}[x]/(x^3) \oplus \mathbb{C}[x]/(x)) \cong \frac{\mathbb{C}\left[\begin{tikzpicture}[baseline={([yshift=-.5ex]current bounding box.center)}]
\node (Cn) at (1.5,0) {$1$};
\node (Cn-1) at (3,0)  {$2$};
\draw [->,bend left=15] (Cn) to node[above] {$\scriptstyle{\beta}$} (Cn-1);
\draw [->,bend left=15] (Cn-1) to node[below] {$\scriptstyle{\gamma}$} (Cn);
 \draw[->] ($(Cn) +
    (-1.3mm, -1.2 mm)$) arc (-30:-330:2.5mm);
   \node at ($(Cn) + (-8 mm, 0mm)$) {$\scriptstyle{\alpha}$};

\end{tikzpicture}\right]}{(\alpha^2-\beta\gamma, \gamma \alpha, \alpha \beta, \beta \gamma)}.
\]

\end{Example}

\section{Obstructions to generalisations for non-abelian quotient singularities} \label{Sec: Obstructions}
It is a natural question whether a result analogous to Theorem \ref{SC for KI} holds for non-abelian quotient surface singularities. The analysis of the Grothendieck group in this section provides an obstruction in many cases. We learned the following result from a discussion with Michael Wemyss and Xiao-Wu Chen.
\begin{Theorem} \label{GrothendieckClassgroup}
Let $(R, \mathfrak{m})$ be an integrally closed complete local domain of Krull dimension two with an algebraically closed residue field $R/\mathfrak{m}$. Then there is an isomorphism of groups
\[
K_0(D_{sg}(R)) \cong Cl(R)
\]
where $Cl(R)$ denotes the ideal class group of $R$.
\end{Theorem}
\begin{proof}
Combine \cite[Corollary 3.9 (2)]{Beligiannishomologicaltheoryofhomologicallyfinitesubcategories} with  \cite[Proposition 3.2.4]{AuslanderReitenGrothendieckgroupsofalgebras}.
\end{proof}

The following result is well-known to experts.
\begin{Proposition}\label{P:GrothendieckLocal}
Let $A$ be a finite dimensional local $k$-algebra over an algebraically closed field. Then there is an isomorphism of groups
\[
K_0(D_{sg}(A)) \cong \mathbb{Z}/(\dim_k A) \mathbb{Z}.
\]
\end{Proposition}
\begin{proof}
Since $A$ is local, $K_0(D^b(A)) \cong \mathbb{Z}$ where the class $[S]$ of the simple (which is one dimensional since $k$ is algebraically closed) is sent to $1$.  Now apply \cite[Corollary 3.9 (2)]{Beligiannishomologicaltheoryofhomologicallyfinitesubcategories} to complete the proof.
\end{proof}

As a consequence, we obtain obstructions to (naive) generalisations of our main result.

\begin{Corollary}
Let $G \subseteq \GL(2, \mathbb{C})$ be a finite subgroup and set $R=\mathbb{C}\llbracket x, y \rrbracket^G$. If $R$ satisfies one of the following conditions
\begin{itemize}
\item[(a)] $R$ is Gorenstein with dual graph of Dynkin type $D$ or $E$. 
\item[(b)] $R$ has dual graph
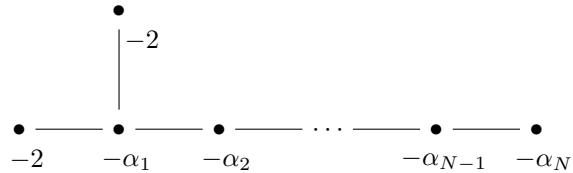
\begin{equation*}\begin{tikzpicture}[description/.style={fill=white,inner sep=2pt}]
    \matrix (n) [matrix of math nodes, row sep=3em,
                 column sep=2.5em, text height=1.5ex, text depth=0.25ex,
                 inner sep=0pt, nodes={inner xsep=0.3333em, inner
ysep=0.3333em}]
    {  & \bullet \\
       \bullet & \bullet & \bullet & \cdots & \bullet & \bullet \\
    };
\path[-] (n-1-2) edge (n-2-2);
\path[-] (n-2-1) edge (n-2-2);
\path[-] (n-2-2) edge (n-2-3);
\path[-] (n-2-3) edge (n-2-4);
\path[-] (n-2-4) edge (n-2-5);
\path[-] (n-2-5) edge (n-2-6);

 \node at ($(n-2-1.south) + (1mm, -1.5mm)$) {$-2$};
 \node at ($(n-2-2.south) + (1mm, -1.5mm)$) {$-\alpha_1$};
 \node at ($(n-2-3.south) + (1mm, -1.5mm)$) {$-\alpha_2$};
 \node at ($(n-1-2.south) + (3mm, -1.5mm)$) {$-2$};
 \node at ($(n-2-5.south) + (1mm, -1.5mm)$) {$-\alpha_{N-1}$};
 \node at ($(n-2-6.south) + (1mm, -1.5mm)$) {$-\alpha_{N}$};
\end{tikzpicture}\end{equation*}
where the Hirzebruch-Jung continued fraction $[\alpha_1, \ldots, \alpha_N]$ satisfies 
\[[\alpha_1, \ldots, \alpha_N]=\frac{n}{2m}\] for coprime integers $1<2m <n$; (for example $N=2$, $\alpha_1=a$, $\alpha_2=2b$ with $a \ge 2, b \ge 1$). 
\end{itemize}
then there exists no finite dimensional\footnote{not necessarily commutative} local $\mathbb{C}$-algebra $S$ with 
\[
D_{sg}(R) \cong D_{sg}(S).
\]
\end{Corollary}

\begin{proof}
Theorem \ref{GrothendieckClassgroup} and Proposition \ref{P:GrothendieckLocal} show that a finite dimensional local algebra $S$ with $D_{sg}(S) \cong D_{sg}(R)$ satisfies $\dim_{\mathbb{C}} S = |Cl(R)|$. It is known for quotient singularities $R=\mathbb{C}\llbracket x, y \rrbracket^G$ that $Cl(R) \cong G/[G, G]$, see e.g. \cite{BrieskornInventiones1968}.

Now we consider case (a).  The singularity categories $E_8, E_7, E_6, D_n$ $(n\ge 4)$ have 8,7,6, and $n$ indecomposable objects respectively and have finite dimensional $\Hom$ spaces.  One can compute that $Cl(R)$ has order $4$ for  singularities of type $D_n$ and order $9 - n$ for singularities of type $E_n$ (with $n=6, 7, 8$), cf. \cite[Satz 2.11]{BrieskornInventiones1968}. 

To prove the claim in type $E$ we list the local algebras with dimension $\le 3$ and show these algebras do not have singularity categories of the correct form.

A finite dimensional, local, associative $\mathbb{C}$-algebra $S$ of dimension $\le 3$ is of the form
\[
S \cong \left\{ \begin{array}{c c} 
\mathbb{C}[x]/(x^i) &\text{ for } 1 \le i \le 3, \text{ or } \\
\mathbb{C}[x,y]/(x,y)^2 &
\end{array} \right.
\]
The singularity category of $\mathbb{C}[x]/(x^i)$ contains $i-1$ indecomposables. So for $i \leq 3$ it cannot equal a singularity category of type $E$. The singularity category of $\mathbb{C}[x,y]/(x,y)^2$ is Hom-infinite so cannot equal a type $E$ singularity category. This proves the type $E$ claim.

To prove the type $D$ claim we consider all 4 dimensional, local, associative $\mathbb{C}$-algebras, which we classify by the dimension of their socle - since $S$ is local $\dim \mathrm{soc} \, S \leq 3$.

If $\dim \mathrm{soc}\, S=1$, then $S$ is weakly symmetric and since it is local it is selfinjective by \cite[Corollary IV.6.5]{SkowronskiYamagataFrobeniusAlgebras1}. In particular, $D_{sg}(S) \cong S-\underline{\mod}$ by Buchweitz Theorem \cite[Theorem 4.4.1]{BuchweitzMCM}. Now there are two cases: 
\begin{itemize}
\item  if $S \cong \mathbb{C}[x]/(x ^4)$, then $S-\underline{\mod}$ has precisely $3$ indecomposable objects.
\item otherwise $S/\mathrm{soc}(S) \cong \mathbb{C}\langle x, y \rangle/(x, y)^2$, and this algebra has infinite representation type and therefore  $S-\mod \supseteq S/\mathrm{soc}(S)-\mod$ has infinite representation type. It follows that $S-\underline{\mod}$ has infinite representation type. 
\end{itemize}
 This completes the case of a one-dimensional socle.
 
We next consider the case $\dim \mathrm{soc}\, S = 2$. In this case, $S \cong \mathbb{C}\langle x, y\rangle/I$. We claim that  $D_{sg}(S)$ is Hom-infinite. Let  $T=S/\mathrm{rad} S$ be the simple $S$-module, which is one dimensional since $\mathbb{C}$ is algebraically closed. It is sufficient to show that the number of indecomposable direct summands of $\Omega^n(T)$ is unbounded for growing $n$ and that all these summands have infinite projective dimension as $S$-modules. Since it well-known that $\Omega^n(T) \cong T[-n]$ in $D_{sg}(S)$ (see e.g. \cite[Lemma 2.2]{XWChenSingCatofRadicalsquare0alg}) this shows that $D_{sg}(S)$ is Hom-infinite. We compute the syzygies of $T$. As $\dim \mathrm{soc}\, S=2$ and $\dim S=4$ there are $\lambda, \mu \in \mathbb{C}$ such that $z:=\lambda x + \mu y \in \mathrm{soc}\, S$. Using that $S$ is local one can check that $\Omega(T) \cong \mathrm{rad} \, S \cong Sz \oplus U \cong T \oplus U$ where $U$ is $2$-dimensional and indecomposable. The syzygy of  an arbitrary two dimensional indecomposable $S$-module $V$ is two dimensional. There are two cases
\begin{itemize}
\item $\Omega(V)$ is indecomposable and hence two dimensional.
\item $\Omega(V)$ is decomposable. Then  $\Omega(V) \cong T \oplus T$ since $S$ is local.
\end{itemize}
It follows that $T$ and any indecomposable $2$-dimensional $S$-module have infinite projective dimension. Moreover, these are the only possible direct summands of $\Omega^n(T)$. Finally, $\Omega^n(T)$ has at least $n+1$ indecomposable direct summands. This completes the proof of the two-dimensional socle case.

If $\dim \mathrm{soc} \, S = 3$, then $S \cong \mathbb{C}[x, y, z]/(x, y, z)^2$. This has a singularity category which is not Hom-finite (in fact it is the Kn\"{o}rrer invariant algebra $K_{4, 1}$ from above), see e.g. \cite[Theorem C]{XWChenSingCatofRadicalsquare0alg}.

None of these cases contain $n \ge 4 $ indecomposables and are Hom-finite, so this proves the part (a) claim for type D.

For part (b), we note that the corresponding group is $G=D_{n, 2m}$ in the notation of \cite{Riemenschneider77}. One can check that the abelianisation of $G$ is isomorphic to $\mathbb{Z}/2\mathbb{Z} \times \mathbb{Z}/2(n-2m)\mathbb{Z}$. In particular, it is not cyclic and therefore (using Theorem \ref{GrothendieckClassgroup} and Proposition \ref{P:GrothendieckLocal} again)  there cannot be a finite dimensional local $\mathbb{C}$-algebra  $S$ with equivalent  singularity category.
\end{proof}

\section{Descriptions of the algebras} \label{Algebras}
The previous sections have recalled the algebras $\Lambda_{r,a}$, $A_{r,a}$, and $R_{r,a}$,  introduced the algebras $K_{r,a}$, and proved the main result relating the singularity categories of $R_{r,a}$ and $K_{r,a}$. In this section we give explicit presentations of these algebras in terms of generators and relations, and we show that $\Lambda_{r,a}$ can be constructed from the representation theory of $K_{r,a}$; namely as a noncommutative resolution in the sense of Dao, Iyama, Takahashi, and Vial \cite{DITV}.

\subsection{Description of the reconstruction algebra $A_{r,a}^{\op}$} \label{Reconstruction algebras}
We recall a presentation given in \cite{RCAA} of the algebra $A^{\op}:=A^{\op}_{r,a}=A^{\op}_{[\alpha_1, \dots, \alpha_n]}$ as a quiver with relations. The quiver has $n+1$ vertices in correspondence with the simples $t_0, \dots, t_{n}$. Below we recall the dimension of $\Ext_{A^{\op}}^i(t_i,t_j)$ as calculated in {\cite[Theorem 3.2]{WemyssGL2}}. Then the presentation of $A^{\op}_{r,a}$ as a quiver with relations satisfies that that the number of arrows to vertex $j$ from vertex $i$ equals the dimension of $\Ext_{A^{\op}}^1(t_i, t_j)$, and the number of generators for the relations between paths to vertex $j$ from $i$ equals the dimension of $\Ext_{A^{\op}}^2(t_i, t_j)$. As is noted in \cite[Corollary 3.3]{RCAA} this is true as we are in the complete local setting, see for example \cite[Proposition 3.4]{BuanIyamaReitenSmith}.

\begin{Lemma}[{\cite[Theorem 3.2]{WemyssGL2}}] \label{lem:recon exts} The dimension of $\Ext^k_{A^{\op}}(t_i,t_j)$ is $0$ for $k>3$ and otherwise is given in the following list.
\begin{align*}
&\dim \Ext_{A^{\op}}^1(t_i,t_j) = \left\{ \begin{array}{ll}
 1 & \text{\scriptsize{ if $n>1$, $1 \le i \le n$, and $|i-j|=1$}} \\
\alpha_n-1 & \text{\scriptsize{ if $n>1$, $i=0$, and $j=n$ or $1$}} \\
\alpha_i-2 & \text{\scriptsize{ if $n>1$, $i=0$, and $1<j<n$}} \\
 2 & \text{\scriptsize{ if $n=1$, $i =1$, and $j=0$}} \\
\alpha_n & \text{\scriptsize{ if $n=1$, $i=0$, and $j=1$}} \\
0 & \text{\scriptsize{ otherwise}}
\end{array}
\right.
\\
&\dim \Ext_{A^{\op}}^2(t_i,t_j) = \left\{ \begin{array}{ll}
 \alpha_i-1 & \text{\scriptsize{ if $1 \le i \le n$ and $i=j$}} \\
\sum_{k=1}^n (\alpha_k-2)+1 & \text{\scriptsize{$i=j=0$}} \\
0 & \text{\scriptsize{ otherwise}}
\end{array}
\right.
\\
&\dim \Ext_{A^{\op}}^3(t_i,t_j) = \left\{ \begin{array}{ll}
 \alpha_j-2 & \text{\scriptsize{ if $j=0$ and $1 \le i \le n$}} \\
0 & \text{\scriptsize{ otherwise}}
\end{array}
\right.
\end{align*}
\end{Lemma}

This describes the shape of the quiver defining the reconstruction algebra. 

\begin{Definition} \label{def:ReconQuiver}
Define a quiver $Q^{\textnormal{Recon}}_{[\alpha_1, \dots, \alpha_n]}$ with vertices \[
Q_0^{\textnormal{Recon}}:=\{ 0,1,2, \dots, n \}
\]
and arrows 
\[
Q_1^{\textnormal{Recon}}:=\{ a_0, a_1, \dots, a_n, c_0, c_1, \dots, c_n, k_2, \dots, k_{v_{n}} \}
\]
whose heads and tails defined by
\begin{align*}
h(c_i)&=i & t(c_i)&=i-1\\
h(a_i)&=i-1 & t(a_i)&=i \\ 
h(k_i) &=0 & t(k_j)& = i  \text{ for $u_i+1  <  j \le v_i$.}
\end{align*}
where we work modulo $n+1$ in the vertex labelling, let $u_i=  \sum_{k=1}^{i-1} (\alpha_{k}-2)$ and $v_i= \sum_{k =1}^i (\alpha_k-2)+1$, and notate $k_1:= a_{1}$, $k_{v_n+1}:=c_{0}$, $a_0=a_{n+1}$, $c_0=c_{n+1}$, $A_{0}^{t}:=a_0a_n \dots a_{t+1}$ and $C_{0}^{t}:=c_1 \dots c_t$. We also use the notation $t(j):=t(k_j)$ for the tail vertex of the arrow $k_j$.
\end{Definition}

A presentation of the relations was also explicitly determined in \cite{RCAA}.

\begin{Definition}
For $1 \le i \le n$ consider the following elements of $\mathbb{C}Q_{[\alpha_1, \dots, \alpha_n]}^{\textnormal{Recon}}$:
\begin{align*}
&\text{if $\alpha_i>2$:} \qquad \left\{ \begin{array}{l}
a_{i} c_{ i} -k_{u_i+2}A_{0}^{i} \\
 k_j C_{0}^{i} -k_{j+1} A_{0}^{ i}  \qquad \text{for } u_i +1 < j < v_i \\
 k_{v_i}C_{0}^{ i} - c_{i+1}a_{i+1} \\
 \end{array} \right.  \\
&\text{if $\alpha_i=2$:} \qquad \quad a_{i} c_{i}  -c_{ i+1} a_{i+1}
\end{align*}
and for $i=0$ consider the elements
\begin{align*}
 \quad &  A_{0 }^{t({j+1}) }k_{j+1}- C_{0}^{ t(j)} k_j  \qquad \text{for } 1 \le j \le \sum (a_i-2)+2.
\end{align*} 
Define $I_{[\alpha_1, \dots, \alpha_n]}^{\textnormal{Recon}}$ to be the two-sided ideal of $\mathbb{C}Q_{[\alpha_1, \dots, \alpha_n]}^{\textnormal{Recon}}$ generated by these elements.
\end{Definition}

\begin{Proposition}[{\cite[Definition 2.3]{RCAA}}] \label{Reconstruction presentation Proposition}
The reconstruction algebra $A^{\op}_{r,a}=A^{\op}_{[\alpha_1, \dots, \alpha_n]}$ can be presented as the path algebra of a quiver with relations
\[
A^{\op}_{r,a}:= \frac{\mathbb{C}Q_{[\alpha_1, \dots, \alpha_n]}^{\textnormal{Recon}}}{I_{[\alpha_1, \dots, \alpha_n]}^{\textnormal{Recon}}}.
\]
\end{Proposition}

\begin{Example}[$r=17, \, a=5$] The algebra $A^{\op}_{[4,2,3]}=A^{\op}_{17,5}$ can be presented as the path algebra of the following quiver with relations.
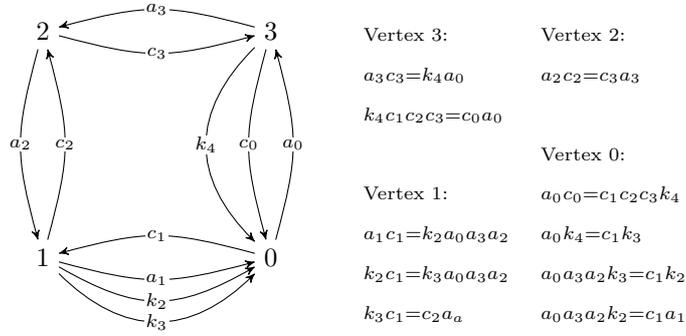
\begin{align*}
\begin{aligned}
\begin{tikzpicture} [bend angle=45, looseness=1.2]
\node (C1) at (-3,-3)  {$1$};
\node (C2) at (-3,0)  {$2$};
\node (C3) at (0,0) {$3$};
\node (C*) at (0,-3) {$0$};
\draw [->,bend right=15] (C*) to node[gap] {$\scriptstyle{c_{1}}$} (C1);
\draw [->,bend right=15] (C1) to node[gap] {$\scriptstyle{a_{1}}$} (C*);
\draw [->,bend right=15] (C1) to node[gap] {$\scriptstyle{c_{2}}$} (C2);
\draw [->,bend right=15] (C2) to node[gap] {$\scriptstyle{a_{2}}$} (C1);
\draw [->,bend right=45] (C1) to node[gap] {$\scriptstyle{k_{3}}$} (C*);
\draw [->,bend right=30] (C1) to node[gap] {$\scriptstyle{k_{2}}$} (C*);
\draw [->,bend right=15] (C2) to node[gap] {$\scriptstyle{c_{3}}$} (C3);
\draw [->,bend right=15] (C3) to node[gap] {$\scriptstyle{a_{3}}$} (C2);
\draw [->,bend right=15] (C3) to node[gap] {$\scriptstyle{c_{0}}$} (C*);
\draw [->,bend right=15] (C*) to node[gap] {$\scriptstyle{a_{0}}$} (C3);
\draw [->,bend right=45] (C3) to node[gap] {$\scriptstyle{k_{4}}$} (C*);
\end{tikzpicture}
\end{aligned}
& \qquad 
\begin{aligned}
&\scriptstyle{\text{Vertex 3:}}\\
&\scriptstyle{a_3c_3}\scriptstyle{=k_4 a_0} \\
&\scriptstyle{k_4 c_1c_2c_3}\scriptstyle{=c_0a_0} \\
\\
&\scriptstyle{\text{Vertex 1:}}\\
&\scriptstyle{a_1c_1 }\scriptstyle{= k_2a_0a_3a_2 }\\
&\scriptstyle{k_2  c_1} \scriptstyle{= k_3 a_0a_3a_2}\\
&\scriptstyle{k_3  c_1} \scriptstyle{= c_2a_a}
\end{aligned}
 \quad
\begin{aligned}
&\scriptstyle{\text{Vertex 2:}}\\
 &\scriptstyle{a_2c_2} \scriptstyle{=c_3a_3}  \\
\\
&\scriptstyle{\text{Vertex 0:}}\\
&\scriptstyle{a_0 c_0}\scriptstyle{= c_1c_2c_3 k_4}  \\
&\scriptstyle{ a_0 k_4 }\scriptstyle{=c_1 k_3} \\
&\scriptstyle{a_0 a_3 a_2  k_3} \scriptstyle{= c_1 k_2}\\
&\scriptstyle{a_0 a_3 a_2  k_2 } \scriptstyle{= c_1 a_1}
\end{aligned}
\end{align*}
\end{Example}

\subsection{Description of the invariant algebra $R_{r,a}$} \label{Sec:R presentation}
This section recalls a presentation of the invariant algebra $R_{r,a}=\mathbb{C}[[x,y]]^{\frac{1}{r}(1,a)}$ due to Riemenschneider.

\begin{Theorem}[{\cite[Sections 1 and 2]{RiemenschneiderNach}}] The ring $R_{r,a}:=\mathbb{C}[[x,y]]^{\frac{1}{r}(1,a)}$ is isomorphic to the quotient of $\mathbb{C}[[z_0, \dots, z_{l+1}]]$ by the ideal generated by the elements
\begin{equation*}
z_{j+1} z_i  - z_{i+1} \left( \prod_{k=i+1}^{j}  z_k^{\beta_k-2} \right) z_{j}
\end{equation*}
for $0 \le i <j \le l+1$, where the $\beta_k$ and $l$ are defined by the Hirzebruch-Jung continued fraction $r/(r-a)= [ \beta_1, \dots, \beta_l]$ dual to $r/a=[ \alpha_1, \dots, \alpha_n]$. The embedding dimension of $R_{r,a}$ is $l+2$.
\end{Theorem}

We now recap some properties of Hirzebruch-Jung continued fractions.
\begin{Definition} \label{def:HJFraction}
For $0 < a< r$ coprime integers, recall that the Hirzebruch-Jung continued fraction expansion is defined by 
\[
\frac{r}{a} = \alpha_1 - \cfrac{1}{\alpha_{2}
          -\cfrac{1}{\dots - \cfrac{1}{\alpha_n} } } =\left[ \alpha_1, \dots, \alpha_n \right].
\]
\end{Definition}

Associated to a fraction $r/a=[ \alpha_1, \dots, \alpha_n]$ is its point diagram (see Example \ref{Ex:Point Diagram}):  
\begin{enumerate}
\item On row 1 draw $\alpha_1-1$ consecutive points.
\item On row $j$ draw $\alpha_j-1$  consecutive points starting immediately below the last point on row $j-1$.
\end{enumerate}
Riemenschneider duality relates a fraction expansion $r/a=[ \alpha_1, \dots, \alpha_n]$ to its dual fraction expansion $r/(r-a)=[\beta_1, \dots, \beta_l]$.

\begin{Theorem}[Riemenschneider duality, {\cite[Section 3]{RiemenschneiderNach}}] \label{Riemenschneider Duality Theorem}
The number of points in row $i$ of the point diagram is $\alpha_i-1$ and the number of points in column $j$  of the point diagram is $\beta_j-1$.

In particular the following relations hold: 
\begin{enumerate}
\item $\sum_{i=1}^n(\alpha_i-1)=\sum_{i=1}^l(\beta_i-1)$,
\item $\sum_{i=1}^n(\alpha_i-2)+1=l$, and
\item $\sum_{i=1}^l(\beta_i-2)+1=n$.
\end{enumerate}
\end{Theorem}

\begin{Example} \label{Ex:Point Diagram} The following point diagram is for the fraction $18/5=[4,3,2]$ with dual $18/13=[2,2,3,3]$.
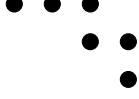
\begin{center}
\begin{tikzpicture}
\filldraw (0,0) circle (3pt);
\filldraw (0.5,0) circle (3pt);
\filldraw (1,0) circle (3pt);
\filldraw (1,-0.5) circle (3pt);
\filldraw (1.5,-0.5) circle (3pt);
\filldraw (1.5,-1) circle (3pt);
\end{tikzpicture}
\end{center}
\end{Example}
We will later need to induct on the length of fraction expansions.
\begin{Corollary} \label{HJfractionCorollary}
Suppose that $0<a<r$ are coprime integers with Hirzebruch-Jung continued fraction expansions $r/a=[\alpha_1, \dots, \alpha_n]$ and  $r/(r-a)=[\beta_1, \dots, \beta_l]$. Then:
\begin{enumerate}
\item
$[\alpha_1, \dots, \alpha_{n+1}]$ is dual to $[\beta_1, \dots \beta_l+1,2, \dots, 2]$ where there are $\alpha_{n+1}-2$ values of $2$ added at the end of the fraction expansion 
\item
$[\alpha_0, \dots, \alpha_n]$ is dual to $[2, \dots, 2, \beta_1+1, \dots \beta_l]$ where there are $\alpha_{0}-2$ additional  values of $2$ added at the start of the fraction expansion. 
\item
For $i>0$ $[\alpha_{i+1}, \dots, \alpha_n]$ is dual to $[\beta_j-c,\beta_{j+1}, \dots \beta_l]$ where $\sum_{k=j}^l (\beta_k-2) -c+1  =n-i $ and $0 <  c <\beta_j-2$.
\item
For $i<n$ $[\alpha_{1}, \dots, \alpha_i]$ is dual to $[\beta_1, \dots, \beta_j-c]$ where $\sum_{k=1}^j (\beta_k-2)-c+1   =i $ and $0 <  c <\beta_j-2$.
\item
Recall the value $t(j)$ from Definition \ref{def:ReconQuiver}. Then $t(j)=\sum_{i=1}^{j-1} (\beta_i-2)+1$ for $1 \le j \le l$.
\end{enumerate}
\end{Corollary}
\begin{proof}
Items (1),(2),(3) and (4) follow immediately from Riemenschneider duality, Theorem \ref{Riemenschneider Duality Theorem}.

We assume we have a fraction $[\alpha_1, \dots, \alpha_n]$ with dual $[\beta_1, \dots, \beta_l]$, and we prove (5) by induction on $l$. As the base case, suppose $l=1$. Then $t(k_1)=t(a_1)=1$ by definition. Now suppose that $t(j)=\sum_{i=1}^{j-1} (\beta_i-2)+1$ for $1 \le j \le l$ and consider the fraction $[\beta_1, \dots, \beta_l, \beta_{l+1}]$ which is dual to $[\alpha_1, \dots, \alpha_n+1, 2, \dots, 2]$ by part (1). It follows that $t(l+1)=n$, and by Theorem \ref{Riemenschneider Duality Theorem}, $\sum_{i=1}^{l} (\beta_i-2)+1=n$. Hence the result follows by induction.
\end{proof}

\subsection{Description of the algebra $\Lambda_{r,a}$} \label{HPalgebra}
We now give a presentation of the algebra $\Lambda:=\Lambda_{r,a}$ analogous to that of $A^{\op}_{r,a}$. Below we calculate the groups $\Ext_{\Lambda}^i(\sigma_i, \sigma_j)$. Then there is a presentation of $\Lambda_{r,a}$ as the path algebra of a quiver with relations such that  the vertices are in correspondence with the simples $\sigma_0, \dots, \sigma_{n}$, the number of arrows to vertex $j$ from vertex $i$ equals the dimension of $\Ext_{\Lambda}^1(\sigma_i, \sigma_j)$, and the number of generators for the relations between paths to vertex $j$ from $i$ equals the dimension of $\Ext_{\Lambda}^2(\sigma_i, \sigma_j)$.

\begin{Lemma} \label{Ext Quiver Calculation Lemma} The dimension of $\Ext^k_\Lambda(\sigma_i,\sigma_j)$ is $0$ for $k > 2$ and otherwise is given in the following list.

\begin{align*}
 &\dim \Ext_{\Lambda}^1(\sigma_i,\sigma_j) = \left\{ \begin{array}{ll}
 1 & \text{\scriptsize{ if $i \neq 0$ and $|i-j|=1$,}} \\
\alpha_1-1 & \text{\scriptsize{ if $i=0$ and $j=1$,}} \\
\alpha_j-2 & \text{\scriptsize{ if $i=0$ and $1<j \le n$, or}} \\
0 & \text{\scriptsize{ otherwise.}}
\end{array}
\right.
\\
&\dim \Ext_{\Lambda}^2(\sigma_i,\sigma_j) = \left\{ \begin{array}{ll}
 \alpha_i-1 & \text{\scriptsize{ if $i=j \neq 0$, or}} \\
0 & \text{\scriptsize{ otherwise.}}
\end{array}
\right.
\end{align*}
\end{Lemma}

\begin{proof}

We note that the equivalence functor $F$ from Proposition $\ref{Equivalence of objects supported on curve Proposition}$ allows us to use using Lemma \ref{lem:recon exts} to calculate $\Ext^k(\sigma_i,\sigma_j)$ for all $i,j$ not equal to $0$:
\[
\dim \Ext^k_{\Lambda_{r,a}}(\sigma_i,\sigma_j)=\dim \Ext^k_{A_{r,a}}(s_i,s_j)= \dim \Ext^k_{A_{r,a}^{\op}}(t_j,t_i)
\]
for $1 \le i,j \le n$.

Hence we are only left to calculate $\Ext^k_{\Lambda_{r,a}}(\sigma_i,\sigma_j)$ when either $i$ or $j$ is equal to 0. Recall the simple modules $\sigma_i$, the standard modules $\mathcal{L}_i$ and the projective modules $\Lambda_i$ for the algebra $\Lambda_{r,a}$ defined in Section \ref{Subcategory}. We first calculate $\dim \Ext^k_{\Lambda}(\mathcal{L}_i,\sigma_j)$ and then use the short exact sequences relating the standard modules to the simple modules to calculate $\dim \Ext_{\Lambda}^k(\sigma_i,\sigma_j)$. Firstly, as $\sigma_{0}=\mathcal{L}_{0}$
\begin{align*}
\dim \Ext^k_{\Lambda}(\mathcal{L}_i, \sigma_{0})= \left\{ \begin{array}{cc} 
1 & \text{\scriptsize{ if $i=0$ and $k=0$}} \\
0 & \text{\scriptsize{otherwise.}}
\end{array}
\right.
\end{align*}
To calculate $\dim \Ext^k_{\Lambda}(\mathcal{L}_i,\sigma_j)$ when $1 \le j \le n$ we can reduce to a cohomology calculation on each rational curve. Suppose $\iota_j:C_j \rightarrow X$ is the closed immersion including the curve $C_j$ into $X$ for $1 \le j \le n$. Then $\mathcal{O}_{C_j} := \iota_{j*}\mathcal{O}_{\mathbb{P}^1}$ and for any divisor $D$ on $X$  it is the case that $\iota_j^* \mathcal{O}_X(-D) = \mathcal{O}_{C_j}(-C_j \cdot D)$. Then by adjunction we can deduce
\begin{equation*}
\Ext^k_{\Lambda}(\mathcal{O}(-D),\sigma_j)= \left\{ \begin{array}{c} \Ext^k_{\mathbb{P}^1}(\mathcal{O}_{\mathbb{P}^1},\mathcal{O}_{\mathbb{P}^1}(-1 +D\cdot C_j) \text{ if } 1<j<n \\ 
\Ext^k_{\mathbb{P}^1}(\mathcal{O}_{\mathbb{P}^1},\mathcal{O}_{\mathbb{P}^1}(D\cdot C_j) \text{ if } j=n.
\end{array} \right.
\end{equation*}
This reduces the calculation to the combinatorics of the intersection theory and the cohomology of line bundles on $\mathbb{P}^1$. Recall that 
\begin{equation*}
\dim \Ext^k_{\mathbb{P}^1}(\mathcal{O}_{\mathbb{P}^1},\mathcal{O}_{\mathbb{P}^1}(j)) = \left\{ \begin{array}{cc} j+1 & \text{ if $j \ge 0$ and $k=0$} \\ j+1 & \text{ if $j<-1$ and $k=1$} \\ 0 & \text{otherwise.} \end{array} \right.
\end{equation*}
In particular, as $\mathcal{L}_i = \mathcal{O}(-C_{i+1} \dots -C_{n})$ it follows that $\Ext_{\Lambda}^k(\mathcal{L}_i,\sigma_j)=0$ for $k\ge 2$ and
\begin{align*}
\dim \Hom_{\Lambda}(\mathcal{L}_i,\sigma_j)= \left\{ \begin{array}{cc} 
1 & \text{\scriptsize{ if $i=j$}} \\
0 & \text{\scriptsize{ otherwise} }
\end{array}
\right.
&
\qquad
\dim \Ext_{\Lambda}^1(\mathcal{L}_i,\sigma_j)= \left\{ \begin{array}{cc} 
0 & \text{\scriptsize{ if $j \le i$}} \\
\alpha_j-1 & \text{\scriptsize{if $j=i+1$}} \\ 
\alpha_j-2 & \text{\scriptsize{if $j \ge i+2 $}. }\\ 
\end{array}
\right.
\end{align*}

As $\mathcal{L}_{0}=\sigma_{0}$ we have now calculated $\dim \Ext^k_{\Lambda}(\sigma_0, \sigma_j)$, and are only left to calculate $\dim \Ext^k_{\Lambda}(\sigma_i, \sigma_0)$. To do this we consider the short exact sequences 
\begin{equation*}
0\rightarrow \mathcal{L}_{i-1} \rightarrow \mathcal{L}_i \rightarrow \sigma_i \rightarrow 0
\end{equation*}
for $1 \le i \le n$. Applying $\Hom_{\Lambda}(-,\sigma_0)$ produces the long exact sequences
\[
 \dots \rightarrow \Ext_{\Lambda}^i(\sigma_i,\sigma_0) \rightarrow \Ext_{\Lambda}^i(\mathcal{L}_i,\sigma_0) \rightarrow \Ext_{\Lambda}^i(\mathcal{L}_{i-1},\sigma_0) \rightarrow \Ext_{\Lambda}^{i+1}(\sigma_i,\sigma_0) \rightarrow \dots
\]
As $\Ext_{\Lambda}^k(\mathcal{L}_i,\sigma_0)=0$ for $k \ge 2$, we can deduce that $\Ext_{\Lambda}^k(\sigma_i,\sigma_0)$ vanishes for $k>2$. Looking at the starting terms in the short exact sequence and using the previous calculations for $\Ext_{\Lambda}^k(\mathcal{L}_i,\sigma_0)$ yields the following exact sequences:

\[
0 \rightarrow \mathbb{C} \rightarrow \Ext^1_{\Lambda}(\sigma_1,\sigma_{0}) \rightarrow 0 \rightarrow 0 \rightarrow \Ext_{\Lambda}^2(\sigma_1,\sigma_{0}) \rightarrow 0. 
\]
for $i=1$ and
\[
 0 \rightarrow  \Ext^1_{\Lambda}(\sigma_i,\sigma_0) \rightarrow 0 \rightarrow 0 \rightarrow \Ext^2_{\Lambda}(\sigma_i,\sigma_0) \rightarrow 0 \\
\]
for $i>1$. This calculates $\dim \Ext^k_{\Lambda}(\sigma_i,\sigma_0)$ for all $i$, and hence we have found the values of $\Ext^k_{\Lambda}(\sigma_i,\sigma_j)$ for all $i,j,k$.
\end{proof}

In order to help explicitly determine the relations we calculate the ranks of the sheaves $\Lambda_i$.
\begin{Lemma} \label{Rank Calculation Lemma}
Let $\lambda_i$ denote the rank of $\Lambda_i$ and $\lambda_{n+1}=0$. Then $\lambda_n=1$ and 
\[
\lambda_{i-1}=\sum_{k = i}^n (\alpha_k-2)\lambda_k+\lambda_{i}+1 = \alpha_i \lambda_i - \lambda_{i+1} \text{ for $1 \le i \le n$}.
\]
Moreover,
\[
\lambda_{i-1}/\lambda_i = [\alpha_{i}, \dots, \alpha_{n}] \text{ for $1 \le i \le n$},
\] the Hirzebruch-Jung continued fraction corresponding to the negative of the self intersection numbers of the first $i$ curves. 
\end{Lemma}
\begin{proof}
Firstly, as $\Lambda_n=\mathcal{L}_n=\mathcal{O}_X$ is a line bundle  $\lambda_n = \rk \, \mathcal{L}_n=1$. Then the relation 
\[
\lambda_{i-1}=\sum_{k = i}^n (\alpha_k-2)\lambda_k^{}+\lambda_{i}+1
\]
follows from the short exact sequence
\[
 0  \rightarrow  \bigoplus_{k = i}^n \Lambda_k^{\oplus \alpha_j-2} \oplus \Lambda_i   \rightarrow \Lambda_{i-1}  \rightarrow   \mathcal{L}_{i-1}  \rightarrow  0 
\]
as rank is additive and $\mathcal{L}_{i-1}$ is a line bundle. By this result $\lambda_{i-1}-\lambda_{i}=\sum_{k = i}^n (\alpha_k-2)\lambda_k+1$, hence
\begin{align*}
\lambda_{i-1}&=\sum_{k = i}^n (\alpha_k-2)\lambda_k+\lambda_{i}+1 \\
&= (\alpha_i-1)\lambda_i + \sum_{k =i+1}^n (\alpha_k-2)\lambda_k+1 = \alpha_i\lambda_i -\lambda_{i+1}.
\end{align*}

It follows that $ \lambda_{i-1}/\lambda_i = [\alpha_{i}, \dots, \alpha_{n}]$ by induction on $i$. Firstly, as $\Lambda_n=\mathcal{L}_n$ it follows that $\lambda_n=1$ and $\lambda_{n-1}= (\alpha_n-2)\lambda_n+\lambda_n+1= \alpha_n$. This establishes the base case $[\alpha_n]=\lambda_{n-1}/\lambda_n$. Now we assume as the induction hypothesis that $\lambda_{j-1}/\lambda_j = [\alpha_{j}, \dots, \alpha_{n}]$ for $j<i$. By the definition of the Hirzebruch-Jung continued fraction
\begin{equation*}
[\alpha_i, \dots, \alpha_n] = \alpha_i- \frac{1}{[\alpha_{i+1}, \dots, \alpha_{n}]}
\end{equation*}
and by the induction hypothesis this equals $\alpha_i - \lambda_{i+1}/\lambda_{i}$. Using the earlier relation we deduce
\begin{align*}
[\alpha_i, \dots, \alpha_1] &= \alpha_i - \lambda_{i+1}/\lambda_{i} \\
&= \frac{\alpha_i \lambda_{i}-\lambda_{i+1}}{\lambda_{i}}= \frac{\lambda_{i-1}}{\lambda_i}.
\end{align*}
We also note that as  $\lambda_{i-1} = \alpha_i \lambda_i - \lambda_{i+1}$ and  $\lambda_n=1$ any consecutive pair of $\lambda_i$ are necessarily coprime. 
\end{proof}
This allows us to calculate the dimension vectors of the simple, standard, and projective modules. To ease notation we set $\hom_{\Lambda}(N,M)= \dim \Hom_{\Lambda}(N,M)$.

\begin{Lemma} \label{Module Dimensions Lemma}
We can calculate the dimensions of the simple, standard, and projective modules.
\begin{enumerate}
\item
$
\hom_{\Lambda}(\Lambda_i, \sigma_j)= \left\{ \begin{array}{cc}
1 & \text{ if $j = i$} \\
0 & \text{ otherwise.}
\end{array} \right.
$

\item
$
\hom_{\Lambda}(\Lambda_i, \mathcal{L}_j)= \left\{ \begin{array}{cc}
1 & \text{ if $j \ge i$} \\
0 & \text{ otherwise.}
\end{array} \right.
$

\item
$
\hom_{\Lambda}(\Lambda_i, \Lambda_j)= \left\{ \begin{array}{ll}
\lambda_j & \text{if $j \ge i$} \\
\sum_{k=j+1}^{n} (\alpha_k-2)\hom_{\Lambda}(\Lambda_i, \Lambda_k) + \hom_{\Lambda}(\Lambda_i, \Lambda_{j+1})  & \text{if $j<i$.}
\end{array} \right.
$
\end{enumerate}

\end{Lemma}
\begin{proof}

Firstly, (1) is clear as the $\Lambda_i$ are the indecomposable projective modules and the $\sigma_j$ are the simple modules in $\Lambda$-$\mod$.

To prove (2) we first note that $\mathcal{L}_{0}=\sigma_{0}$, and recall that $ \hom_{\Lambda}(\Lambda_i, \sigma_j)=\delta_{i,j}$. As $\Lambda_i$ is projective applying $\Hom_{\Lambda}(\Lambda_i,-)$ to the short exact sequence 
\[
0 \rightarrow \mathcal{L}_{j-1} \rightarrow \mathcal{L}_{j} \rightarrow \sigma_{j} \rightarrow 0
\]
produces the short exact sequence 
\[
0 \rightarrow \Hom_{\Lambda}(\Lambda_i,\mathcal{L}_{j-1}) \rightarrow \Hom_{\Lambda}(\Lambda_i,\mathcal{L}_{j}) \rightarrow \Hom_{\Lambda}(\Lambda_i,\sigma_{j}) \rightarrow 0
\]
and hence we deduce that 
\[
\hom_{\Lambda}(\Lambda_i, \mathcal{L}_{j-1})+\delta_{i,j}= \hom_{\Lambda}(\Lambda_i, \mathcal{L}_{j})
\]
where $\delta_{i,j}$ is the Kronecker delta. Then the result follows from this recursive relation in $j$ using the base case $j=1$ as a starting point where $\hom_{\Lambda}(\Lambda_{i}, \mathcal{L}_{0})=\hom_{\Lambda}(\Lambda_{i}, \sigma_{0})=\delta_{i, 0}$. In particular $\hom_{\Lambda}(\Lambda_i, \mathcal{L}_j)=\sum_{k \le j} \delta_{i,k}$.

To prove (3) we note that as $\Lambda_n = \mathcal{L}_n$ by part (1) the result is true for $j=n$ as $\lambda_n=1$. Then applying $\Hom_{\Lambda}(\Lambda_i, -)$ to the short exact sequence 
\[
 0  \rightarrow  \bigoplus_{k = j}^n \Lambda_k^{ \oplus \alpha_k-2} \oplus \Lambda_j  \rightarrow \Lambda_{j-1}  \rightarrow  \mathcal{L}_{j-1} \rightarrow  0
\]
produces the short exact sequence
\[
 0  \rightarrow  \bigoplus_{k = j}^n \Hom_{\Lambda}(\Lambda_i,\Lambda_k)^{\oplus \alpha_k-2} \oplus \Hom_{\Lambda}(\Lambda_i, \Lambda_j)  \rightarrow \Hom_{\Lambda}(\Lambda_i, \Lambda_{j-1})  \rightarrow  \Hom_{\Lambda}(\Lambda_i, \mathcal{L}_{j-1}) \rightarrow  0
\]
and by induction on $j$ we deduce that 
\[
\hom_{\Lambda}(\Lambda_i, \Lambda_{j-1}) = \sum_{k = j}^n ( \alpha_k-2)\hom_{\Lambda}(\Lambda_i,\Lambda_k) +  \hom_{\Lambda}(\Lambda_i,\Lambda_j)+ \hom_{\Lambda}(\Lambda_i, \mathcal{L}_{j-1}).
\]
If $j<i$ then by (2)  $\hom_{\Lambda}(\Lambda_i, \mathcal{L}_{j-1})=1$ and the result follows. If $j-1 \ge i$ it follows from (2) that $ \hom_{\Lambda}(\Lambda_i, \mathcal{L}_{j-1})=1$ and from this it follows  by induction that $\hom_{\Lambda}(\Lambda_i, \Lambda_{j-1})= \sum_{k = j}^n (\alpha_k-2)\lambda_k + \lambda_j +1= \lambda_{j-1}$. 
\end{proof}

We now define a quiver and relations, and then we prove that the path algebra of this quiver with relations is a presentation of $\Lambda_{r,a}$.

\begin{Definition}
We define a quiver $Q_{[\alpha_1, \dots, \alpha_n]}$ associated to a sequence of positive integers $\alpha_i \ge 2$, which in practice will correspond to the negative of the self intersection number of curves in a surface. We again recall the additional notation that $u_i=  \sum_{k<i} (\alpha_{k}-2)$ and $v_i= \sum_{k \le i} (\alpha_k-2)+1$. Then $Q_{[\alpha_1, \dots, \alpha_n]}$ is defined to be the quiver with $n+1$ vertices $Q_0= \{ 0,1,2, \dots, n\}$ and arrows $Q_1=\{ a_1, \dots, a_n, c_1, \dots, c_n, k_2, \dots, k_{v_{n}} \}$ with heads and tails defined as follows;
\begin{align*}
h(c_i)&=i & t(c_i)&=i-1\\
h(a_i)&=i-1 & t(a_i)&=i \\ 
h(k_i) &=0 & t(k_j)& = i  \text{ for $u_i +1<  j \le v_i$ }.
\end{align*}
We will also notate the arrow $a_1$ by $k_{1}$, and $C_i^j:=c_{i+1} \dots c_j$ for $i<j$.
\end{Definition}

\begin{Remark}
This is the same quiver as the reconstruction algebra with two arrows removed $c_{0}$ and $a_{0}$ removed. 
\end{Remark}

\begin{Definition}
For $1 \le i \le n$ consider the following elements of $\mathbb{C}Q_{[\alpha_1, \dots, \alpha_n]}$:
\begin{align*}
&\text{if $\alpha_i>2$:}
\qquad \left\{ \begin{array}{l}
a_{i} c_{ i} \\
k_j C_{0}^{i}   \qquad \text{for } u_i +1 < j < v_i \\
 k_{v_i}C_{0}^{ i} - c_{i+1}a_{i+1} \\
\end{array} \right. \\
&\text{if $\alpha_i=2$:} \qquad \quad  a_{i} c_{i}  -c_{ i+1} a_{i+1}.
\end{align*}
Define $I_{[\alpha_1, \dots, \alpha_n]}$ to be the two-sided ideal of $\mathbb{C}Q_{[\alpha_1, \dots, \alpha_n]}$ generated by these elements.
\end{Definition}
\begin{Remark}
Apart from at vertex 0 these are the same relations as those of the quotient of the reconstruction algebra by the arrows $c_0,a_0$.
\end{Remark}

We are now able to show that the algebra $\Lambda_{r,a}$ can be presented with these relations.

\begin{Proposition} \label{Lambda Presentation Proposition}
The algebra $\Lambda_{[\alpha_1, \dots, \alpha_n]}$ can be presented as $\frac{\mathbb{C} Q_{[\alpha_1, \dots, \alpha_n]}}{I_{[\alpha_1, \dots, \alpha_n]}}$.
\end{Proposition}

\begin{proof}
By Lemma \ref{Ext Quiver Calculation Lemma} we know that $\Lambda_{[\alpha_1,\dots, \alpha_n]}$ can be presented as the quiver $Q_{[\alpha_1, \dots, \alpha_n]}$ with the relations specified by $\alpha_i-1$ generating relations at vertex $i$. We now deduce that these relations can be presented by the claimed ideal, and we will prove this result by induction on $n$ where $\Lambda_{[\alpha_1, \dots, \alpha_n]}$ has $n+1$ vertices.

When $n=1$ then the space of all closed paths at vertex 1 has dimension 1 as $ \dim \Hom_{\Lambda_{[\alpha_1]}}(\Lambda_1, \Lambda_1)=\lambda_1=1$, and hence the space of closed paths is spanned by $e_1$. Then $\alpha_1-1$ generating relations are required. There are $\alpha_1-1$ closed paths $a_1c_1, k_1c_1, \dots, k_{\alpha_1-2}c_1$ generating all closed paths at vertex 1 and as $\lambda_1=1$ these must all equal a multiple of $e_1$ in $\Lambda_{[\alpha_1]}$. However,  as this presentation is of a basic algebra the relations are contained in paths of length $\ge 2$ so it must be the case that all these paths are equal to 0 in $\Lambda_{[\alpha_1]}$. This verifies that the relations can be presented as claimed when $n=1$.

We then consider $\Lambda:=\Lambda_{[\alpha_1, \dots, \alpha_n]}$ for $n>1$. By the induction hypothesis we know that $\Lambda_{[\alpha_2, \dots, \alpha_n]}$ and $\Lambda_{[\alpha_1, \dots, \alpha_{n-1}]}$ can be presented as claimed. There are related to $\Lambda$ by Lemma \ref{LessCurvesLemma}: if we consider the idempotents $f:=e_1+\dots +e_n$ and $e_n$ then $f\Lambda f \cong \Lambda_{[\alpha_2, \dots, \alpha_n]}$ and $\frac{\Lambda}{\Lambda e_n \Lambda} \cong \Lambda_{[\alpha_1, \dots, \alpha_{n-1}]}$. In particular, the presentation of $f \Lambda f$ shows us that we can assume the required relations are satisfied at vertices 2 to $n$ of $\Lambda$. Hence we are left only to show that the relations
\begin{align*}
\quad & k_j c_1=0 \qquad \text{for } 1 \le  j \le \alpha_1-2 \\
\quad & k_{\alpha_1-1}c_1 - c_{2}a_{2}=0.
\end{align*}
 hold at vertex 1, where we write $a_1=k_1$.
From the presentation of $\frac{\Lambda}{\Lambda e_n \Lambda}$ we can deduce that the required relations for $\Lambda$ at vertex one are satisfied up to elements in the kernel of the quotient $\Lambda \rightarrow \frac{\Lambda}{\Lambda e_n \Lambda}$. Hence we can deduce the following relations hold
\begin{align*}
\quad & k_j c_1=p_j \qquad \text{for } 1 \le  j \le \alpha_1-2 \\
\quad & k_{\alpha_1-1}c_1 - c_{2}a_{2}=p_{\alpha_1-1}.
\end{align*}
for elements $p_i \in f \Lambda f$ and in the kernel of the quotient $\Lambda \rightarrow \frac{\Lambda}{\Lambda e_n \Lambda}$. We now aim to show that after a change of basis each $p_i$ can be assumed to be zero. 

We first note that any map of the form $a \mapsto a +pa=(1+p)a$  is invertible for any $a \in Q_1$ and $p$ a path in the quiver. This is as the algebra is nilpotent; for any element of the path algebra $p$ there is some $n$ such that $p^i=0$ for $i>n$. This can be seen as the algebra is finite dimensional, and as the relations are homogeneous (where $c_i$ are degree 0 and $k_i, a_i$ of degree 1) it has a homogeneous basis. Hence there exists a maximal degree at which there exist nonzero elements. Then any loop in the path algebra has positive degree so vanishes at some power.

We now consider the possibilities for a path $p_i$. We begin by noting that any term in $p_i$ that ends in $c_1$ can be removed by change of basis. Such a term is of the form $qc_1$ for some $q$ and the change of basis $k'_i:=k_i-q$ rewrites the equation without the $qc_1$ term and without altering any other relations at any other vertex. Then we set $k_i:=k_i'$. Hence we may assume that each $p_i$ contains no terms that can be expressed to end with $c_1$.

By considering the presentation of the algebra $f \Lambda f \cong \Lambda_{[\alpha_2,\dots, \alpha_n]}$  and the fact that $p$ is in the kernel of $\Lambda \rightarrow \frac{\Lambda}{\Lambda e_n \Lambda}$we can assume that each term in $p$ can be written starting with $c_2 \dots c_n$. Further, as we assume each path does not end with $c_1$ we assume it ends with $a_2$. Hence we assume each $p_i$ is of the form $(c_2 \dots c_n)q_i (a_2)$. Now consider the relation 
\[
k_{\alpha_1-1}c_1-c_2a_2-(c_2 \dots c_n) q (a_2)= k_{\alpha_1-1}c_1-c_2(1 + c_3 \dots c_n q ) a_2
\]
for $q=q_{\alpha_i-1} \in f\Lambda f$. Then we can perform the change of basis $a_2'=(1 + c_3 \dots c_n q )a_2$.  This then affects one other relation at vertex 2. If $\alpha_2 \neq 2$ then the relation $a_2c_2=0$ holds and so we verify that $a'_2c_2=(1 + c_3 \dots c_n q )a_2c_2=0$. Hence we set $a_2:=a'_2$ as a change of basis, and now can assume $p_{\alpha_1-1}=0$. If $\alpha_2=2$ then $a_2c_2-c_3a_3=0$ and we must ensure $a'_2c_2-c_3a_3=0$ still holds. However, 
\begin{align*}
a'_2c_2&=(1 + c_3 \dots c_n q )a_2c_2\\
&=(1 + c_3 \dots c_n q )c_3a_3\\
&=c_3(1 + c_4 \dots c_n q c_3)a_3
\end{align*}
so we must make the base change $a_3'=(1 + c_4 \dots c_n q c_3)a_3$ analogous to the base change for $a_2$. Then continuing this process we will make successive base changes $a_i'$ such that the relations are satisfied. This process will always conclude; eventually we reach either a vertex $j$ with $\alpha_j \neq 2$ or the final vertex $n$. Then we relabel $a_i:=a_i'$ where necessary, and having done this we can assume still that all relations hold at vertices $2, \dots, n$ and also that the relation $k_{\alpha_1-1}c_1-c_2a_2=0$ holds at vertex 1.

Finally we consider the relations $ k_j c_1  -p_j  $ at vertex 1. From the relations in $f\Lambda f$ we can write any term in $p$ such that it ends with $(c_2a_2)$. But now thanks to the relation we have already verified we can assume that $c_2a_2=k_{\alpha_1-1}c_1$. Hence, using the fact that we can remove any term ending with $c_1$ by base change, we are done and can conclude that after all the appropriate base changes $p_i=0$ for all $i$ and the required relations are satisfied. As these are the correct number of generating relations then this is a presentation of the algebra.
\end{proof}

\begin{Example}[$r=17, \, a=5$]  The algebra $\Lambda_{[4,2,3]}=\Lambda_{17,5}$ can be presented as the path algebra of the following quiver with relations.
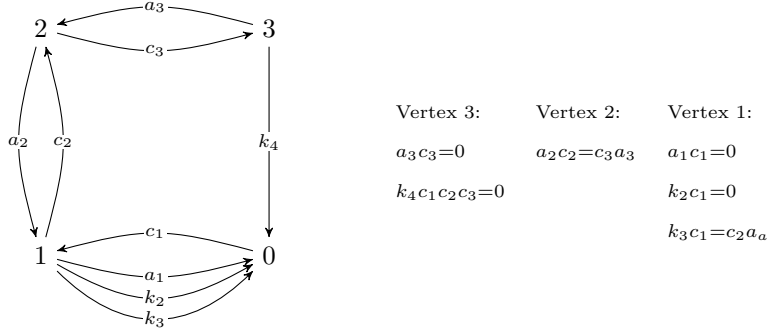
\begin{align*}
\begin{aligned}
\begin{tikzpicture} [bend angle=45, looseness=1.2]
\node (C1) at (-3,-3)  {$1$};
\node (C2) at (-3,0)  {$2$};
\node (C3) at (0,0) {$3$};
\node (C*) at (0,-3) {$0$};
\draw [->,bend right=15] (C*) to node[gap] {$\scriptstyle{c_{1}}$} (C1);
\draw [->,bend right=15] (C1) to node[gap] {$\scriptstyle{a_{1}}$} (C*);
\draw [->,bend right=15] (C1) to node[gap] {$\scriptstyle{c_{2}}$} (C2);
\draw [->,bend right=15] (C2) to node[gap] {$\scriptstyle{a_{2}}$} (C1);
\draw [->,bend right=45] (C1) to node[gap] {$\scriptstyle{k_{3}}$} (C*);
\draw [->,bend right=30] (C1) to node[gap] {$\scriptstyle{k_{2}}$} (C*);
\draw [->,bend right=15] (C2) to node[gap] {$\scriptstyle{c_{3}}$} (C3);
\draw [->,bend right=15] (C3) to node[gap] {$\scriptstyle{a_{3}}$} (C2);
\draw [->,bend right=0] (C3) to node[gap] {$\scriptstyle{k_{4}}$} (C*);
\end{tikzpicture}
\end{aligned}
& \qquad \qquad
\begin{aligned}
&\scriptstyle{\text{Vertex 3:}}\\
&\scriptstyle{a_3c_3}\scriptstyle{=0} \\
&\scriptstyle{k_4 c_1c_2c_3}\scriptstyle{=0} \\
\\
\end{aligned}
\quad
\begin{aligned}
&\scriptstyle{\text{Vertex 2:}}\\
 &\scriptstyle{a_2c_2} \scriptstyle{=c_3a_3}  \\
\\
\\
\end{aligned}
 \quad
\begin{aligned}
&\scriptstyle{\text{Vertex 1:}}\\
&\scriptstyle{a_1c_1 }\scriptstyle{= 0 }\\
&\scriptstyle{k_2  c_1} \scriptstyle{= 0}\\
&\scriptstyle{k_3  c_1} \scriptstyle{= c_2a_a}
\end{aligned}
\end{align*}
\end{Example}

\subsection{Description of the Kn\"{o}rrer invariant algebra $K_{r,a}$} \label{The KI algebra}
Recall the Kn\"{o}rrer invariant algebra from Definition \ref{KIdefinition}: $K_{r,a}:= e_{0} \Lambda_{r,a} e_{0}$. In this section we explicitly describe this algebra in terms of generators and relations. To do this we introduce a noncommutative singularity $\kappa_{r,a}$ in terms of generators and relations, and by considering the indecomposable ideals $I_i$ of $\kappa_{r,a}$ we produce a noncommutative resolution $\End_{\kappa_{r,a}}(\bigoplus I_i)$. We then prove that $\End_{\kappa_{r,a}}(\bigoplus I_i) \cong \Lambda_{r,a}$, which simultaneously shows that  $\kappa_{r,a} \cong K_{r,a}$ and so provides an explicit presentation of $K_{r,a}$.

\begin{Definition} \label{kappa Definition}
Define $\kappa_{r,a}$ to be the quotient of the free algebra in $l$ generators $\mathbb{C} \langle z_1 \dots, z_l \rangle$ by the relations 
\[ z_i z_j=0 \text{ if $i<j$}
\]
and 
\[
z_i \left(  z_i^{\beta_i-2} \right) \left(  z_{i+1}^{\beta_{i+1}-2} \right)  \dots{ \left(  z_{j-1}^{\beta_{j-1}-2} \right)} \left(  z_j^{\beta_j-2} \right) z_j=0 \text{ for $j \le i$}
\]
where $l$ and the $\beta_i$ are defined by the Hirzebruch-Jung continued fraction expansion $r/(r-a)=[\beta_1,\dots, \beta_l]$.
\end{Definition}
We introduce the additional notation that $\kappa^{[\beta_1, \dots, \beta_l]}=\kappa_{r,a}=\kappa_{[\alpha_1, \dots, \alpha_n]}$ if $r/(r-a)=[\beta_1, \dots, \beta_l]$ and $r/a=[\alpha_1, \dots, \alpha_n]$.

\begin{Lemma} \label{Kappa Properties Lemma}
The following are properties of $\kappa_{r,a}=\kappa^{[\beta_1, \dotsm \beta_l]}$.
\begin{enumerate}
\item Any nonzero monomial in $\kappa^{[\beta_1, \dots, \beta_l]}$ can be expressed uniquely in the form 
\[
z_l^{b_l}z_{l-1}^{b_{l-1}} \dots z_1^{b_1}
\]
for $0 \le b_i < \beta_i$ such that there is no $j<i$ with $b_i=\beta_i-1$, $b_j=\beta_j-1$, and $b_k=\beta_k-2$ for all $j<k<i$.
\item The highest degree of a nonzero monomial in the generators is $\sum_{i=1}^l (\beta_i-2)+1$.

\item 
The dimension of the algebra $\kappa_{r,a}$ is $r$.
\end{enumerate}
\end{Lemma}
\begin{proof}
As $\kappa_{r,a}$ has monomial relations a monomial is zero if and only if it contains a relation as a submonomial. This implies nonzero monomials are exactly those written in the form specified in (1).

We now consider (2). Let $m$ be a highest degree nonzero monomial. By part (1) we can write it in the form
\[
z_l^{b_l}z_{l-1}^{b_{l-1}} \dots z_1^{b_1}
\]
for $0 \le b_i < \beta_i$ and such that there is no $i<j$ with $b_i=\beta_i-1$, $b_j=\beta_j-1$, and $b_k=\beta_k-2$ for all $j<k<i$. We now also show that $|m| \le \sum(\beta_i-2)+1$.  Suppose $b_i<\beta_i-3$ for some $i$, then as there are no relations occurring that involve powers of $z_i$ lower than $\beta_i-2$ we can increase $b_i$ by $1$ and find a higher degree monomial. Hence we can assume $\beta_i-3 \le b_i \le \beta_i-1$ for all $i$. Hence any $b_i$ has value $\beta_i-1, \beta_i-2$, or $\beta_i-3$. We let $r_1,r_2$, and $r_3$ denote the respective numbers of such $b_i$, and the degree of $m$ is $\sum \beta_i -r_1-2r_2-3r_3=\sum(\beta_i-2)+r_1-r_3$. The monomial $m$ is zero if it contains a relation as a submonomial, and this will occur whenever there is some $i<j$ such that the path has $b_i=\beta_i-1$, $b_j=\beta_j-1$ and $b_{k}=\beta_k-2$ for all $i<k<j$. Hence, to be nonzero, $r_3 \ge r_1-1$ so that some value of $\beta_k-3$ can occur between any two $\beta_i-1$ and $\beta_j-1$ with $i<j$. Hence the degree of $m$ is $\le \sum(\beta_i-2)+1$. Finally, $z_l^{b_l-2}z_{l-1}^{b_{l-1}-2} \dots z_{2}^{b_{2}-2}z_1^{b_1-1}$ is a nonzero monomial of degree $\sum(\beta_i-2)+1$. 

To prove (3) recall that as $\kappa_{r,a}$ is presented with monomial relations it has a vector space basis given by monomials.  Hence the number of nonzero monomials is the dimension and the nonzero monomials in $\kappa_{r,a}$ are exactly those listed in part (1). We now count such nonzero monomials. 

The number of such monomials with $b_1=0$ is $\dim \kappa^{[\beta_{2},\dots, \beta_n]}$. Then the nonzero monomials of $\kappa_{r,a}$ are spanned by monomials with $b_1=0$ multiplied by $z_1^j$ for some $0 \le j<b_1$, however some of these are zero. From the relations, the only power of $z_1$ that could annihilate a monomial with $b_1=0$  is $z_1^{\beta_1-1}$, and we now count how many elements with $b_1=0$ are annihilated on multiplication by $z_1^{\beta_1-1}$. Monomials with $b_1=0$ annihilated by $z_1^{\beta_{1}-1}$ are in bijection with nonzero monomials with $b_1=b_{2}=0$: given such a nonzero monomial multiple by the highest power of $z_{2}$ that doesn't produce 0. These are exactly the monomials with $b_1=0$ that vanish on multiplication by $z_1^{\beta_1-1}$. There are $\dim \kappa^{[\beta_{2},\dots, \beta_n]}$  monomials with $b_1=0$ and $\dim \kappa^{[\beta_{3},\dots, \beta_1]}$ monomials with $b_{1}=b_2=0$. Hence we can deduce 
\[
\dim \kappa^{[\beta_1, \dots, \beta_n]}=\beta_1 \dim \kappa^{[\beta_{2},\dots, \beta_n]}- \dim \kappa^{[\beta_{3}, \dots, \beta_1]}.
\]

Then by induction
\[
\frac{\dim \kappa^{[\beta_{i}, \dots, \beta_l]}}{\dim \kappa^{[\beta_{i+1}, \dots, \beta_l]}}=[\beta_i, \dots, \beta_l]
\]
as
\begin{align*}
\frac{\dim \kappa^{[\beta_{i}, \dots, \beta_l]}}{\dim \kappa^{[\beta_{i+1}, \dots, \beta_l]}}& = \frac{\beta_i \dim \kappa^{[\beta_{i+1},\dots, \beta_l]}- \dim \kappa^{[\beta_{i+2}, \dots, \beta_l]}}{\dim \kappa^{[\beta_{i+1}, \dots, \beta_l]}} \\
& = \beta_i - \frac{\dim \kappa^{[\beta_{i+2}, \dots, \beta_l]}}{\dim \kappa^{[\beta_{i+1}, \dots, \beta_l]}} \\
&=\beta_i -\frac{1}{[\beta_{i+1}, \dots , \beta_l]}=[\beta_i, \dots, \beta_l]
\end{align*}
where the base case is $\dim \kappa^{[]}= \dim \mathbb{C}=1$ and $\dim \kappa^{[\beta_1]} = \dim \mathbb{C}[z]/(z^{\beta_1})=\beta_1$.

We also note that as  $\dim \kappa^{[\beta_i, \dots, \beta_l]}=\beta_i \dim \kappa^{[\beta_{{i+1}},\dots, \beta_l]}- \dim \kappa^{[\beta_{i+2}, \dots, \beta_l]}.$ and  $\dim \kappa^{[]}$=1 any consecutive pair of $\dim \kappa^{[\beta_i, \dots, \beta_l]},  \dim \kappa^{[\beta_{{i+1}},\dots, \beta_l]}$ are necessarily coprime. As such the dimensions are the partial terms in the calculation of the Hirzebruch-Jung continued fraction expansion of $r/(r-a)$ and $\dim \kappa_{r,a}=\dim \kappa^{[\beta_1, \dots, \beta_n]}=r$.
\end{proof}

We now note a distinguished set of $\kappa_{r,a}$ modules, the indecomposable monomial ideals, and we show that their endomorphism ring recovers the algebra $\Lambda_{r,a}$.  We recall that in a monomial algebra all indecomposable monomial ideals are defined by a single monomial. Recall the dual fraction $[\beta_1, \dots, \beta_l]$.  We define the monomial
\[
m_n:=z_l^{\beta_l-2}z_{l-1}^{\beta_{l-1}-2} \dots z_{2}^{\beta_{2}-2}z_1^{\beta_1-1} \in \kappa_{r,a},
\]
and the submonomials
\[
m_i:=z_j^{c}z_{j-1}^{\beta_{j-1}-2} \dots z_{2}^{\beta_{2}-2}z_1^{\beta_1-1}
\]
where $\sum_{k=1}^{j-1} (\beta_k-2)+c+1 =i$. In particular $0 \le i \le n = \sum_{i=0}^l (\beta_i-2)+1$.

\begin{Proposition} \label{Prop:Number of Ideals}
Any indecomposable monomial ideal of $\kappa_{r,a}$ is isomorphic as a left $\kappa_{r,a}$-module to a monomial ideal generated by $m_i$ for some $i$.
In particular, up to isomorphism as left $\kappa_{r,a}$-modules there are $n+1$  indecomposable monomial ideals.
\end{Proposition}
\begin{proof}
Firstly, any indecomposable monomial ideal is of the form $(m)$ for some monomial $m$, and as a left module such an ideal defined up to isomorphism by the set of monomials $q$ such that $qm=0$. By Lemma \ref{Kappa Properties Lemma} we may assume that $m=z_l^{b_l}z_{l-1}^{b_{l-1}} \dots z_1^{b_1}$ for some $0 \le b_i < \beta_i$ such that there is no $i<j$ with $b_i=\beta_i-1$, $b_j=\beta_j-1$, and $b_k=\beta_k-2$ for all $j<k<i$. Then either $(m)=(1)$ or there is a maximal $i$ such that $b_i$ is nonzero, $l$ say. Then there exists a minimal $c$ such that $z_i^c m=0$. Then $(m) \cong (m_j)$ where $j=\sum (\beta_k-2)+c-1$.
\end{proof}

\begin{Definition}
Define the monomial ideals
\[
I_i:=(m_i).
\]
\end{Definition} 
In particular $I_0 \cong \kappa_{r,a}$ as a left $\kappa_{r,a}$-module.

\begin{Proposition} \label{Monomial ideals Proposition}
The monomial ideal 
\[
I_i=(m_i)=(z_j^{c}z_{j-1}^{\beta_{j-1}-2} \dots z_{2}^{\beta_{2}-2}z_1^{\beta_1-1})
\]
is isomorphic as a left $\kappa_{r,a}$-module to the quotient module 
\[
M_i:= \frac{\kappa_{r,a}}{J_i}\cong \kappa_{[\alpha_{i+1}, \dots, \alpha_n]}
\]
where $J_i:=\left\langle z_1, \dots, z_{j-1},  z_k^{\beta_k-1}  z_{k-1}^{\beta_{k-1}-2}  \dots z_{j+1}^{\beta_{j+1}-2}   z_{j}^{\beta_j-c-2} \text{ for $j \le k \le l$} \right\rangle$. In particular, $\dim I_i = \dim M_i = \lambda_i$.
\end{Proposition}
\begin{proof}
The monomial ideal $I_i=(m_i)=(z_j^{c}z_{j-1}^{\beta_{j-1}-2} \dots z_{2}^{\beta_{2}-2}z_1^{\beta_1-1})$ is defined as a left $\kappa_{r,a}$ module by the monomials $q \in \kappa_{r,a}$ such that $q m_i=0$. From the relations it is clear that the generating set of such monomials is exactly 
\[
 z_1, \dots, z_{j-1}, z_{j}^{\beta_j-c-1}, \text{ and } z_k^{\beta_k-1}  z_{k-1}^{\beta_{k-1}-2}  \dots z_{j+1}^{\beta_{j+1}-2}   z_{j}^{\beta_j-c-2} \text{ for $j+1 \le k \le l$.}
\]
Hence
\[
(m_i) \cong \frac{\kappa_{r,a}}{J_i}
\]
and it is clear that 
\[
\kappa_{r,a}/J_i \cong \kappa^{[\beta_j-c,\beta_{j+1}, \dots, \beta_l]}
\]
as an algebra. By Lemma \ref{HJfractionCorollary} the Hirzebruch-Jung continued fraction expansion $[\beta_j-c,\beta_{j+1}, \dots, \beta_l]$ is dual to the fraction expansion $[\alpha_{i+1}, \dots, \alpha_n]$, and $\kappa_{[\alpha_{i+1}, \dots, \alpha_n]} \cong \kappa^{[\beta_j-c,\beta_{j+1}, \dots, \beta_l]}$ by Lemma \ref{KI Presentation Lemma}. Hence
\[
\kappa_{r,a}/J_i \cong \kappa_{[\alpha_{i+1}, \dots, \alpha_n]},
\]
so $\dim I_i = \dim  \kappa_{[\alpha_{i+1}, \dots, \alpha_n]} = \lambda_i$ by Lemma \ref{Rank Calculation Lemma}.
\end{proof}

The following lemma aids us in proving the main result of this section.

\begin{Lemma} \label{Inclusions Lemma}
For $j>k$ there are inclusions
\[
\Hom_{\kappa_{r,a}}(M_j,M_j) \subset \Hom_{\kappa_{r,a}}(M_k,M_j) \subset M_j
\]
and 
\[
\Hom_{\kappa_{r,a}}(M_j,M_j) \subset \Hom_{\kappa_{r,a}}(M_j,M_k) \subset \Hom_{\kappa_{r,a}}(M_k,M_k) \subset M_k.
\]
\end{Lemma}
\begin{proof}
In general, the surjection $\kappa_{r,a} \twoheadrightarrow M_i $ defines an inclusion of sets $\Hom_{\kappa_{r,a}}(M_i,M_j) \subset \Hom_{\kappa_{r,a}}(\kappa_{r,a},M_j) \cong M_j$.

For $j>k$ applying $\Hom_{\kappa_{r,a}}(-,M_k)$ to the surjection $M_k \twoheadrightarrow M_j$ induces an inclusion of sets $\Hom_{\kappa_{r,a}}(M_j,M_k) \hookrightarrow \Hom_{\kappa_{r,a}}(M_k,M_k)$, applying $\Hom_{\kappa_{r,a}}(-,M_j)$ to the surjection $M_k \twoheadrightarrow M_j$ induces an inclusion of sets $\Hom_{\kappa_{r,a}}(M_j,M_j) \hookrightarrow \Hom_{\kappa}(M_k,M_j)$, and applying $\Hom_{\kappa_{r,a}}(I_j,-)$ to the inclusion of modules $I_j \hookrightarrow I_k$ induces an inclusion of sets $\Hom_{\kappa_{r,a}}(I_j,I_j) \subset \Hom_{\kappa_{r,a}}(I_j,I_k)$. Putting this together proves the result.
\end{proof}

\begin{Theorem} \label{Endomorphism Algebra Isomorphism Theorem}
There is an isomorphism of $\mathbb{C}$-algebras $\Lambda_{r,a} \cong \End_{\kappa_{r,a}}(\bigoplus_{i=0}^n I_i)$.
\end{Theorem}

\begin{proof}

We will define a morphism $\phi$ from $\Lambda_{r,a}$ to $\End_{\kappa_{r,a}}(\bigoplus I_i)$ and then show that it is bijective to conclude that it is an isomorphism. Recall that there are isomorphisms $I_i \cong M_i \cong \kappa_{r,a}/J_i$, and we will make use of the equivalence  $\End_{\kappa_{r,a}}(\bigoplus_{i=0}^n M_i) \cong \End_{\kappa_{r,a}}(\bigoplus_{i=0}^n I_i)$ to describe the morphism.

We define the map $\phi:\Lambda_{r,a} \rightarrow \End_{\kappa_{r,a}}(\bigoplus_{i=0}^n M_i)$ by sending the vertex idempotents $e_i$ to the identity maps $id_{M_i}:M_i \rightarrow M_i$, the arrows $c_i:i-1 \rightarrow i$ to the quotient maps $M_{i-1} \twoheadrightarrow M_{i}$, the arrows $a_i:i \rightarrow {i-1}$ to the multiplication maps $z_{v_i}: M_{i} \rightarrow M_{i-1}$, and the arrows $k_j: i \rightarrow 0$ to the multiplication maps $z_j:M_i \rightarrow M_{0}$.We recall that  $u_i:=  \sum_{k<i} (\alpha_{k}-2)$, $v_i:= \sum_{k \le i} (\alpha_k-2)+1$ and $t(k_j):=t(j)$ is defined to be the unique $i$ such that $u_i+1 < j \le v_i$.

Under the alternative presentation $\End_{\kappa_{r,a}}(\bigoplus_{i=0}^n I_i)$ the map $\phi$ sends the vertex idempotents $e_i$ to the identity maps $id_{I_i}:I_i \rightarrow I_i$, the arrows $c_i:i-1 \rightarrow i$ to the multiplication maps $z_{v_i}:I_{i-1} \twoheadrightarrow I_{i}$, the arrows $a_i:i \rightarrow {i-1}$ to the inclusion maps $I_{i} \rightarrow I_{i-1}$, and the arrows $k_j: i \rightarrow 0$ to multiplication by the element $z_j/z_{v_i}:M_i \rightarrow M_{0}$.

The generating relations among the paths in $\Lambda_{r,a}$ are satisfied by the corresponding maps in $\End_{\kappa_{r,a}}(\bigoplus_{i=0}^n M_i)$, and hence $\phi$ is an algebra homomorphism. We will show that homomorphism $\phi$ is bijective and hence an isomorphism. 

Let $\phi_{i,j}$ denote the restriction of the algebra homomorphism to a linear map $\phi_{i,j}: e_i \Lambda e_j \rightarrow \Hom_{\kappa_{r,a}}(M_i,M_j)$.  We now show that $\phi_{i,j}$ is a bijection for $i,j \ge k$ by induction on $n-k$. The base case $k=n$ is clear as $\phi_{n,n}:e_n \Lambda e_n \cong \mathbb{C} \rightarrow \Hom_{\kappa_{r,a}}(M_n,M_n)\cong \mathbb{C}$ is nonzero and hence a bijection. 

We now consider the case $k<n$. The morphism $\phi$ restricts to the linear map $\bigoplus_{i,j \ge k} \phi_{i,j}$ 
\[
\left( \begin{array}{c|ccc} 
\phi_{k,k} & \phi_{k,k+1} & \dots & \phi_{k,n} \\
\hline \\
\phi_{k+1,k} & \phi_{k+1,k+1} & \dots & \phi_{k+1,n}\\
\vdots & \vdots & \ddots & \vdots \\
\phi_{n,k} & \phi_{n,k+1} & \dots & \phi_{n,n}
\end{array} \right)
\]
and by the induction hypothesis we can assume that $\phi_{i,j}$ is a bijection if $i,j>k$. We will show that $\phi_{k,j}$, $\phi_{k,k}$ and $\phi_{j,k}$ for $j > k $ are bijective in turn.

We first consider $\phi_{k,j}$. As $k<j$ we have the inclusions
\[
\Hom_{\kappa_{r,a}}(M_j,M_j) \subset \Hom_{\kappa_{r,a}}(M_k,M_j) \subset M_j
\]
and by the induction hypothesis $\Hom_{\kappa_{r,a}}(M_j,M_j)=e_j\Lambda e_j$ so $\dim \Hom_{\kappa_{r,a}}(M_j,M_j) = \lambda_j$ by Lemma \ref{Module Dimensions Lemma}. As $\dim M_j = \lambda_j$ by Lemma \ref{Monomial ideals Proposition}, it follows that the series of inclusions are a series of bijections $\Hom_{\kappa_{r,a}}(M_j,M_j) \cong \Hom_{\kappa_{r,a}}(M_k,M_j) \cong M_j \cong e_j\Lambda e_j$. This implies that every map in $\Hom_{\kappa_{r,a}}(M_k,M_j)$ factors through $\Hom(M_j,M_j)$. Then for any $f \in \Hom_{\kappa_{r,a}}(M_k,M_j)$ there is a corresponding $p$ in $e_j \Lambda e_j$ such that $\phi_{j,j}(p)=f$ and hence $\phi_{k,j}(C_k^jp)=\phi_{k,j}(C_k^j)\phi_{j,j}(p)=f \in \Hom_{\kappa_{r,a}}(M_k,M_j)$. Hence $\phi_{k,j}$ is a surjection, and hence as $\dim e_k \Lambda e_j = \dim \Hom_{\kappa_{r,a}}(M_k,M_j)= \lambda_j$ it is a bijection.

We then consider $\phi_{k,k}$ and will show first it is surjective and then it is bijective for dimension reasons. We consider the inclusion of sets \[
M_{k+1} \cong \Hom_{\kappa_{r,a}}(M_{k+1},M_{k+1}) \subset \Hom_{\kappa_{r,a}}(M_k,M_k) \subset M_k.
\] Consider an element $f \in \Hom_{\kappa_{r,a}}(M_k,M_k)$. If $f = id$ then $f = \phi_{k,k}(e_k)$. Otherwise $f$ corresponds to multiplication by a non-identity element $f(1)=m \in M_k$. Then $m=m' z_j$ for some $m'$ with nonzero image in $M_{k+1}$ and $j \ge v_{k}$. Then there is some path $p \in e_{k+1} \Lambda e_{k+1}$ such that $\phi_{k+1,k+1}(p)=m'$ and some path $q \in e_{k+1}\Lambda e_{k}$ such that $\phi_{k+1,k}(q)=z_j$ so $\phi_{k,k}(c_{k}pq)=m'z_j=m$. Hence $\phi_{k,k} \circ \iota$ is surjective. As $\lambda_k=\dim e_k \Lambda e_k \le \dim M_k = \lambda_k$ by Lemma \ref{Module Dimensions Lemma} and Proposition \ref{Monomial ideals Proposition} the map $\phi_{k,k}$ is in fact a bijection.

We are left to show that $\phi_{j,k}$ is a bijection for $j>k$. As $j>k$ there are a series of inclusions
\[
\Hom_{\kappa_{r,a}}(M_j,M_j) \subset \Hom_{\kappa_{r,a}}(M_j,M_k) \subset \Hom_{\kappa_{r,a}}(M_k,M_k)
\]
by Lemma \ref{Inclusions Lemma}.

Any element $f \in \Hom_{\kappa_{r,a}}(I_j,I_k) \cong \Hom_{\kappa_{r,a}}(M_j,M_k)$ is uniquely defined by an element $inc(f) \in \Hom_{\kappa_{r,a}}(I_k,I_k) \cong \Hom_{\kappa_{r,a}}(M_k,M_k) $. There is a unique loop $p' \in e_k \Lambda e_k$ such that $\phi_{k,k}(p')=inc(f) \in \Hom_{\kappa_{r,a}}(I_k,I_k)$. Then $A_j^kp' \in e_j \Lambda e_k$ is the unique element that maps to $f \in \Hom_{\kappa_{r,a}}(I_j,I_k)$; $\phi_{j,k}(A_j^kp')= \phi_{j_k}(A_j^k)\phi(p')=inc(f)$.

Hence $\phi_{i,j}$ is a bijection for all $i,j \ge k$, so by induction $\phi_{i,j}$ is bijective for all $i,j \ge 0$ and so $\phi$ is an isomorphism.
\end{proof}

Considering the idempotent $e_0$ produces an isomorphism $e_0 \Lambda_{r,a} e_0 \cong e_0 \Hom_{\kappa_{r,a}}(\bigoplus I_i) e_0$ that yields the following immediate corollary.
\begin{Corollary} \label{KI Presentation Lemma}
There is a $\mathbb{C}$-algebra isomorphism $K_{r,a} \cong \kappa_{r,a}$.
\end{Corollary}

The following is an immediate corollary of Lemmas \ref{Kappa Properties Lemma} and \ref{KI Presentation Lemma} and relates invariants of $K_{r,a} \cong \kappa_{r,a}$ to the combinatorics of the continued fractions $r/a=[\alpha_1, \dots, \alpha_n]$ and $r/(r-a)=[\beta_1, \dots, \beta_l]$.

\begin{Proposition}\label{Properties of KI Algebras Proposition}
The \emph{Kn\"{o}rrer invariant algebra} $K_{r,a}$ has the following properties:
\begin{enumerate}
\item It has dimension $r$, the order of the group defining the corresponding cyclic surface quotient singularity.
\item  The proper monomial left ideal of $K_{r,a}$ of largest $\mathbb{C}$-dimension  has dimension $a$.
\item It has monomial of highest degree $n$, the number of exceptional curves in the minimal resolution of the corresponding cyclic surface quotient singularity.
\item It is generated by $l$ elements, where $l$ is 2 less than the embedding dimension of the corresponding cyclic surface quotient singularity.
\end{enumerate}
\end{Proposition}
\begin{proof}
After recalling $K_{r,a} \cong \kappa_{r,a}$ by Lemma \ref{KI Presentation Lemma}, part (4) is immediate as  $\kappa_{r,a}$ has $l$ generators, parts (1) and (3) follow from Proposition Lemma \ref{Kappa Properties Lemma} where $1+\sum_{i=1}^l(\beta_i-2)=n$  due to Lemma \ref{Riemenschneider Duality Theorem}, and part (2) is a consequence of Proposition \ref{Monomial ideals Proposition}.
\end{proof}

\begin{Example}[r=17, a=5] Below we include the presentation of $K_{[4,2,3]}=K_{17,5}$ as the algebra realised from the closed paths at vertex $0$ of $\Lambda_{17,5}$.
In particular the Hirzebruch-Jung continued fractions are $[\alpha_1, \alpha_2, \alpha_3]=[4,2,3]=\frac{17}{5}$ and $[\beta_1, \dots, \beta_4]=[2,2,4,2]=\frac{17}{12}$. 
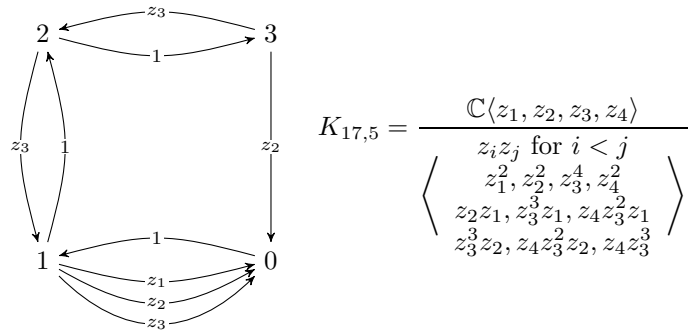
\begin{align*}
\begin{aligned}
\begin{tikzpicture} [bend angle=45, looseness=1.2]
\node (C1) at (-3,-3)  {$1$};
\node (C2) at (-3,0)  {$2$};
\node (C3) at (0,0) {$3$};
\node (C*) at (0,-3) {$0$};
\draw [->,bend right=15] (C*) to node[gap] {$\scriptstyle{1}$} (C1);
\draw [->,bend right=15] (C1) to node[gap] {$\scriptstyle{z_{1}}$} (C*);
\draw [->,bend right=15] (C1) to node[gap] {$\scriptstyle{1}$} (C2);
\draw [->,bend right=15] (C2) to node[gap] {$\scriptstyle{z_3}$} (C1);
\draw [->,bend right=45] (C1) to node[gap] {$\scriptstyle{z_3}$} (C*);
\draw [->,bend right=30] (C1) to node[gap] {$\scriptstyle{z_2}$} (C*);
\draw [->,bend right=15] (C2) to node[gap] {$\scriptstyle{1}$} (C3);
\draw [->,bend right=15] (C3) to node[gap] {$\scriptstyle{z_3}$} (C2);
\draw [->,bend right=0] (C3) to node[gap] {$\scriptstyle{z_2}$} (C*);
\end{tikzpicture}
\end{aligned}
& \quad 
\begin{aligned}
K_{17,5}= \frac{\mathbb{C}\langle z_1, z_2, z_3, z_4 \rangle }{\left\langle 
\begin{array}{c}
 z_iz_j \text{ for $i<j$} \\
z_1^2, z_2^2, z_3^4,z_4^2 \\
z_2z_1, z_3^3z_1, z_4z_3^2z_1 \\
z_3^3z_2, z_4z_3^2z_2, z_4z_3^3
\end{array} 
\right\rangle}
\end{aligned}
\end{align*}
\end{Example}

\subsection{Example} \label{Algebra:Examples}

As the Kn\"{o}rrer invariant algebra is presented with monomial relations, it is uniquely determined by its nonzero monomials. We present the algebra as a left module over itself by a diagram where the root of the diagram represents the identity element $1$, the labelled arrows $i$ represent multiplication on the left by $z_i$, and the nodes are in one to one correspondence with the nonzero monomials. This uniquely determines the algebra. We will call this the \emph{monomial diagram}.

\begin{Example}[r=17, a=5]
The monomial diagram for $K_{17,5}$ is
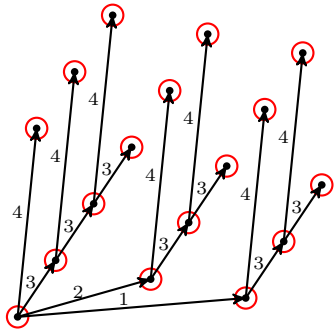
\begin{center}
\begin{tikzpicture}

\fill[color=black] (0,0) circle (0.05cm);
\draw[thick,red] (0,0) circle (4pt);
\fill[color=black] (0.25,2.5) circle (0.05cm);
\draw[thick,red] (0.25,2.5) circle (4pt);
\fill[color=black] (0.5,0.75) circle (0.05cm);
\draw[thick,red] (0.5,0.75) circle (4pt);
\fill[color=black] (0.75,3.25) circle (0.05cm);
\draw[thick,red] (0.75,3.25) circle (4pt);
\fill[color=black] (1,1.5) circle (0.05cm);
\draw[thick,red] (1,1.5) circle (4pt);
\fill[color=black] (1.25,4) circle (0.05cm);
\draw[thick,red] (1.25,4) circle (4pt);
\fill[color=black] (1.5,2.25) circle (0.05cm);
\draw[thick,red] (1.5,2.25) circle (4pt);
\fill[color=black] (1.75,0.5) circle (0.05cm);
\draw[thick,red] (1.75,0.5) circle (4pt);
\fill[color=black] (2,3) circle (0.05cm);
\draw[thick,red] (2,3) circle (4pt);
\fill[color=black] (2.25,1.25) circle (0.05cm);
\draw[thick,red] (2.25,1.25) circle (4pt);
\fill[color=black] (2.5,3.75) circle (0.05cm);
\draw[thick,red] (2.5,3.75) circle (4pt);
\fill[color=black] (2.75,2) circle (0.05cm);
\draw[thick,red] (2.75,2) circle (4pt);
\fill[color=black] (3,0.25) circle (0.05cm);
\draw[thick,red] (3,0.25) circle (4pt);
\fill[color=black] (3.25,2.75) circle (0.05cm);
\draw[thick,red] (3.25,2.75) circle (4pt);
\fill[color=black] (3.5,1) circle (0.05cm);
\draw[thick,red] (3.5,1) circle (4pt);
\fill[color=black] (3.75,3.5) circle (0.05cm);
\draw[thick,red] (3.75,3.5) circle (4pt);
\fill[color=black] (4,1.75) circle (0.05cm);
\draw[thick,red] (4,1.75) circle (4pt);

\draw[thick,->] (0,0) -- (3,0.25)  node[midway, above left=-4pt]{$\scriptstyle{1}$};
\draw[thick,->] (3,0.25) -- (3.5,1)  node[midway, above left=-5pt]{$\scriptstyle{3}$};
\draw[thick,->] (3.5,1) -- (3.75,3.5)  node[midway, above left=-3pt]{$\scriptstyle{4}$};
\draw[thick,->] (3,0.25) -- (3.25,2.75)  node[midway, above left=-3pt]{$\scriptstyle{4}$};
\draw[thick,->] (3.5,1) -- (4,1.75)  node[midway, above left=-5pt]{$\scriptstyle{3}$};

\draw[thick,->] (0,0) -- (1.75,0.5)  node[midway, above left=-5pt]{$\scriptstyle{2}$};
\draw[thick,->] (1.75,0.5) -- (2,3)  node[midway, above left=-3pt]{$\scriptstyle{4}$};
\draw[thick,->] (1.75,0.5) -- (2.25,1.25)  node[midway, above left=-5pt]{$\scriptstyle{3}$};
\draw[thick,->] (2.25,1.25) -- (2.5,3.75)  node[midway, above left=-3pt]{$\scriptstyle{4}$};
\draw[thick,->] (2.25,1.25) -- (2.75,2)  node[midway, above left=-5pt]{$\scriptstyle{3}$};

\draw[thick,->] (0,0) -- (0.5,0.75)  node[midway, above left=-5pt]{$\scriptstyle{3}$};
\draw[thick,->] (0.5,0.75) -- (1,1.5)  node[midway, above left=-5pt]{$\scriptstyle{3}$};
\draw[thick,->] (0.5,0.75) -- (0.75,3.25)  node[midway, above left=-3pt]{$\scriptstyle{4}$};
\draw[thick,->] (1,1.5) -- (1.5,2.25)  node[midway, above left=-5pt]{$\scriptstyle{3}$};
\draw[thick,->] (1,1.5) -- (1.25,4)  node[midway, above left=-3pt]{$\scriptstyle{4}$};

\draw[thick,->] (0,0) -- (0.25,2.5)  node[midway, above left=-3pt]{$\scriptstyle{4}$};

\end{tikzpicture}
\end{center}
where we use $i$ in place of the generators $z_i$.

We note that it is straightforward to read off the monomial ideals from this diagram.
\begin{center}
\begin{tikzpicture}

\fill[color=black] (0,0) circle (0.05cm);
\draw[thick,red] (0,0) circle (4pt);
\fill[color=black] (0.25,2.5) circle (0.05cm);
\draw[thick,red] (0.25,2.5) circle (4pt);
\fill[color=black] (0.5,0.75) circle (0.05cm);
\draw[thick,red] (0.5,0.75) circle (4pt);
\fill[color=black] (0.75,3.25) circle (0.05cm);
\draw[thick,red] (0.75,3.25) circle (4pt);
\fill[color=black] (1,1.5) circle (0.05cm);
\draw[thick,red] (1,1.5) circle (4pt);
\fill[color=black] (1.25,4) circle (0.05cm);
\draw[thick,red] (1.25,4) circle (4pt);
\fill[color=black] (1.5,2.25) circle (0.05cm);
\draw[thick,red] (1.5,2.25) circle (4pt);
\fill[color=black] (1.75,0.5) circle (0.05cm);
\draw[thick,red] (1.75,0.5) circle (4pt);
\fill[color=black] (2,3) circle (0.05cm);
\draw[thick,red] (2,3) circle (4pt);
\fill[color=black] (2.25,1.25) circle (0.05cm);
\draw[thick,red] (2.25,1.25) circle (4pt);
\fill[color=black] (2.5,3.75) circle (0.05cm);
\draw[thick,red] (2.5,3.75) circle (4pt);
\fill[color=black] (2.75,2) circle (0.05cm);
\draw[thick,red] (2.75,2) circle (4pt);
\fill[color=black] (3,0.25) circle (0.05cm);
\draw[thick,red] (3,0.25) circle (4pt);
\fill[color=black] (3.25,2.75) circle (0.05cm);
\draw[thick,red] (3.25,2.75) circle (4pt);
\fill[color=black] (3.5,1) circle (0.05cm);
\draw[thick,red] (3.5,1) circle (4pt);
\fill[color=black] (3.75,3.5) circle (0.05cm);
\draw[thick,red] (3.75,3.5) circle (4pt);
\fill[color=black] (4,1.75) circle (0.05cm);
\draw[thick,red] (4,1.75) circle (4pt);

\draw[thick,->] (0,0) -- (3,0.25)  node[midway, above left=-4pt]{$\scriptstyle{1}$};
\draw[thick,->] (3,0.25) -- (3.5,1)  node[midway, above left=-5pt]{$\scriptstyle{3}$};
\draw[thick,->] (3.5,1) -- (3.75,3.5)  node[midway, above left=-3pt]{$\scriptstyle{4}$};
\draw[thick,->] (3,0.25) -- (3.25,2.75)  node[midway, above left=-3pt]{$\scriptstyle{4}$};
\draw[thick,->] (3.5,1) -- (4,1.75)  node[midway, above left=-5pt]{$\scriptstyle{3}$};

\draw[thick,->] (0,0) -- (1.75,0.5)  node[midway, above left=-5pt]{$\scriptstyle{2}$};
\draw[thick,->] (1.75,0.5) -- (2,3)  node[midway, above left=-3pt]{$\scriptstyle{4}$};
\draw[thick,->] (1.75,0.5) -- (2.25,1.25)  node[midway, above left=-5pt]{$\scriptstyle{3}$};
\draw[thick,->] (2.25,1.25) -- (2.5,3.75)  node[midway, above left=-3pt]{$\scriptstyle{4}$};
\draw[thick,->] (2.25,1.25) -- (2.75,2)  node[midway, above left=-5pt]{$\scriptstyle{3}$};

\draw[thick,->] (0,0) -- (0.5,0.75)  node[midway, above left=-5pt]{$\scriptstyle{3}$};
\draw[thick,->] (0.5,0.75) -- (1,1.5)  node[midway, above left=-5pt]{$\scriptstyle{3}$};
\draw[thick,->] (0.5,0.75) -- (0.75,3.25)  node[midway, above left=-3pt]{$\scriptstyle{4}$};
\draw[thick,->] (1,1.5) -- (1.5,2.25)  node[midway, above left=-5pt]{$\scriptstyle{3}$};
\draw[thick,->] (1,1.5) -- (1.25,4)  node[midway, above left=-3pt]{$\scriptstyle{2}$};

\draw[thick,->] (0,0) -- (0.25,2.5)  node[midway, above left=-3pt]{$\scriptstyle{4}$};

\end{tikzpicture}
,
 \begin{tikzpicture}
\fill[color=black] (3,0.25) circle (0.05cm);
\draw[thick,red] (3,0.25) circle (4pt);
\fill[color=black] (3.25,2.75) circle (0.05cm);
\draw[thick,red] (3.25,2.75) circle (4pt);
\fill[color=black] (3.5,1) circle (0.05cm);
\draw[thick,red] (3.5,1) circle (4pt);
\fill[color=black] (3.75,3.5) circle (0.05cm);
\draw[thick,red] (3.75,3.5) circle (4pt);
\fill[color=black] (4,1.75) circle (0.05cm);
\draw[thick,red] (4,1.75) circle (4pt);

\draw[thick,->] (3,0.25) -- (3.5,1)  node[midway, above left=-5pt]{$\scriptstyle{3}$};
\draw[thick,->] (3.5,1) -- (3.75,3.5)  node[midway, above left=-3pt]{$\scriptstyle{4}$};
\draw[thick,->] (3,0.25) -- (3.25,2.75)  node[midway, above left=-3pt]{$\scriptstyle{4}$};
\draw[thick,->] (3.5,1) -- (4,1.75)  node[midway, above left=-5pt]{$\scriptstyle{3}$};
\end{tikzpicture}
,
 \begin{tikzpicture}
\fill[color=black] (3.5,1) circle (0.05cm);
\draw[thick,red] (3.5,1) circle (4pt);
\fill[color=black] (3.75,3.5) circle (0.05cm);
\draw[thick,red] (3.75,3.5) circle (4pt);
\fill[color=black] (4,1.75) circle (0.05cm);
\draw[thick,red] (4,1.75) circle (4pt);

\draw[thick,->] (3.5,1) -- (3.75,3.5)  node[midway, above left=-3pt]{$\scriptstyle{4}$};
\draw[thick,->] (3.5,1) -- (4,1.75)  node[midway, above left=-5pt]{$\scriptstyle{3}$};
\end{tikzpicture}
,
 \begin{tikzpicture}
\fill[color=black] (3.5,1) circle (0.05cm);
\draw[thick,red] (3.5,1) circle (4pt);
\end{tikzpicture}
.
\end{center}
\end{Example}

\bibliographystyle{alpha}

\end{document}